\documentclass[a4paper,10pt,twoside]{extarticle}
\usepackage{graphicx}
\usepackage{caption}
\usepackage{subcaption}
\usepackage{multirow}
\usepackage{bbm}
\usepackage[utf8]{inputenc}
\usepackage[T1]{fontenc}
\usepackage{ dsfont }
\usepackage[english]{babel}

\usepackage{charter}

\usepackage{indentfirst}
\usepackage{xcolor}
\usepackage{verbatim}

\usepackage{hyperref}

\usepackage{pifont}
\newcommand{\cmark}{\textrm{\ding{52}}}%

\usepackage[top=3.cm, bottom=4.0cm, left=2.2cm, right=2.2cm]{geometry}
\usepackage{amsmath}
\usepackage{cases}
\usepackage{amsthm}	
\usepackage{amsfonts}	
\usepackage{amssymb	
            ,bbm
            ,units 
            ,stmaryrd
            ,mathrsfs
           }
\usepackage{enumerate}

\usepackage[square,numbers,sort&compress]{natbib}

\usepackage{
            ulem		
           ,soul		
} 

\numberwithin{equation}{section}
\numberwithin{figure}{section}
 \usepackage[nodayofweek]{datetime}

\renewcommand*{\thefootnote}{\fnsymbol{footnote}}

\theoremstyle{plain}
\newtheorem{theorem}{Theorem}[section]
\newtheorem{lemma}{Lemma}[section]

\newtheorem{corollary}{Corollary}[section]
\newtheorem{assumption}{Assumption}[section]
\newtheorem{remark}{Remark}[section]

\newcommand{\bN}{\mathbb{N}}

\newcommand{\bR}{\mathbb{R}}

\newcommand{\cP}{\mathcal{P}}

\definecolor{darkgreen}{rgb}{0,0.35,0}

 \allowdisplaybreaks
\title
{Tamed Euler approximation for fully superlinear growth McKean-Vlasov SDE and their particle systems: sharp rates for strong propagation of chaos, convergence and ergodicity}

\author{
\normalsize Simran Soni\textit{$^{a}$} \\
        \small   simran{\textunderscore}s@ma.iitr.ac.in
\and
\normalsize Neelima\textit{$^{b}$} \\
        \small   neelima{\textunderscore}maths@ramjas.du.ac.in 
\and 
\normalsize Chaman Kumar\textit{$^{a,}$}\footnote{CK acknowledges support from the London Mathematical Society under Scheme 5 (grant no. 52340) which partially funded his visit to the United Kindom.} 
\\
        \small   chaman.kumar@ma.iitr.ac.in 
\and
 \normalsize Gon\c calo dos Reis\textit{$^{c,d,}$}\footnote{G.d.R.~acknowledges support from the \emph{Funda{\c c}\~ao para a Ci\^{e}ncia e a Tecnologia} (Portuguese Foundation for Science and Technology) through the projects UIDB/00297/2020 (https://doi.org/10.54499/UIDB/00297/2020) and UIDP/00297/2020
(https://doi.org/10.54499/UIDP/00297/2020) (Center for Mathematics and Applications, NOVA Math) and by the UK Research and Innovation (UKRI) under the UK government’s Horizon Europe funding Guarantee [Project UKRI343].} 
\\
        \small  G.dosReis@ed.ac.uk
}

\date{%
    \footnotesize 
    $^{a}$~Department of Mathematics, Indian Institute of Technology Roorkee, Roorkee, 247 667, India
    \\
    $^{b}$~Department of Mathematics, Ramjas College, University of Delhi, Delhi, 110 007,  India
    \\
    $^{c}$~School of Mathematics, University of Edinburgh, JCMB, 
    Peter Guthrie Tait Road, Edinburgh, EH9 3FD, UK
    \\
    $^{d}$~Centro de Matem\'atica e Aplica\c c\~{o}es (Nova MATH), FCT, UNL, 2829-516 Caparica, Portugal
    \\
    \longdate \today \ (\currenttime)
}

\begin{document}

\selectlanguage{english}

\maketitle
\renewcommand*{\thefootnote}{\arabic{footnote}}
\begin{abstract}
    We study McKean--Vlasov Stochastic Differential Equations (MV-SDEs) whose drift and diffusion coefficients are of superlinear growth in \textit{all} their variables thus also superlinear in the measure component (the meaning is specified in the body of the paper). We address the finite and infinite time horizon case. 

    Our contribution is fourfold. 
    (a) We establish well-posedness for this class of equations and the corresponding interacting particle system. 
    (b) We prove two propagation of chaos results with explicit $L^2$-convergence rates: the first, is a general one where the rate degrades as the system's dimension $d$ increases; the second, attains the sharp rate $N^{-1/2}$ (in particle number $N$) uniformly over the dimension $d$ at the cost of a Vlasov kernel structure that is general and of superlinear growth for the measure dependency---the latter's proof fully avoids the Kantorovich-Rubinstein duality argument. 
    (c) Unlike existing works---based on semi-implicit schemes or truncated Euler schemes---we propose a fully explicit tamed Euler scheme that has reduced computational cost (comparatively). 
    The explicit scheme is shown to converge in strong $L^p$-sense with rate $1/2$ (in timestep).  
    (d) Lastly, we establish exponential ergodicity properties and long-time behavior for the MV-SDE, the corresponding interacting particle system, and the tamed scheme. The latter result is, to the best of our knowledge, fully novel. 
\end{abstract}


\textbf{MSC2020 subject classifications: } Primary: 60H35 , 65C30
Secondary: 65C05 , 65C35

\newpage 
\footnotesize
\tableofcontents
\normalsize
\section{Introduction}
\normalsize
McKean–Vlasov Stochastic Differential Equations (MV-SDEs), also referred to as mean-field or distribution-dependent SDEs, were first introduced by McKean in the 1960s as probabilistic counterparts to nonlinear parabolic PDEs \cite{mckean1966}. The key element is that such equations can be represented as limits of interacting particle systems, where each particle evolves under the influence of the empirical distribution of the system.
Through this particle dynamics to mean-field correspondence, MV-SDEs serve as a bridge between microscopic models of interacting agents and the macroscopic equations that describe their collective behavior. In a sense, MV-SDEs represent a type of dimensionality reduction of the large interacting population system and thus MV-SDEs appear nowadays in many applications in physics, biology, finance, machine learning and other fields \cite{Sznitman1991TopicsinPoC, kolokoltsov2010nonlinear, carmona2018a, carmona2018b,baladron2012, bolley2011, bossy2015, dreyer2011, guhlke2018,MR4776392,sabbar2025review,neufeld2025multilevel}.  
In the majority of cases explicit solutions of MV-SDEs are, as expected, out of reach. This makes numerical methods essential for both theoretical study and practical applications. 
A highly tractable methodology to numerically approximate an MV-SDE is via the particle system interpretation. Concretely, one writes an interacting $N$-particle system (IPS) that converges to the target MV-SDE (in particle number) and then one numerically approximates said IPS. 
From a theoretical point of view as the number of particles $N$ grows, the empirical law from the IPS converges to the law of the limiting MV-SDE in a phenomenon known as \emph{Propagation of Chaos} (PoC)  
\cite{Sznitman1991TopicsinPoC,lackerleflem2023sharpPoC,MR4489768,MR4489769}. The PoC result in combination with numerical discretization results form the recipe for this numerical approximation procedure. A detailed discussion on numerical methods for MV-SDEs that completely avoid the IPS approach is given in~\cite{agarwal2023numerical}.

We consider the following class of MV-SDEs 
\begin{equation}
\label{Mckean-Intro}
       dX_t = 
       \Big\{  \int_{\mathbb R^d} f(X_t,y)\mu_t(dy)+b(t,X_t,\mu_t)   \Big\}dt 
       + 
       \Big\{ \int_{\mathbb R^d} g(X_t,y)\mu_t(dy)+\sigma(t,X_t,\mu_t) \Big\} dW_t,
    \end{equation}
where $\mu_t$ denotes the law of $X_t$ for all $t\in [0, T]$, the initial value $X_0$ is a  sufficiently integrable $\mathbb R^d$-valued random variable that is independent of the $\bR^l$-valued Brownian motion $W$, and we take some measurable maps $f:\mathbb{R}^d \times \mathbb R^d\mapsto \mathbb{R}^d,$ $g:\mathbb{R}^d \times \mathbb R^d\mapsto \mathbb{R}^{d\times l},$ $b:[0,T] \times\mathbb{R}^d \times \mathcal{P}_2(\mathbb R^d) \mapsto  \mathbb R^d$ and $\sigma: [0,T] \times\mathbb{R}^d \times \mathcal{P}_2(\mathbb R^d) \mapsto \mathbb R^{d\times l}$ where $\mathcal{P}_2(\mathbb R^d)$ denotes the set of probability measures over $\mathbb R^d$ with finite second moment. 
The existence and uniqueness of the strong solution of MV-SDE \eqref{Mckean} with coefficients that meet Lipschitz conditions and grow linearly with respect to both measure and state variables are presently well established  \cite{Sznitman1991TopicsinPoC, carmona2018a, carmona2018b} including classes of coefficients more general than those of \eqref{Mckean-Intro}. 
Our focus is the superlinear growth class.

The importance of the general superlinear class, when $f,g,b,\sigma$ are of superlinear growth, is well articulated in \cite{chen2025} and it includes a review of use-cases. 
Nonetheless, in \cite{chen2025} the authors address only the convolution class that writes the terms involving $f,g$ in \eqref{Mckean-Intro} as $f(x,y)=K(x-y)$ for some function $K$ and thus $\int_{\mathbb R^d} f(X_t,y)\mu_t(dy)=(K* \mu_t)(X_t)$ where $K$ is of superlinear growth and $*$ stands for the convolution symbol. 
We call the general structure of the $f,g$ terms in \eqref{Mckean-Intro} as Vlasov kernels \cite{mckean1966} while those of the form $(K* \mu)(\cdot)$ are called convolution kernels. 
Although \cite{chen2024,chen2025} was a step forward in the sense that there was no literature addressing the time-discretisation of MV-SDEs with superlinear growth in the measure component, many superlinear situations of interest fall beyond the convolution archetype but are within the type of \eqref{Mckean-Intro}, e.g., \cite[\textsection 1.2.2]{bolley2011}, \cite[\textsection 3]{guillin2025some}, \cite{iguchi2025parameter}, \cite[Eq. (1.3)]{ringh2025kalman}. 
In \cite{guillin2025some}, the authors investigate how the random batch method \cite{MR4019084} affects the approximation of the invariant distribution of a McKean–Vlasov SDE of Langevin type. Their analysis leads to a modified McKean–Vlasov equation in which the variance of the solution process appears explicitly in the drift term. 
In \cite{ringh2025kalman}, the authors study the covariance preconditioned mean-field Langevin equation in the context of Consensus Based Optimization (CBO),  \cite{PiersHinds2025}, which involves analysing a class of MV-SDEs where the variance of the solution process appears in the diffusion coefficient; we note that their class of MV-SDEs mixes both convolution and Vlasov kernels (as in \eqref{Mckean-Intro} but beyond \cite{chen2024,chen2025}). 
Lastly, \cite[\textsection 1.2.2]{bolley2011} highlights several open-questions regarding extensions to nonlinear diffusion coefficients. 

\smallskip

The main aim of this paper is to close several literature gaps regarding MV-SDE when its coefficients are of general superlinear growth in their variables. We address wellposedness, particle approximation and propagation of chaos, then propose an explicit tamed Euler scheme, of reduced computational cost comparative to existing literature, to approximate \eqref{Mckean-Intro} showing in particular the feasibility of such schemes in finite and infinite time and with explicit convergence rates. The works closest to ours are \cite{chen2025,yuanping2024explicit,liu2025long}, and our contribution here is beyond their settings. In particular, we provide answers to questions left open in their contributions. 
\smallskip

\emph{Well-posedness and Propagation of Chaos.} 
Results on well-posedness of the MV-SDE with the superlinear drift coefficient in the state variable are nowadays well-known \cite{baladron2012,reis2019,kumar2021,reisinger2022,chaman2022,neelima2020}. In \cite{chaman2022}, wellposedness for MV-SDEs with common noise having superlinear growth in the state variable of all the coefficients is established. 
For MV-SDE driven by L\'evy noise, \cite{neelima2020} proved the existence and uniqueness of the strong solution under a new set of assumptions related to nonglobal Lipschitz conditions that allow the drift, diffusion and jump coefficients to have superlinear growth in the state variable. The common assumption in these works is that the coefficients are assumed Wasserstein-Lipschitz continuous in the measure variable. 

Considering convolution kernels in the drift coefficient, \cite{chen2024} proved the wellposedness of the MV-SDE when the kernel has superlinear growth and hence the authors effectively allow the drift coefficient to be fully superlinear in both state and measure. 
The result is extended in \cite{chen2025} for MV-SDE with fully superlinear coefficients in both the drift and diffusion coefficients in space and measure.
However, in both of these works, only convolution kernels (plus Wasserstein-Lipschitz dependencies) are addressed.

Alongside wellposedness for the MV-SDE, one deals also with the approximating interacting particle system with propagation of chaos (PoC) connecting the particle systems to its mean-field limits \cite{Sznitman1991TopicsinPoC,lackerleflem2023sharpPoC,MR4489768,MR4489769}. 
In the superlinear settings with non-constant diffusion and general Wasserstein-$2$ continuous measure functional, recent contributions establish the PoC with explicit convergence rates that typically encapsulate a dependence on the state-space dimension  \cite{reis2019,neelima2020,kumar2021,reisinger2022,chaman2022,li2023,yuanping2024explicit,jian2025modified}. 
Quantitative PoC rates in particle number $N$ often express a degradation in the ambient dimension $d$, concretely $N^{-1/d}$ (for $d>4$). This type of rates stem from a certain geometric phenomenon of the approximating empirical measure (under the Wasserstein-$2$ metric) where as the dimension grows, the ``holes'' in the empirical distribution (from uniform i.i.d.~sampling) get larger, and the optimal transport cost scales accordingly \cite{fournier2015}. 
Nonetheless, in a variety of settings, the PoC rate has been shown to be the sharp $N^{-1/2}$ without the dimension-decay, most recently in  \cite{zhang2025,baohao2025PoC-sharprate} under structural Lipschitz-type conditions on the (nonlinear) dependence on the measure---concretely Wasserstein-$1$ type functionals. See \cite{lackerleflem2023sharpPoC,bernou2025uniform} for, to our knowledge, best known results and \cite{MR4489768,MR4489769} for a wide review. 

The sharp $L^2$-strong PoC rate $N^{-1/2}$ has been attained in a variety of works studying numerical methods \cite{belomestny2018projected,awadelkarim2024multilevel,DellaMaestraHoffmann2022NonparametricEstimation,Hinds2025wellposedness,zhang2025} using arguments different from the usual ones. 
This is achieved in \cite{belomestny2018projected} using the classical Rosenthal inequality (see Lemma \ref{lemma:classicalRosenthalIneq}) in combination with strong uniformly Lipschitz conditions and scalar interaction kernels; in \cite[Lemma A.4]{awadelkarim2024multilevel} the Marcinkiewicz-Zygmund inequality is used instead of Rosenthal's inequality (these two inequalities are different but akin) alongside time-homogeneous coefficients, deterministic initial condition, and uniformly bounded and Lipschitz Vlasov kernels. 
More recently, and carried out independently, \cite{DellaMaestraHoffmann2022NonparametricEstimation,zhang2025,Hinds2025wellposedness} obtain the sharp $N^{-1/2}$ PoC rate by leveraging the Kantorovich-Rubinstein duality argument under the Wasserstein-$1$ metric; in \cite{DellaMaestraHoffmann2022NonparametricEstimation}, the diffusion coefficient is independent of the measure; in \cite{Hinds2025wellposedness} the authors rely explicitly on the fact that the MV-SDE under study is of reflected type on a bounded domain; and \cite{zhang2025} combines the Kantorovich duality argument with the conditional Rosenthal inequality (see Lemma \ref{lemma:conditionalRosenthalIneq}) but requires Lipschitz assumptions in the measure component.  
Establishing a viable mechanics for sharp dimension-free PoC rates for general superlinear and non-Lipschitz measure dependencies, remains an open question. 

\textit{Our contribution.} We establish strong well-posedness for \eqref{Mckean-Intro} with   coefficients that are fully superlinear in both the state and the measure variable. 
Our setting is more general than the convolution kernels class of \cite{chen2025} and extends previous work on superlinear MV-SDEs. Further, we do not assume any differentiability or ellipticity requirements. 
The methodology here departs significantly from \cite{chen2025}: we have identified a suitable Banach space where the fixed-point argument is applied over the interval $[0,T]$ unlike in \cite{chen2025} where the fixed-point argument is applied over sufficiently small intervals repeatedly. This essentially means that the moment bound of the solution process follows directly from the moment bound of the approximating MV-SDE. Thus, we do not require the `extra symmetry' assumptions of \cite{chen2025}. 

As a second contribution, we provide new quantitative PoC results. In the general setting with Wasserstein-$2$ functionals (in $b,\sigma$), we obtain dimension-dependent convergence rates which match the behavior observed in many contributions. 
More importantly, in the full Vlasov kernel case we obtain the sharp dimension-independent rate $N^{-1/2}$, complementing and advancing recent results that were previously confined to Lipschitz measure dependencies. 
Our results are broader than \cite{belomestny2018projected,zhang2025,PiersHinds2025} and avoid altogether Kantorovich duality arguments by showing that one needs only a certain conditioning argument, that is in no way tied to the Kantorovich duality, and some higher moment of the MV-SDE's solution. 
\smallskip

\emph{Numerical schemes for superlinear growth (non-Lipschitz in measure).}
In relation to numerical schemes for MV-SDEs with superlinear growth,  
several strategies are available with the recent \cite{jian2025modified} offering an overview of methods dealing with `particle corruption' \cite{dos2022,zhang2025,yuanping2024explicit}: (i) implicit scheme or backward Euler \cite{dos2022,liu2025long}
(ii) split-step (semi-implicit) integrators that separate the stiff components via an implicit step and control growth for fully superlinear MV-SDEs \cite{chen2022,chen2024,chen2025}; 
(iii) explicit time-adaptive or problem-aware discretizations \cite{reisinger2022,Tran2025MR4815916}; 
and (iv) explicit Euler schemes with truncation \cite{yuanping2024explicit} or projection \cite{liu2025long} or taming \cite{dos2022,chaman2022,Tran2025MR4815916,zhang2025}--- and all these schemes attain a strong $L^2$-rate of $1/2$ in the discretisation timestep. 
With the exception of \cite{chen2024,chen2025}, these contributions successfully address superlinear growth in space but assume a Wasserstein-Lipschitz behavior in the measure component. We point to Table \ref{table:literature} for an overview.

\begin{table}[htb]
\centering
\footnotesize
\setlength{\tabcolsep}{3pt}
\caption{Literature overview}
\begin{tabular}{|c|cccc|cccc|c|c|c|}
\hline
\multirow{3}{*}{Paper} &  \multicolumn{4}{c|}{Drift map}  & \multicolumn{4}{c|}{Diffusion maps}  &  \multirow{2}{*}{Kernel} & \multirow{3}{*}{Scheme type} & \multirow{3}{*}{Rate}  \\ \cline{2-9}

 & \multicolumn{2}{c|}{Space}  & \multicolumn{2}{c|}{Measure}  & \multicolumn{2}{c|}{Space}   & \multicolumn{2}{c|}{Measure}  & &  &  \\ \cline{2-9}
  
 & \multicolumn{1}{c|}{Lip.} & \multicolumn{1}{c|}{SuperLin.} & \multicolumn{1}{c|}{Lip} & SuperLin. & \multicolumn{1}{c|}{Lip.} & \multicolumn{1}{c|}{SuperLin.} &\multicolumn{1}{c|}{Lip} & \multicolumn{1}{c|}{SuperLin.} & type &  & \\ \hline
 \hline 

\cite{kumar2021} 
& \multicolumn{1}{c|}{\checkmark}   & \multicolumn{1}{c|}{\checkmark}    
& \multicolumn{1}{c|}{\checkmark}   & no 
& \multicolumn{1}{c|} {\checkmark}   & \multicolumn{1}{c|}{no}  
& \checkmark & \multicolumn{1}{|c|}{no} 
& --- &  Milstein (explicit)  & 1.0  \\ \hline


\multirow{2}{*}{\cite{dos2022}}   
& \multicolumn{1}{c|}{\multirow{2}{*}{\checkmark}}       & \multicolumn{1}{c|}{\multirow{2}{*}{\checkmark}}    
& \multicolumn{1}{c|}{\multirow{2}{*}{\checkmark}}       & \multirow{2}{*}{no}  
& \multicolumn{1}{c|}{\multirow{2}{*}{\checkmark}}       & \multicolumn{1}{c|}{\multirow{2}{*}{no}}  
& \multicolumn{1}{c|}{\multirow{2}{*}{\checkmark}}        & \multicolumn{1}{c|}{\multirow{2}{*}{no}} 
& \multirow{2}{*}{---} &  Tamed EM (explicit) & \multirow{2}{*}{0.5}    \\ 

&  \multicolumn{1}{c|}{}  &  \multicolumn{1}{c|}{}
&  \multicolumn{1}{c|}{}  & \multicolumn{1}{c|}{}
&  \multicolumn{1}{c|}{}  & \multicolumn{1}{c|}{}
& \multicolumn{1}{c|}{}   & \multicolumn{1}{c|}{}
&  & Backward EM (Implicit) &   \\ \hline


\multirow{2}{*}{\cite{reisinger2022}}   
& \multicolumn{1}{c|}{\multirow{2}{*}{\checkmark}}       & \multicolumn{1}{c|}{\multirow{2}{*}{\checkmark}}    
& \multicolumn{1}{c|}{\multirow{1}{*}{\checkmark}}       & \multirow{2}{*}{no}  
& \multicolumn{1}{c|}{\multirow{2}{*}{\checkmark}}       & \multicolumn{1}{c|}{\checkmark}  
& \multicolumn{1}{c|}{\checkmark}       & \multicolumn{1}{c|}{no} 
& \multirow{1}{*}{---} & Adaptive EM (explicit)  & 0.5    \\ \cline{4-4} \cline{7-11} 

&  \multicolumn{1}{c|}{}  &  \multicolumn{1}{c|}{}
&  \multicolumn{1}{c|}{\multirow{1}{*}{\checkmark $W^{(1)}$}}  & \multicolumn{1}{c|}{}
&  \multicolumn{1}{c|}{}  & \multicolumn{1}{c|}{no}
& \multicolumn{2}{c|}{no dependency}    
& \multirow{1}{*}{Vlasov} & Adaptive Milstein (explicit) &  1.0 \\ \hline


\cite{chen2022} 
& \multicolumn{1}{c|}{\checkmark}   & \multicolumn{1}{c|}{\checkmark}    
& \multicolumn{1}{c|}{\checkmark}   & no 
& \multicolumn{1}{c|} {\checkmark}   & \multicolumn{1}{c|}{no}  
& \checkmark & \multicolumn{1}{|c|}{no} 
& --- & Split-Step EM (semi-Implicit)  & 0.5   \\ \hline


\cite{li2023}${}^\star$ 
& \multicolumn{2}{c|}{Locally Lip.}       
& \multicolumn{1}{c|}{\checkmark}   & no 
& \multicolumn{2}{c|}{Locally Lip.}     
& \checkmark  & \multicolumn{1}{|c|}{no} 
& --- & EM (explicit) & 0.5    \\ \hline


\cite{chen2024}  
& \multicolumn{1}{c|}{\checkmark}   & \multicolumn{1}{c|}{\checkmark}    
& \multicolumn{1}{c|}{\checkmark}   & \checkmark 
& \multicolumn{1}{c|} {\checkmark}   & \multicolumn{1}{c|}{no}  
& \checkmark  & \multicolumn{1}{|c|}{no} 
& Convol.+$W^{(2)}$  & Split-Step EM (semi-Implicit) & 0.5    \\ \hline


$\dagger$\cite{yuanping2024explicit}${}^{\star}$ 
& \multicolumn{2}{c|}{Locally Lip.}       
& \multicolumn{1}{c|}{\checkmark}   & no 
& \multicolumn{2}{c|}{Locally Lip.} 
& \multicolumn{1}{c|}{\checkmark} & \multicolumn{1}{c|}{no} 
& --- & Truncated EM  (explicit) & 0.5    \\ \hline


\multirow{2}{*}{$\dagger$\cite{liu2025long}}   
& \multicolumn{1}{c|}{\multirow{2}{*}{\checkmark}}       & \multicolumn{1}{c|}{\multirow{2}{*}{\checkmark}}    
& \multicolumn{1}{c|}{\multirow{2}{*}{\checkmark}}       & \multirow{2}{*}{no}  
& \multicolumn{1}{c|}{\multirow{2}{*}{\checkmark}}       & \multicolumn{1}{c|}{\multirow{1}{*}{\checkmark}}  
& \multicolumn{1}{c|}{\multirow{2}{*}{\checkmark}}        & \multicolumn{1}{c|}{\multirow{2}{*}{no}} 
& \multirow{2}{*}{---} & Projected EM (explicit)  & \multirow{2}{*}{0.5}    \\ 
\cline{7-7}\cline{11-11} 

&  \multicolumn{1}{c|}{}  &  \multicolumn{1}{c|}{}
&  \multicolumn{1}{c|}{}  & \multicolumn{1}{c|}{}
&  \multicolumn{1}{c|}{}  & \multicolumn{1}{c|}{no}
& \multicolumn{1}{c|}{}   & \multicolumn{1}{c|}{}
&  & Backward EM (Implicit) &   \\ \hline


$\dagger$\cite{Tran2025MR4815916}${}^\star$   
& \multicolumn{1}{c|}{\checkmark}   & \multicolumn{1}{c|}{\checkmark}    
& \multicolumn{1}{c|}{\checkmark}   &  no
& \multicolumn{1}{c|}{\checkmark}   & \multicolumn{1}{c|}{\checkmark} 
& \multicolumn{1}{c|}{\checkmark}  & no
& --- & Tamed-adaptive EM scheme & 0.5  \\ \hline


\cite{chen2025}
& \multicolumn{1}{c|}{\checkmark}   & \multicolumn{1}{c|}{\checkmark}    
& \multicolumn{1}{c|}{\checkmark}   & \checkmark 
& \multicolumn{1}{c|} {\checkmark}   & \multicolumn{1}{c|}{\checkmark}  
& \checkmark  & \multicolumn{1}{|c|}{\checkmark} 
& Convol.+$W^{(2)}$  & Split-Step EM (semi-Implicit) & 0.5    \\ \hline


\hline
This work   & \multicolumn{1}{c|}{\cmark}  & \multicolumn{1}{c|}{\cmark}  & \multicolumn{1}{c|}{\cmark}  &  \cmark  & \multicolumn{1}{c|}{\cmark}  & \multicolumn{1}{c|}{\cmark}  &  \cmark & \multicolumn{1}{|c|}{\cmark} & Vlasov+$W^{(2)}$  &  Tamed EM(explicit)  & 0.5 \\ \hline \hline
\end{tabular}
\\
\raggedright{
`Lip.' = `Lipschitz' referring to satisfying a Lipschitz condition or also covering the Lipschitz condition; 
`SuperLin.' = `Superlinear growth' referring to satisfying a condition (e.g., polynomial growth, one-sided Lipschitz or Khasminskii-type) allowing for more than linear growth (see assumptions of Section \ref{sec:main:welpos}); 
Otherwise mentioned, `Lipschitz in measure' means Wasserstein-2 Lipschitz (see Eq. \eqref{eq:wasserstein-metric}); 
All Milstein methods require additional differentiability assumptions; Rate denotes strong $L^2$-rate; 
`EM' = Euler-Maruyama; `Convol.'=`Convolution kernel' type measure dependency. \\
}
\raggedright{
\textit{Additional notes}:
\\
\cite{li2023}${}^\star$ works with locally Lipschitz coefficients in the state variable, but under uniform linear growth assumptions; 
\\
\cite{yuanping2024explicit}${}^\star$ works with locally Lipschitz in the state variable and a Khasminskii-type growth condition, but the scheme's rate is established only under stronger growth conditions.
\\
\cite{Tran2025MR4815916}${}^\star$ include jumps and besides taming also a time-adaptive grid is used (as in \cite{reisinger2022}); 
\\
$\dagger$\cite{yuanping2024explicit,liu2025long}, \cite{Tran2025MR4815916} are the only works here that investigate as well the scheme's longtime behavior with \cite{yuanping2024explicit,chen2025} studying also ergodicity. All other works study the finite time case $T<\infty$.
}
\label{table:literature}
\end{table}  
\normalsize

Numerical experiments reported in the split-step literature, \cite[Remark 3.1]{chen2024}, suggest that ``taming'' can be ill-suited in convolution models with non-constant diffusion, unless strong dissipativity is present. This left open the question regarding the viability of taming schemes (also posed by \cite{yuanping2024explicit,jian2025modified}) under fully superlinear growth (in state and measure) in finite or infinite time.

\emph{Our contribution:} We propose an explicit tamed Euler scheme for MV-SDEs with fully superlinear growth coefficients (in both state and measure) that attains a strong convergence of order $1/2$ in $L^{p}$---no smoothness, ellipticity or constant diffusion coefficient is required. This closes the open questions left by \cite{chen2025,yuanping2024explicit,jian2025modified}.
\smallskip

\emph{Ergodicity and ergodic numerics (with taming).}
On the continuous-time side, there is a rich literature regarding ergodic properties for MV-SDEs and we mention only \cite{Wang2018-DDSDE-Landau-type} as this introduction is far too short to do justice to the existing body of work on the topic. 
In the SDE context, there is a parallel line of work proving that stabilized explicit or implicit schemes can approximate invariant distributions---but it is noteworthy to observe that only recently have contributions appeared regarding ergodicity of explicit tamed Euler schemes under one-sided Lipschitz and polynomial growth \cite{Brehier2023MR4655541,bao2024geometric,liu2024geometric,ottobrecrisanangeli2025}. 
In the MV-SDE context, ergodicity results and/or longtime behavior for schemes of superlinear growth MV-SDE under taming are largely missing from the literature with \cite{yuanping2024explicit,liu2025long} being the exception. 

In \cite{yuanping2024explicit}, the authors study a fully explicit Euler scheme of truncated/projected type for McKean–Vlasov SDEs whose drift exhibits superlinear growth in space and is Wasserstein–Lipschitz in the measure variable. Although the diffusion coefficient may grow polynomially in both the state and measure arguments, both coefficients remain Wasserstein–Lipschitz in measure. The analysis establishes convergence of the scheme over both finite and infinite time horizons.

\emph{Our contribution.} We establish exponential ergodicity not only for the MV-SDE but also for its interacting particle approximation and for the proposed explicit tamed Euler scheme. Regarding the latter, this is the first result demonstrating long-time stability and ergodicity for an explicit numerical scheme in the fully superlinear MV-SDE setting (and of the taming type in particular; see \cite{yuanping2024explicit,jian2025modified}). Our proof builds on ideas from \cite{bao2024geometric} and \cite{chen2025} where in the former (their section 2.2) the proof was carried out in Wasserstein-$1$ distance whilst here we work with Wasserstein-$2$ distance and with general coefficients.

\medskip
This paper is organized as follows. Section \ref{sec:main:welpos} describes the class of MV-SDEs we work with and contains the well-posedness results. Section \ref{sec:IPS-and-PoC} defines the associated IPS and establish the two PoC results of the paper. The taming scheme, properties and its convergence is established in Section \ref{sec:Taming-in Euler scheme}. The final part, Section \ref{sec:Ergodicity}  contains the several ergodicity result and those results for the tamed scheme appear in Section \ref{sec:ergodicity tamed Euler scheme}. Appendix \ref{appendix} collects several useful auxiliary results.

\subsection*{Notations}
Let $\mathbb N$ be the set of natural numbers and $\mathbb N_0:=\mathbb N \cup \{0\}$. Let  $\langle \cdot , \cdot \rangle $ denotes the inner product in $d$-dimensional Euclidean space of real numbers $\mathbb R^d$; both the Euclidean norm in $\bR^d$ and the Frobenius norm in $\bR^{d \times m}$ is denoted by $|\cdot|$.
The transpose of a matrix $A$ is denoted by $A^*$. 
Also, $\delta_x$ denotes the Dirac measure at point $x$. 
Further, $\mathcal B(\mathbb R^d)$  represents the Borel $\sigma$-field on $\mathbb R^d$ with $\mathcal{P} (\mathbb R^d)$ denoting the space of probability measures on $(\mathbb R^d,\mathbb B(\mathbb R^d))$. For $r\geq 1$, let $\mathcal{P}_r (\mathbb R^d)$ denotes the subspace of $\mathcal{P} (\mathbb R^d)$ having finite $r$-th moment, i.e., $\displaystyle \int_{\mathbb R^d} |x|^r \mu(dx) < \infty$. 
For $\mu, \nu \in \mathcal{P}_r (\mathbb{R}^d)$, the $r$-Wasserstein metric is defined as
\begin{equation}
\label{eq:wasserstein-metric}
    W^{(r)}(\mu,\nu) := \underset{\pi \in \Pi(\mu, \nu)}{\inf} \Bigl( \int_{\mathbb R^d} \int_{\mathbb R^d} |x-y|^r \pi(dx,dy) \Bigl)^{\frac{1}{r}}
\end{equation}
where $\Pi(\mu,\nu)$ is the set of all couplings of $\mu,\nu \in \mathcal{P}_r(\mathbb{R}^d)$. 
For $a,b\in \bR$,  we define the minimum of $a$ and $b$ as $a \wedge b$. 
We use the symbol $K>0$ to denote a generic constant that can vary from appearance to appearance. 
The floor function is denoted by $\lfloor \cdot \rfloor$. 
The positive part of a real number $a$ is denoted by $a^{+}$.  $\mathbbm{1}_A$ is indicator function of set $A$. 

\section{Well-posedness of MV-SDEs: full superlinear growth under the Vlasov structure}
\label{sec:main:welpos}
In this section, we state the class and framework for the equations we work with alongside well-posedness and moment estimates. 

Let $(\Omega,\{\mathscr{F}_t\}_{t\geq 0 },\mathscr{F}, \mathbb P )$ be a filtered probability space that satisfies the usual conditions, and let $W:=\{W_t\}_{t\geq 0}$ be a $l$-dimensional Brownian motion defined over it. 
For a fixed constant $T>0$,
assume that $f:\mathbb{R}^d \times \mathbb R^d\mapsto \mathbb{R}^d,$ $g:\mathbb{R}^d \times \mathbb R^d\mapsto \mathbb{R}^{d\times l},$ $b:[0,T] \times\mathbb{R}^d \times \mathcal{P}_2(\mathbb R^d) \mapsto  \mathbb R^d$ and $\sigma: [0,T] \times\mathbb{R}^d \times \mathcal{P}_2(\mathbb R^d) \mapsto \mathbb R^{d\times l}$ are Borel measurable functions. 
Define the following maps,  
\begin{align}
        \nu(t,x,\mu) & := \int_{\mathbb R^d} f(x,y)\mu(dy)+b(t,x,\mu) \quad \mbox{ and } \quad 
     \bar{\sigma}(t,x,\mu) 
     := \int_{\mathbb R^d} g(x,y)\mu(dy)+\sigma(t,x,\mu)\label{eq:v:sig}
    \end{align}
    for all $t \in [0, T]$, $x\in\mathbb R^d$ and $\mu \in \mathcal{P}_2(\mathbb R^d)$. 
In this article, we consider the following McKean--Vlasov Stochastic Differential Equation (MV-SDE), 
\begin{equation}
\label{Mckean}
       X_t = X_0 + \int_0 ^t \nu(s,X_s,\mu_s ^X )ds + \int_0 ^t  \bar{\sigma}(s,X_s,\mu_s ^X)dW_s
    \end{equation}
almost surely,  where $\mu_t ^X$ denotes the law of $X_t$ for all $t\in [0, T]$, and the initial value $X_0$ is an $\mathscr{F}_0$-measurable $\mathbb R^d$-valued random variable and is independent of $W$.


Notice that the coefficients  $\nu$ and  $\bar{\sigma}$ of MV-SDE \eqref{Mckean} are represented in an additive form of two functions where the first term provides a specific measure dependence through Vlasov kernels $f$ and $g$ while  their general measure dependence are expressed via the second term. 
Under this setup, the superlinearity in the measure is due to the superlinearity of the kernels $f$ and $g$ and the superlinearity in the state variable is due to the superlinearity of $\nu$ and $\bar{\sigma}$ in the state variable. 
Thus, MV-SDE \eqref{Mckean} is allowed to have fully superlinear structure.

\smallskip

For our framework, there are two leading fixed parameters $q,p_0$---let $q > 0$ and $p_0 > 2(q+1)$---where $q$ represents the highest polynomial degree domination on the map $f$ (and later on $b$) and $p_0$ is the integrability of the initial condition. 
In order to settle well-posedness and moments for MV-SDE \eqref{Mckean} we make the following assumptions.
\begin{assumption}{\label{assum_initial}}
   $\mathbb  E|X_0|^{p_0} < \infty$. 
\end{assumption}
\begin{assumption} \label{as:eu:b:sig}
    There exists a constant $L > 0$ such that
    \begin{align*}
         \langle x , b(t,x,\mu) \rangle + (p_0 -1)|\sigma(t,x,\mu)|^2 & \leq L\{1+ |x|^2 + W^{(2)} (\mu, \delta_0 )^2 \},
       \\
       \langle x -x' , b(t,x,\mu)- b(t,x',\mu') \rangle + |\sigma(t,x,\mu)- \sigma(t,x',\mu')|^2 & \leq L\{|x-x'|^2 + W^{(2)} (\mu, \mu' )^2 \},
    \end{align*}
    for all $t\in [0,T]$, $x,x' \in \mathbb R^d$ and $ \mu,\mu' \in \mathcal{P}_2(\mathbb R^d)$. 
    The maps $b,\sigma$ are assumed jointly continuous in their variables.
\end{assumption}
\begin{assumption}
\label{as:eu:f:g}
There exist  constants $L > 0$ and $q>0$ such that
   \begin{align*}
      \langle (x-y)-(x'-y')  , f(x,y) - f(x',y') \rangle +  &  (p_0 -1)|g(x , y)-g(x', y')|^2 
       \leq L|(x-y)-(x'-y')|^2 ,
        \\
        |f(x  , y)-f(x'  , y')| & \leq L \{1 + |x-y|+|x'-y'|\}^{q} \big|(x  - y)-(x'-y')\big|,
   \end{align*}
    for all  $x,x',y, y' \in \mathbb R^d$. 
   \end{assumption}
One can very roughly interpret the first condition of Assumption \ref{as:eu:f:g} when $L<0$ as the balance of the dynamics encapsulating an ``attraction'' between particles that increases with distance; this is a behavior different from a plasma. We refer to \cite{Veretennikov2006Ergodic} for further details.
  
   \begin{remark} \label{rem:fg:grw:eu}
       Due to Assumption \ref{as:eu:f:g}, there is a constant $K>0$ such that 
       \begin{align*}
        2 \langle x-y  , f(x,y) \rangle +  &  (p_0 -1)|g(x , y)|^2 
       \leq K(1+|x-y|^2),
       \\
       |g(x,y)-g(x',y')|^2 & \leq K \{1 + |x-y|+|x'-y'|\}^{q} \big|(x-y)  -(x'-y')\big|^2,  
       \\
           |f(x,y)| & \leq \, K (1 + |x-y|)^{q+1}, 
           \\
           |g(x,y)|^2  & \leq \, K (1 + |x-y|)^{q+2}, 
       \end{align*} 
        for any $x,x', y, y' \in \mathbb R^d$.  
   \end{remark}
\begin{remark}
It can be seen that the Vlasov kernels $f$ and $g$ are assumed to be independent of time. 
    However, one can extend all the results of this article, albeit with slight modifications, when they additionally depend on time (in a continuous fashion). 
\end{remark}
Our first result establishes well-posedness and moment estimates for \eqref{Mckean}, and the proof methodology is loosely inspired by  \cite{chen2025, chaman2022}. 
Regarding Assumption \ref{as:eu:b:sig}, we comment that a polynomial growth assumption in $b$ like the one stated for $f$ in Assumption \ref{as:eu:f:g}, the second condition, is not needed to establish well-posedness. That assumption in $f$ is needed to establish that a certain auxiliary contraction functional (the map $\Gamma$) is suitably well-defined (details below). 
\begin{theorem}[Existence and Uniqueness]
\label{E-U}
Take $q> 0$ and $p_0~> 2(q+1)$. 
Let Assumptions \ref{assum_initial}, \ref{as:eu:b:sig} and \ref{as:eu:f:g} hold. 
Then, there exists a unique strong solution to MV-SDE \eqref{Mckean} and the following estimate holds 
    \begin{align*}
       \sup_{ t \in [0,T]} \mathbb E  |X_t|^{p_0}  \leq K(1+ \mathbb E|X_0|^{p_0}) 
    \end{align*}
where $K>0$ is a constant that depends on $T$, $d$, $l$ and $L$. 
    \end{theorem}
We emphasize that the constant $K$ does depend on the dimension parameter $d$ and such is one of the reasons why later on we need to prove moment bounds for IPS \eqref{IPS} in Theorem \ref{thm:eu-ips}.
\begin{proof} 
Recall $p_0> 2(q + 1)$ and $q> 0$   from Assumptions \ref{assum_initial}  and \ref{as:eu:f:g}, respectively. 
Let us define the space $\mathcal{M}_{q}$  of continuous functions $\psi= (\psi_1, \psi_2) :  \mathbb R^d \mapsto \mathbb R^d \times \mathbb R^{d \times l}$ satisfying the conditions, 
\begin{align}
     \sup_{ x \in \mathbb R^d} \frac{|\psi (x)|}{1+ |x|^{q+1}} & < \infty, \notag
     \\
   \mbox{ and }    \langle x  -x', \psi_1(x) - \psi_1(x') \rangle & + (p_0-1)
   |\psi_2(x)-\psi_2(x')|^2   \leq L  |x-x'|^2  \label{assum_wellposedness}
\end{align}
 for a constant $L>0$ for all $x, x' \in \mathbb R^d$. From \eqref{assum_wellposedness} one derives that 
\begin{align*}
    2\langle x,  \psi_1(x) \rangle + (p_0 -1) |\psi_2(x)|^2 \leq K \{1 + |x|^2 \} 
\end{align*}
where the positive constant $K$ depends on $L$, $p_0$, $\psi_1(0)$ and $\psi_2(0)$. 
 It can be seen that the space $\mathcal{M}_{q}$ is a Banach space under the norm 
\[
 |\psi|_{q} : = \sup_{ x \in \mathbb R^d} \frac{|\psi (x)|}{1+ |x|^{q+1}} \leq  \sup_{ x \in \mathbb R^d} \frac{|\psi_1 (x)|+|\psi_2 (x)|}{1+ |x|^{q+1}}< \infty, 
\]
and hence the path space $C([0,T],\mathcal{M}_{q})$ is a Banach space under the norm 
\[
 |\varphi|_{[0,T],q} : = \sup_{ t \in [0,T]} |\varphi (t)|_q=\sup_{ t \in [0,T]} \sup_{ x \in \mathbb R^d} \frac{|\varphi (t,x)|}{1+ |x|^{q+1}} \leq \sup_{ t \in [0,T]} \sup_{ x \in \mathbb R^d} \frac{|\varphi_1 (t,x)|+|\varphi_2 (t,x)|}{1+ |x|^{q+1}}< \infty.
\]
We aim to apply a fixed-point argument on this space. For this, we define the functional $\Gamma$ as 
 \[
 \Gamma : C([0,T],\mathcal{M}_{q}) \mapsto C([0,T],\mathcal{M}_{q})   
 \] 
 that maps $\varphi \mapsto \Gamma [\varphi]$ with 
\begin{align*}
    \Gamma [\varphi] (t,x) 
    &
    =
    \Big(\Gamma[\varphi]_1(t,x), \Gamma[\varphi]_2(t,x)\Big) 
    :=
    \Bigl ( \int_{\mathbb R^d} f(x,y)\mu_t ^\varphi (dy) , \int_{\mathbb R^d} g(x,y)\mu_t ^\varphi (dy)  \Bigl )
    =
    \Big(\mathbb E f(x,X_t^\varphi), \mathbb E g(x,X_t^\varphi)\Big)
\end{align*} 
for all $t \in [0,T]$, $ x \in \mathbb R^d$ where $\{\mu_t^\varphi\}_{t \in [0,T]}$ is the flow of marginal laws of the solution $X^\varphi$ of the following MV-SDE
\begin{align*}
    dX_t ^\varphi  = \nu ^\varphi (t, X_t ^\varphi, \mu_t ^{\varphi}) dt + \bar{\sigma}^\varphi (t, X_t ^\varphi, \mu_t ^{\varphi}) dW_t,   \quad X_0 ^\varphi = X_0
\end{align*}
almost surely for $t \in [0,T]$. 
In the above, $\nu ^\varphi$ and $\bar{\sigma}^\varphi$ are given for all $t \in [0,T]$, $ x \in \mathbb R^d$ and $\mu \in \mathcal{P}_2(\mathbb R^d)$ by 
\begin{align*}
        \nu ^\varphi (t, x, \mu) = \varphi_1 (t,x) + b(t, x, \mu)
        \quad\textrm{and}\quad
        \bar{\sigma}^\varphi (t, x, \mu) =  \varphi_2(t,x) + \sigma (t, x, \mu).
\end{align*}
 It is known from \cite[Theorem 2.1]{chaman2022} that the above MV-SDE has a unique solution and the flow $\{\mu_t^\varphi\}_{t \in [0,T]}$ of marginals is continuous. 
Furthermore, $\displaystyle \sup_{t \in [0,T]} \mathbb  E|X_t^{\varphi}|^{p_0} \leq K(1+\mathbb  E |X_0|^{p_0})e^{KT}$ where the constant $K>0$ depends on $T$, $d$, $l$ and $L$.
Note that if $\varphi$ is a fixed point of $\Gamma$, then $X^\varphi$ will be a solution of MV-SDE (\ref{Mckean}). 
Observe that if $\varphi \in C([0,T],\mathcal{M}_{q})$ and Assumption \ref{as:eu:f:g} holds, then for all $t \in [0,T]$ and $ x,x' \in \mathbb R^d$ we have 
\begin{align*}
    \langle x  -x' , \Gamma[\varphi]_1(t,x) 
    &
    - \Gamma[\varphi]_1(t,x') \rangle + (p_0 -1)|\Gamma[\varphi]_2(t,x )-\Gamma[\varphi]_2(t,x' )|^2 
    \\
     & \leq \mathbb E \big\{ \langle x  -x' , f(x,X_t^\varphi) -  f(x',X_t^\varphi) \rangle + (p_0 -1) |g(x,X_t^\varphi) -  g(x',X_t^\varphi)|^2 \big\}  
     \\
     &\leq L |x  -x'|^2.
\end{align*}
Now using Remark \ref{rem:fg:grw:eu} and  \cite[Theorem 2.1]{chaman2022} we find a norm estimate as follows 
\begin{align*}
    |\Gamma [\varphi]|_{[0,T],q}  
    \leq 
    \sup_{ t \in [0,T]} \sup_{x \in \mathbb R^d} \frac{|\mathbb E f(x,X_t^\varphi)|+|\mathbb E g(x,X_t^\varphi)|}{1+|x|^{q+1}} 
    & 
    \leq K \sup_{ t \in [0,T]} \sup_{ x \in \mathbb R^d} \frac{\displaystyle \mathbb E (1+|x-X_t^\varphi|)^{q+1}}{1+ |x|^{q+1}}
    \\
   &  
    \leq K \bigl (1+\sup_{t \in [0,T]} \mathbb E |X_t ^\varphi|^{q+1} \bigl) <\infty.
\end{align*}
 Therefore, $\Gamma$ is a well-defined map.
  Furthermore, for $\varphi_1 = (\varphi_{1,1},\varphi_{1,2}), \varphi_2 = (\varphi_{2,1},\varphi_{2,2}) \in C([0,T],\mathcal{M}_{q})$, one uses Assumption \ref{as:eu:f:g} and Remark \ref{rem:fg:grw:eu} to obtain the following estimates.  
    \begin{align}
        |\Gamma [\varphi_1]-\Gamma [\varphi_2]|_{[0,T],q} & \leq K \sup_{ t \in [0,T]} \sup_{ x \in \mathbb R^d} \frac{ \mathbb E [|f(x,X_t ^{\varphi_1})-f(x,X_t ^{\varphi_2})|+ |g(x,X_t ^{\varphi_1})-g(x,X_t ^{\varphi_2})| ]}{1+ |x|^{q+1}}  \nonumber
        \\
        &  \leq K \sup_{ t \in [0,T]} \sup_{ x \in \mathbb R^d} \frac{ \mathbb E [|X_t ^{\varphi_1}-X_t ^{\varphi_2}|(1+|x|^{q})(1 + |X_t ^{\varphi_1}| + |X_t ^{\varphi_2}|)^{q} ]}{1+ |x|^{q+1}} \nonumber
        \\
        & \leq  K \Bigl (\sup_{t\in [0,T]} \mathbb E |X_t ^{\varphi_1}-X_t ^{\varphi_2}|^2  \sup_{t\in [0,T]} \mathbb E (1+|X_t ^{\varphi_1}|^{2q}+ |X_t ^{\varphi_2}|^{2q}) \Bigl )^{1/2} \nonumber
    \end{align} 
    which, on using \cite[Theorem 2.1]{chaman2022} yields, 
    \begin{align}\label{con}
        |\Gamma [\varphi_1]-\Gamma [\varphi_2]|_{[0,T],q}^2 \leq K \sup_{t\in [0,T]} \mathbb E |X_t ^{\varphi_1}-X_t ^{\varphi_2}|^2.
    \end{align}
Using It\^o'{s} formula,  
\begin{align*}
    \mathbb E |X_t ^{\varphi_1}-X_t ^{\varphi_2}|^2  
    &
    \leq  2 \mathbb E \int_0^t \{ \langle X_s ^{\varphi_1}-X_s ^{\varphi_2},b(s,X_s ^{\varphi_1},\mu_s ^ {\varphi_1})-b(s,X_s ^{\varphi_2},\mu_s ^ {\varphi_2}) \rangle + |\sigma(s,X_s ^{\varphi_1},\mu_s ^ {\varphi_1})-\sigma(s,X_s ^{\varphi_2},\mu_s ^ {\varphi_2}|^2 \} ds 
    \\
    & \quad + 2 \mathbb E \int_0^t \{ \langle X_s ^{\varphi_1}-X_s ^{\varphi_2}, \varphi_{1,1}(s,X_s ^{\varphi_1})-\varphi_{2,1}(s,X_s ^{\varphi_2}) \rangle + |\varphi_{1,2}(s,X_s ^{\varphi_1})-\varphi_{2,2}(s,X_s ^{\varphi_2})|^2 \} ds
\end{align*}
which due to Assumption \ref{as:eu:b:sig} and Equation \eqref{assum_wellposedness} yields   
\begin{align*}
   \mathbb E  |X_t ^{\varphi_1} -X_t ^{\varphi_2}|^2   \leq & \,  K \mathbb E \int_0^t  |X_s ^{\varphi_1}-X_s ^{\varphi_2}|^2 ds 
   \\
    &  + 2 \mathbb E \int_0^t \{ \langle X_s ^{\varphi_1}-X_s ^{\varphi_2}, \varphi_{1,1}(s,X_s ^{\varphi_1})-\varphi_{1,1}(s,X_s ^{\varphi_2}) \rangle + 2|\varphi_{1,2}(s,X_s ^{\varphi_1})-\varphi_{1,2}(s,X_s ^{\varphi_2})|^2 \} ds 
    \\
     &  + 2 \mathbb E \int_0^t \{ \langle X_s ^{\varphi_1}-X_s ^{\varphi_2}, \varphi_{1,1}(s,X_s ^{\varphi_2})- \varphi_{2,1}(s,X_s ^{\varphi_2}) \rangle + 4 \mathbb E \int_0^t |\varphi_{1,2}(s,X_s ^{\varphi_2})-\varphi_{2,2} (s,X_s ^{\varphi_2})|^2 ds
     \\
      \leq &\, K \mathbb E \int_0^t  |X_s ^{\varphi_1}-X_s ^{\varphi_2}|^2 ds 
     \\
     &  + 4 \mathbb E \int_0 ^t \{|\varphi_{1,1}(s,X_s ^{\varphi_2})-\varphi_{2,1}(s,X_s ^{\varphi_2})|^2 + |\varphi_{1,2}(s,X_s ^{\varphi_2})-\varphi_{2,2} (s,X_s ^{\varphi_2})|^2 \}ds
     \\
      \leq & \, K \mathbb E \int_0^t  |X_s ^{\varphi_1}-X_s ^{\varphi_2}|^2 ds  + K   \int_0^t |\varphi_1 - \varphi_2|^2_{[0,s],q} (1+\mathbb E|X_s ^{\varphi_2}|^{2(q+1)} ) ds<\infty 
\end{align*}
 for all $t \in [0,T]$. 
By Gronwall's lemma, one gets
\begin{align*}
   \sup_{t \in [0,T]}  \mathbb E|X_t ^{\varphi_1}-X_t ^{\varphi_2}|^2 \leq   K  \int_0^T |\varphi_1 - \varphi_2|^2_{[0,s],q} ds
\end{align*}
and hence from \eqref{con}, 
\begin{align*}
\big|\, \Gamma [\varphi_1]-\Gamma [\varphi_2]\,\big|^2_{[0,T],q} 
    & \leq K \int_0^T |\varphi_1 - \varphi_2|^2_{[0,t_1],q} dt_1
\end{align*}
 which on  further iteration of the above inequality yields the well-known simplex estimation 
 \begin{align*}
     \big|\,\Gamma^j [\varphi_1]-\Gamma^j [\varphi_2] \,\big|^2_{[0,T],q} 
     & 
     \leq K^j \int_0^T \int_0^{t_1} \cdots \int_0^{t_{j-1}} |\varphi_1 - \varphi_2|^2_{[0,t_j],q} dt_j \ldots dt_1  \leq \frac{(KT)^j}{j!} |\varphi_1 - \varphi_2|^2_{[0,t_j],q} 
     \\
      &
      \leq \frac{(KT)^j}{j!} |\varphi_1 - \varphi_2|^2_{[0,T],q}
 \end{align*}
for all $j \in \mathbb N$ where $\displaystyle \sum_{j=1}^\infty \frac{(KT)^j}{j!} =e^{KT} <\infty$. 
Banach fixed-point theorem deliver the desired result and the proof concludes.
\end{proof}

\section{The particle systems and sharp rates for Propagation of Chaos}
\label{sec:IPS-and-PoC}
Consider $N$ independent and identically distributed copies $\{X_0^i\}_{i \in \{1, \ldots, N\}}$  and $\{W^i\}_{i\in \{1,\ldots, N\}}$ of the initial value $X_0$ and the Brownian motion $W$, respectively.
Then, the \textit{non-interacting particles system (nIPS)} connected with MV-SDE  \eqref{Mckean} is given~by
\begin{align}{\label{NIPS}}
    X_t ^i = X_0 ^i + \int_0^t \nu(s,X_s^{i},\mu_s^{X})ds + \int_0^t \bar{\sigma}(s,X_s^{i},\mu_s ^{X})dW^i_s
\end{align}
almost surely for all $t \in [0,T]$ and $i \in \{1,\ldots,N \}$ 
where $\mu_t ^{X}$ is the common law of the particles $\{X_t^i\}_{i \in \{1, \ldots, N\}}$, i.e.,  $\mu_t^{X^i} = \mu_t^X$  for all   $i \in \{1,\ldots,N \}$ (see \cite[Proposition 2.11]{carmona2018b}).

The $(\bR^d)^N$-valued \textit{interacting particle system (IPS)} associated with MV-SDE \eqref{Mckean} is given~by  
    \begin{equation}{\label{IPS}}
         X_t ^{i,N} = X_0 ^{i} + \int_0 ^t\nu(s,X_s^{i,N},\mu_s^{X,N})ds + \int_0^t \bar{\sigma}(s,X_s^{i,N},\mu_s ^{X,N})dW^i_s 
    \end{equation}
 almost surely for all $t \in [0,T]$ and  $i\in \{1,\ldots,N\}$   where 
 \begin{align*}
\mu_t^{X,N}  : =\frac{1}{N} \sum_{j=1}^N \delta_{X_t^{j,N}}
\end{align*}
is the empirical law of the particles $\{X_t^{i,N}\}_{i \in \{1,\ldots, N\}}$. 
Clearly,   from Equation \eqref{eq:v:sig}, 
\begin{align*}
     \nu(t, X_t^{i,N}, \mu_t^{X,N}) & = \frac{1}{N}\sum_{j=1}^{N}f(X_t^{i,N},X_t^{j,N})   + b(t,X_t^{i,N},\mu_t^{X,N}), 
     \\
    \mbox{ and }\ \bar{\sigma} (t, X_t^{i,N},\mu_t^{X,N} ) & =  \frac{1}{N}\sum_{j=1}^{N}g(X_t^{i,N},X_t^{j,N}) + \sigma(t,X_t^{i,N},\mu_t^{X,N}) 
 \end{align*}
 almost surely for all $i\in \{1,\ldots,N\}$ and $t \in [0,T]$.
The convergence of IPS \eqref{IPS} to nIPS \eqref{NIPS} is popularly known in the literature  as \textit{propagation of chaos} (PoC).

In addition, the following assumptions are needed.
Recall $q > 0$ and $p_0 > 2(q+1)$ from Section \ref{sec:main:welpos}. 
\begin{assumption}{\label{as:anti-sys}}
 The Vlasov kernel $f$ is anti-symmetric, i.e., 
  $f(x,y)  = - f(y,x)$ for all $x,y \in \mathbb R^d$.
Also, there exists a constant $L>0$ such that for all $x,y \in \mathbb R^d$
        \begin{align*}
        (|x|^{p_0 -2}-|y|^{p_0 -2})\langle x+y,f(x,y) \rangle & \leq L (|x|^{p_0}+|y|^{p_0}),
        \\
        \langle x-y,f(x,y) \rangle + 2(p_0-1) |g(x,y)|^2 &\leq L (1+|x-y|^2).
        \end{align*} 
    \end{assumption}
The existence and uniqueness of the IPS \eqref{IPS} is due to \cite{gyongy1980}. 
However, the moment estimate from \cite{gyongy1980} cannot be used directly as the bound appears therein depends on the number of particles $N$ that explode when $N$ tends to infinity. In the following, we provide the moment bound for IPS \eqref{IPS}, which does not explode. 
\begin{theorem}[Existence and uniqueness of IPS] \label{thm:eu-ips}
     Let the assumptions of Theorem \ref{E-U} hold. 
     Then, there exists a unique strong solution to  IPS \eqref{IPS}  associated  with MV-SDE \eqref{Mckean}. In addition, if Assumption \ref{as:anti-sys} is satisfied, then
    \begin{align*}
         \sup_{i \in \{1,\ldots,N \}} \sup_{t \in [0, T]}\mathbb E |X_t ^{i,N}|^{p_0}  \leq K(1+ \mathbb E|X_0|^{p_0})e^{KT}
    \end{align*}
   where $K>0$ is a constant independent of $N$. 
\end{theorem}
\begin{proof}
The existence of a unique strong solution of IPS  \eqref{IPS} follows from \cite{gyongy1980}.
Now, using It\^{o}\textquotesingle{s} formula, 
   \begin{align*}
        \mathbb E |X_t^{i,N}|^{p_0}   = &\, \mathbb E |X_0|^{p_0}  + p_0 \mathbb E \int_{0}^t  |X_s^{i,N}|^{p_0-2}  \Big\langle X_s^{i,N}, b (s,X_{s}^{i,N},\mu_{s}^{X,N}) + \frac{1}{N} \sum_{j=1}^N f (X_{s}^{i,N},X_{s}^{j,N}) \Big\rangle  ds 
        \\
        & + p_0 \mathbb E \int_{0}^t  |X_{s}^{i,N}|^{p_0-2}  \Big\langle X_{s}^{i,N}, \big(\sigma (s,X_{s}^{i,N},\mu_{s}^{X,N}) + \frac{1}{N} \sum_{j=1}^N g (X_{s}^{i,N},X_{s}^{j,N,n}) \big) dW^i_s\Big\rangle   
        \\
        & +  \frac{p_0(p_0-2)}{2} \int_{0}^t  \mathbb E |X_{s}^{i,N}|^{p_0-4} \Big|\Big\{\sigma (s,X_{s}^{i,N},\mu_{s}^{X,N}) + \frac{1}{N} \sum_{j=1}^N g (X_{s}^{i,N},X_{s}^{j,N}) \Big\}^*X_{s}^{i,N}\Big|^2 ds 
        \\
        & + \frac{p_0}{2} \mathbb E \int_{0}^t  |X_{s}^{i,N}|^{p_0-2} \Big|\sigma (s,X_{s}^{i,N},\mu_{s}^{X,N}) + \frac{1}{N} \sum_{j=1}^N g (X_{s}^{i,N},X_{s}^{j,N}) \Big|^2 ds
        \\
      \leq &\, \mathbb E |X_0|^{p_0}  + p_0 \mathbb E \int_{0}^t  |X_s^{i,N}|^{p_0-2} \big\{ \big\langle X_s^{i,N}, b (s,X_{s}^{i,N},\mu_{s}^{X,N})  \big\rangle + (p_0-1) \big|\sigma (s,X_{s}^{i,N},\mu_{s}^{X,N})\big|^2 \big\} ds
      \\
     &  + p_0 \mathbb E \int_{0}^t  |X_s^{i,N}|^{p_0-2} \big\{ \big\langle X_s^{i,N}, \frac{1}{N} \sum_{j=1}^N f (X_{s}^{i,N},X_{s}^{j,N})  \big\rangle + (p_0-1) \big|\frac{1}{N} \sum_{j=1}^N g (X_{s}^{i,N},X_{s}^{j,N})\big|^2 \big\} ds
\end{align*}
which on averaging over all particles and using Assumption \ref{as:eu:b:sig} yields, \begin{align*}
   \frac{1}{N} \sum_{i=1}^N \mathbb E |X_t^{i,N}|^{p_0}  \leq &\, \mathbb E |X_0|^{p_0}  +   \frac{K}{N} \sum_{i=1}^N \mathbb E \int_{0}^t  |X_s^{i,N}|^{p_0-2} \big\{ 1+ | X_s^{i,N}|^2+ W^{(2)}(\mu_{s}^{X,N}, \delta_0)^2  \big\} ds
      \\
     &  +  \frac{p_0}{N^2} \sum_{i, j=1}^N \mathbb E \int_{0}^t  |X_s^{i,N}|^{p_0-2} \big\{ \big\langle X_s^{i,N},  f (X_{s}^{i,N},X_{s}^{j,N})  \big\rangle + (p_0-1) \big| g (X_{s}^{i,N},X_{s}^{j,N})\big|^2 \big\} ds
\end{align*} 
 and then using Corollary \ref{lem:symm-growth}  
\begin{align*}
     \frac{1}{N} \sum_{i=1}^N \mathbb E |X_t^{i,N}|^{p_0}  \leq &\, \mathbb E |X_0|^{p_0}+ K+ \int_0^t \frac{1}{N} \sum_{i=1}^N \mathbb E |X_s^{i,N}|^{p_0} ds
\end{align*}
 for all $t \in [0,T]$. 
 The proof of moment estimate finishes on using Gr\"onwall's inequality and exchangeability of IPS (see \cite[Section 2.1.2]{carmona2018b}). 
\end{proof}

In the remainder of this section, we provide two results on the rate of convergence of PoC. In Theorem \ref{PoC} the PoC convergence rate degrades across the ambient dimension parameter $d$ of the MV-SDE \eqref{Mckean}. This is related to the use of a $W^{(2)}$-convergence result in \cite[Theorem 5.8]{carmona2018a} and is unavoidable as $\mu\mapsto (b,\sigma) (\cdot,\cdot,\mu)$ are assumed to be general $W^{(2)}$-Lipschitz functionals.
When additional structure is available, for instance via Vlasov kernels that in no way need to be $W^{(1)}$-Lipschitz (as existing present day literature \cite{belomestny2018projected,zhang2025,Hinds2025wellposedness}), then the convergence rate of the propagation of chaos can be shown to hold uniformly over the dimension $d$---our second main result is Theorem \ref{dim-ind-poc} and it holds far beyond the assumptions in \cite{belomestny2018projected,zhang2025,Hinds2025wellposedness}.

\subsection[Propagation of chaos -- General case]{Propagation of chaos -- The general Wasserstein-2 case}
\label{sec:PoC-d-dep}
The following theorem gives the dimension-dependent  PoC convergence rate.  
 \begin{theorem}[Dimension-dependent PoC convergence rate]{\label{PoC}}
Let the conditions of Theorem \ref{thm:eu-ips} hold true for some  $p_0\geq 5$ and $p_0>   2(q+1)$ ($q>0$).  
Then,  the IPS \eqref{IPS} converges to the nIPS  \eqref{NIPS} with the following rate of convergence,
        \begin{align*}
        \underset{i\in \{1,\ldots,N\}}{\sup} ~ \underset{t \in [0, T]}{\sup} \mathbb E |X_t^{i,N}-X_t^i|^2 \leq K \begin{cases}
            N^{-1/2},          & \text{if } d < 4, \\
             N^{-1/2}\log N,   & \text{if} ~ d=4, \\
             N^{-2/d},         & \text{if } d > 4.
        \end{cases}
    \end{align*}
\end{theorem}


\begin{proof} 

Recall that well-posedness and moment estimates for the nIPS \eqref{NIPS} and IPS \eqref{IPS} were established in Theorem \ref{E-U} and Theorem \ref{thm:eu-ips} respectively. 
 
Notice that as Assumption \ref{assum_initial} holds for  $p_0=5$, Assumption \ref{as:eu:f:g} yields that there exists a constant $L>0$ such that 
    \begin{align}
             \langle (x-y)-(x'-y')  , f(x,y) - f(x',y') \rangle  + 4|g(x , y)-g(x', y')|^2 
         \leq L|(x-y)-(x'-y')|^2 \label{eq:poc-coer}
    \end{align}
     for all $x,x',y, y'\in \mathbb R^d$.
Now, recall nIPS \eqref{NIPS} and IPS \eqref{IPS}, and apply It\^{o}'s formula to obtain,
\begin{align}
         \mathbb E |X_t^i -  X_t^{i,N}|^2  \leq  & \,  2 \, \mathbb E \int_0^t  \{ \langle X_s^i - X_s^{i,N}, b(s,X_s^i,\mu_s^{X^i}) - b(s,X_{s}^{i,N},\mu_s ^{X,N}) \rangle   + |\sigma(s,X_s^i,\mu_s ^{X^i})-\sigma(s,X_s^{i,N},\mu_s ^{X,N})|^2 \}ds \notag
         \\
         &  + 2 \, \mathbb E\int_0^t \Big\langle X_s^i - X_s^{i,N}, \int_{\mathbb R^d} f(X_s^i, x) \mu_s^{X^i}(dx) - \frac{1}{N} \sum_{j=1}^N f({X}^{i,N}_{s},{X}^{j,N}_{s}) \Big\rangle ds \notag
         \\
         &  +2 \, \mathbb E \int_0^t  \Bigl| \int_{\mathbb R^d} g(X_s^i , x) \mu_s^{X^i}(dx) -\frac{1}{N}\sum_{j=1}^N g({X}^{i,N}_{s},{X}^{j,N}_{s})\Bigl|^2ds \label{eq:ito:poc:dd}
\end{align}
         which, on averaging over all particles and  using Assumption \ref{as:eu:b:sig} yields, 
         \begin{align*}
\frac{1}{N} \sum_{i=1}^N \mathbb E  |X_t^i  - X_t^{i,N}|^2 
          \leq & \, \frac{K}{N}  \int_0^t  \sum_{i=1}^N  \mathbb E |X_s^i - X_s^{i,N}|^2 ds + \frac{K}{N} \sum_{i=1}^N   \int_0^t \mathbb E W^{(2)} (\mu_s^{X^i}, \mu_s ^{X,N})^2 ds 
          \\
         &  + \frac{2}{N} \sum_{i=1}^N \int_0^t  \mathbb E\Big\langle X_s^i - X_s^{i,N}, \frac{1}{N} \sum_{j=1}^N \int_{\mathbb R^d} \big(f(X_s^i, x) -  f({X}^{i}_{s},{X}^{j}_{s})\big)\mu_s^{X^i}(dx) \Big \rangle ds 
         \\
         &  + \frac{4}{N^3} \sum_{i=1}^N \int_0^t  \mathbb E \Bigl | \sum_{j=1}^N \int_{\mathbb R^d} (g(X_s^i , x) -  g({X}^{i}_{s},{X}^{j}_{s}))  \mu_s^{X^i}(dx) \Bigl |^2 ds 
         \\
          + \frac{2}{N^2} \sum_{i,j=1}^N \mathbb E\int_0^t & \, \big\{  \langle X_s^i - X_s^{i,N},   f({X}^{i}_{s},{X}^{j}_{s}) - f({X}^{i,N}_{s},{X}^{j,N}_{s})  \rangle + 2  |  g({X}^{i}_{s},{X}^{j}_{s}) - g({X}^{i,N}_{s},{X}^{j,N}_{s})  |^2  \big\}
         ds
    \end{align*}
    for all $t \in [0,T]$. 
    Thus, on using Young's inequality and equation \eqref{eq:poc-coer} (in the last term of the above inequality), one gets, 
\begin{align*}
   \frac{1}{N} \sum_{i=1}^N \mathbb E  |X_t^i - X_t^{i,N}|^2 & \, \leq   K  \int_0^t \frac{1}{N} \sum_{i=1}^N \mathbb E|X_s^i - X_s^{i,N}|^2 ds + \frac{K}{N} \sum_{i=1}^N   \int_0^t \mathbb E W^{(2)} (\mu_s^{X^i}, \mu_s ^{X,N})^2 ds 
   \\
      + \frac{1}{N} \sum_{i=1}^N&    \int_0^t \mathbb E |X_s^i - X_s^{i,N}|^2 ds +  \frac{1}{N} \sum_{i=1}^N \int_0^t \mathbb E \Bigl |\frac{1}{N} \sum_{j=1}^N \int_{\mathbb R^d} (f(X_s^i , x)  -  f({X}^{i}_{s},{X}^{j}_{s}) ) \mu_s^{X^i}(dx) \Bigl |^2 ds
     \\
      &  + \frac{4}{N^3} \sum_{i=1}^N  \mathbb E\int_0^t \Bigl | \sum_{j=1}^N \int_{\mathbb R^d} (g(X_s^i , x)  -  g({X}^{i}_{s},{X}^{j}_{s}) ) \mu_s^{X^i}(dx) \Bigl |^2 ds 
      \\
     &  + \frac{1}{N^2} \sum_{i,j=1}^N  \mathbb E\int_0^t  \big\{ \big \langle (X_s^i - X_s^{i,N})-(X_s^j - X_s^{j,N}),   f({X}^{i}_{s},{X}^{j}_{s}) - f({X}^{i,N}_{s},{X}^{j,N}_{s})  \big \rangle 
     \\
      & \qquad + 4 |  g({X}^{i}_{s},{X}^{j}_{s}) - g({X}^{i,N}_{s},{X}^{j,N}_{s}) |^2 \big\}
         ds 
\end{align*}
for all $t \in [0,T]$. 
Now, note that
 \begin{align*}
     \mathbb E  \Bigl | & \sum_{j=1}^N \int_{\mathbb R^d} (f(X_s^i , x)  -  f({X}^{i}_{s},{X}^{j}_{s}) ) \mu_s^{X^i}(dx) \Bigl |^2 
     \\
    & = \sum_{j=1}^N \sum_{k=1}^N \mathbb E \Bigl [ \Bigl \langle \int_{\mathbb R^d} (f(X_s^i , x)  -  f({X}^{i}_{s},{X}^{j}_{s}) ) \mu_s^{X^i}(dx),\int_{\mathbb R^d} (f(X_s^i , x)  -  f({X}^{i}_{s},{X}^{k}_{s}) ) \mu_s^{X^i}(dx) \Bigl \rangle \Bigl]
    \end{align*}
    for all $s \in [0,T]$. 
Observe that the cross-product terms, i.e., when $j \neq k$, are zero upon 
using the independent and identical distribution property of the nIPS \eqref{NIPS}. 
Consequently, due to Remark \ref{rem:fg:grw:eu} and Theorem \ref{E-U},  one deduces
    \begin{align}
     \mathbb E \Bigl | \sum_{j=1}^N \int_{\mathbb R^d}& (f(X_s^i , x)  -  f({X}^{i}_{s},{X}^{j}_{s}) ) \mu_s^{X^i}(dx) \Bigl |^2  = \sum_{j=1}^N \mathbb E \Bigl|\int_{\mathbb R^d} (f(X_s^i , x)  -  f({X}^{i}_{s},{X}^{j}_{s}) ) \mu_s^{X^i}(dx) \Bigl |^2 \notag
    \\
    & \leq K \sum_{j=1}^N \mathbb E \int_{\mathbb R^d} ((1+|X_s^i- x|)^{2(q+1)} + (1+|{X}^{i}_{s}-{X}^{j}_{s}|)^{2(q+1)} ) \mu_s^{X^i}(dx) \leq KN \label{eq:f:rate}
 \end{align}
 for all $s \in [0,T]$. 
By performing similar calculations (using Remark \ref{rem:fg:grw:eu}), one also obtains 
\begin{align}
    \mathbb E & \Bigl | \sum_{j=1}^N \int_{\mathbb R^d} (g(X_s^i , x)  -  g({X}^{i}_{s},{X}^{j}_{s}) ) \mu_s^{X^i}(dx) \Bigl |^2  \leq KN \label{eq:g:rate}
\end{align}
for all $s \in [0,T]$. Using the above estimates,  one obtains, 
 \begin{align*}
     \frac{1}{N} \sum_{i=1}^N \mathbb E  |X_t^i - X_t^{i,N}|^2 & \leq  K  \int_0^t \frac{1}{N} \sum_{i=1}^N  \mathbb E|X_s^i - X_s^{i,N}|^2 ds  + \frac{K}{N}  + \frac{K}{N} \sum_{i=1}^N   \int_0^t \mathbb E W^{(2)} (\mu_s^{X^i}, \mu_s ^{X,N})^2 ds 
     \\
     & \leq  K  \int_0^t \frac{1}{N} \sum_{i=1}^N  \mathbb E|X_s^i - X_s^{i,N}|^2 ds  + \frac{K}{N}  + \frac{K}{N} \sum_{i=1}^N   \int_0^t \mathbb E W^{(2)} \Big(\mu_s^{X^i}, \frac{1}{N} \sum_{j=1}^N \delta_{X_s^{j}}\Big)^2 ds 
 \end{align*}
 for all $t \in [0,T]$ and thus Gr\"onwall\textquotesingle{s} lemma yields, 
 \begin{align*}
     \frac{1}{N} \sum_{i=1}^N \mathbb E  |X_t^i - X_t^{i,N}|^2 & \leq \frac{K}{N}  + \frac{K}{N} \sum_{i=1}^N   \int_0^t \mathbb E W^{(2)} \Big(\mu_s^{X^i}, \frac{1}{N} \sum_{j=1}^N \delta_{X_s^{j}}\Big)^2 ds
 \end{align*}
for all $t \in [0,T]$. From 
\cite[Theorem 5.8]{carmona2018a}, one has
\begin{align}{\label{rate:empi}}
 \mathbb E W^{(2)} \Bigl(\frac{1}{N}\sum_{j=1}^N \delta_{X_t^j}, \mu_t ^{X} \Bigl)^2  = \begin{cases}
            N^{-1/2},        & \text{if } d < 4, 
            \\
             N^{-1/2}\log N,  & \text{if} ~ d=4, 
             \\
             N^{-2/d},        & \text{if } d > 4, 
        \end{cases}
\end{align}
for all $t \in [0,T]$ and using the fact that IPS \eqref{IPS} and nIPS \eqref{NIPS} are exchangeable (see \cite[Section 2.1.2]{carmona2018b}), which in turn completes the proof.         
\end{proof}
\subsection{Propagation of chaos -- the sharp (dimension-independent) rate result}
\label{sec:PoC-ind}
For the purpose of  dimension-independent  PoC convergence rate,  consider a special form of MV-SDE \eqref{Mckean} by taking the following forms of $v$ and $\bar{\sigma}$, 
 \begin{align} \label{eq:nu-barsig-inp}
        \nu(t,x,\mu) & = \int_{\mathbb R^d} f(x,y)\mu(dy) + \int_{\mathbb R^d} \tilde{b}(t,x,y) \mu(dy) \mbox{ and } 
        \overline{\sigma}(t,x,\mu)  = \int_{\mathbb R^d} g(x,y)\mu(dy)+ \int_{\mathbb R^d} \tilde{\sigma}(t,x,y) \mu(dy)
    \end{align}
for all $t \in [0,T], x\in \mathbb R^d$ and $\mu \in \mathcal{P}_2 (\mathbb R ^d)$ where $(\tilde{b}, \tilde{\sigma}):[0,T] \times \mathbb R^d \times \mathbb R^d \mapsto (\mathbb R^d, \mathbb R^{d\times l})$ are Borel-measurable functions, and $f$ and $g$ as considered before.
These particular forms of $\nu$ and $\bar{\sigma}$ (as given in Equation \eqref{eq:nu-barsig-inp}) are for this subsection and thereafter we return to their original forms given in Equation \eqref{eq:v:sig}. 
However, the same notations, namely $X$ and $X^{i, N}$, are used for the solution of MV-SDE \eqref{Mckean} and its IPS \eqref{IPS}, which  should not cause any confusion in the minds of the readers. 

In this special case, the sharp PoC convergence rate is achieved to be $1/2$ for any $p \in [2, 2p_0/(q+1)]$, unlike the general $W^{(2)}$-measure dependence case discussed above in Theorem \ref{PoC} where the dimension-dependent PoC convergence rate is proved in mean square.  
Thus, fix a constant $p \in [2, 2p_0/(q+1)]$ and make assumptions as given below.  
\begin{assumption}{\label{as:DIPoC}}
There exists a constant $L>0$ such that
    \begin{align*}
      \langle x, \tilde{b}(t,x,y) \rangle  +(p_0-1) |\tilde{\sigma}(t,x,y)|^2 & \leq L \{|x|^2+|y|^2\},
      \\
    \langle x-x', \tilde{b}(t,x,y) - \tilde{b}(t,x',y') \rangle  +2(p-1) |\tilde{\sigma} (t,x,y) - \tilde{\sigma}(t,x',y')|^2 & \leq L \{|x-x'|^2+|y-y'|^2\},
    \\
    |\tilde{b}(t,x,y) - \tilde{b}(t,x,y')|  & \, \leq L |y-y'|, 
    \end{align*}
for all $t \in [0,T]$ and $x$, $x'$, $y$, $y'$ $\in \mathbb R^d$.
The maps $\tilde{b}$ and $\tilde{\sigma}$ are assumed jointly continuous in their variables.
\end{assumption}
\begin{assumption}\label{as:f:g:PoC-2}
    There exists a constant $L>0$ such that 
    \begin{align*}
        (|x-x'|^{p-2}-|y-y'|^{p-2})  \langle  (x+y) -(x'+y'),f(x,y)  - f(x',y')\rangle \leq &\, L(|x-x'|^{p}+ |y-y'|^{p}), 
        \\
         \langle (x  - y)-  (x'  - y'), f(x,y) -  f(x',y')  \rangle  + 4 (p-1) \, |  g(x,y) - g(x',y')  |^2   \leq  & \, L |(x-y)-(x'-y')|^2,  
    \end{align*}
    for all  $x,y,x',y' \in \mathbb R^d$.
\end{assumption}
Clearly, Assumption \ref{as:DIPoC} implies Assumption \ref{as:eu:b:sig} if we take
\begin{align*}
    b(t,x,\mu)= \int_{\mathbb R^d} \tilde{b}(t,x,y) \mu(dy) \mbox{ and } \sigma(t,x,\mu)= \int_{\mathbb R^d} \tilde{\sigma}(t,x,y) \mu(dy),
\end{align*}
for all $t \in [0,T]$, $x \in \mathbb R^d$ and $\mu \in \mathcal{P}_2(\mathbb{R}^d)$.  
Thus, due to Theorem \ref{E-U}, MV-SDE \eqref{Mckean} with coefficients as given in Equation \eqref{eq:nu-barsig-inp} possesses a unique strong solution as well as the $p_0$-moment stability. 
For the sharp PoC convergence order of $1/2$, additional assumptions are needed on the kernels $f$ and $g$. 
\begin{remark} \label{rem:til:sig:lips}
    Assumption \ref{as:DIPoC} implies that there exists a constant $K>0$ such that 
   \begin{align*}
       |\tilde{\sigma}(t,x,y) - \tilde{\sigma}(t,x,y')|^2 & \leq K |y-y'|^2,
   \end{align*}
   for all $t \in [0,T]$ and $x$, $y$, $y'$ $\in \mathbb R^d$.
\end{remark}

For the sharp rate PoC result next, we keep the anti-symmetry condition of $f$ from Assumption \ref{as:anti-sys}, but the remainder of that assumption is replaced by Assumptions \ref{as:DIPoC}  and  \ref{as:f:g:PoC-2}. Further, as this is a special case of MV--SDE \ref{Mckean}, therefore well-posedness and moments follow from Theorem \ref{E-U}  and Theorem \ref{thm:eu-ips}.

\begin{theorem}[Sharp PoC rate across dimension] \label{dim-ind-poc}

Let $q>0$, $p_0 >2(q+1)$ and take $p \in [2, 2p_0/(q+1)]$. Let $f$ be anti-symmetric and Assumptions \ref{assum_initial}, \ref{as:eu:f:g},  \ref{as:DIPoC}  and  \ref{as:f:g:PoC-2} hold. Take IPS \eqref{IPS} and nIPS \eqref{NIPS} (and MV-SDE \eqref{Mckean}) with coefficients as given in Equation \eqref{eq:nu-barsig-inp}. 

Then, the conclusions of Theorem \ref{thm:eu-ips} and Theorem \ref{E-U} holds for the equations, i.e., wellposedness and moment estimates. 
Moreover, the IPS \eqref{IPS} with coefficients given in Equation \eqref{eq:nu-barsig-inp} 
converges to the corresponding nIPS \eqref{NIPS} with order $1/2$ (uniformly across the dimension parameter $d$) in the following sense
    \begin{align*}
     \displaystyle  \sup_{i \in \{1,\ldots,N \}} \sup_{t \in [0,T]} \mathbb E|X_t^i - X_t^{i,N}|^{p} \leq KN^{-\frac{p}{2}}
    \end{align*}
  for any $p \in [2, 2p_0/(q+1)]$ where $K>0$ is a constant independent of $N \in \mathbb N$.     
\end{theorem}
Since $p_0>2(q+1)$ then $2p_0/(q+1)>4$ thus $p$ will always satisfy $p\in [2,4]$ at the very least. 
We highlight that the proof of this result makes use of the conditional Rosenthal inequality (Lemma \ref{lemma:conditionalRosenthalIneq}) which in turn requires establishing high-order moments for the difference $|X_t^i - X_t^{i,N}|$. 
It is for this reason that the Assumption \ref{as:f:g:PoC-2}  is needed. 
On the other hand, the conditional Rosenthal inequality allows us to fully avoid using Kantorovich-Rubinstein duality arguments on $W^{(1)}$, as in  \cite{PiersHinds2025,zhang2025}, and hence have superlinear growth in $f$ and $g$. 

Lastly, we remark that the mechanics we use to prove  Theorem \ref{dim-ind-poc} easily recovers the results of \cite{zhang2025} (and \cite{PiersHinds2025} for the non-reflection case) and thus our results are broader in spectrum. 
Also, if $\nu$ and $\bar{\sigma}$ in Equation \eqref{eq:nu-barsig-inp} are replaced by 
\begin{align*}
        \nu(t,x,\mu) & = \int_{\mathbb R^d} f(x,y)\mu(dy) + F\Big(t,x,\int_{\mathbb R^d} \tilde b(t,x,y)\mu(dy)\Big)
        \\
        \mbox{ and } 
        \overline{\sigma}(s,x,\mu)  & = \int_{\mathbb R^d} g(x,y)\mu(dy)+ G\Big(t,x,\int_{\mathbb R^d} \tilde \sigma(t,x,y)\mu(dy)\Big) 
    \end{align*}
where $F:[0,T] \times \mathbb{R}^d \times \mathbb{R}^d \mapsto \mathbb R^d$ and $G:[0,T] \times \mathbb{R}^d \times \mathbb{R}^{d\times l} \mapsto \mathbb R^{d\times l}$, then our proof of the above theorem can be adapted even when $F$ is assumed to be one-sided Lipschitz continuous in the second variable and Lipschitz continuous in the third variable, and $G$ is Lipschitz continuous in both second and third variables, whilst \cite{zhang2025} allows for only Lipschitz continuity in $F$ and $G$ in the both variables. 
Thus, in our framework, both $f$ and $g$ are superlinear in both variables and $F$ is superlinear in the second variable, but these functions are linear in the framework of \cite{zhang2025}. 
It is worth mentioning that the well-posedness of MV-SDE \eqref{Mckean} and the associated IPS \eqref{IPS} with the above coefficients follows from our Theorem \ref{E-U} and Theorem \ref{thm:eu-ips}, respectively. 
\begin{proof}
The existence and uniqueness of IPS \eqref{IPS} with coefficients \eqref{eq:nu-barsig-inp} follows from \cite{gyongy1980} with the non-explosive moment estimate of IPS \eqref{IPS} established using the same arguments used in the proof of Theorem \ref{thm:eu-ips}. The well-posedness of MV-SDE \eqref{Mckean} with coefficients \eqref{eq:nu-barsig-inp} follows from our Theorem \ref{E-U}. 

We now address the PoC result. Using It\^o's formula we have, 
   \begin{align*}
       \frac{1}{N} \sum_{i=1}^N & \mathbb E|X_t^i  - X_t^{i,N}|^p 
       \\
       &
       \leq 
       \frac{p}{N} \sum_{i=1}^N \int_0^t \mathbb E|X_s^i - X_s^{i,N}|^{p-2} \Big\{ \Big\langle X_s^i - X_s^{i,N}, \int_{\mathbb R^d} \tilde{b}(s,X_s^{i},y)\mu_s^{X^i}(dy) - \int_{\mathbb R^d} \tilde{b}(s,X_s^{i,N},y)\mu_s^{X,N}(dy)\Big\rangle \nonumber
         \\
         & \qquad +(p-1) \Big|\int_{\mathbb R^d} \tilde{\sigma}(s,X_s^{i},y)\mu_s^{X^i}(dy)  - \int_{\mathbb R^d} \tilde{\sigma}(s,X_s^{i,N},y)\mu_s^{X,N}(dy) \Big|^2 \Big\} ds \nonumber
         \\
        &  +\frac{p}{N} \sum_{i=1}^N  \mathbb E\int_0^t |X_s^i - X_s^{i,N}|^{p-2} \Bigl \{\Big \langle X_s^i - X_s^{i,N}, \int_{\mathbb R^d} f(X_s^i, x) \mu_s^{X^i}(dx) - \frac{1}{N} \sum_{j=1}^N f({X}^{i,N}_{s},{X}^{j,N}_{s}) \Big \rangle \nonumber
        \\
        & \qquad + (p-1) \Bigl| \int_{\mathbb R^d} g(X_s^i,x) \mu_s^{X^i}(dx) -\frac{1}{N}\sum_{j=1}^N g({X}^{i,N}_{s},{X}^{j,N}_{s})\Bigl|^2 \Bigl \}ds  \notag
    \end{align*}
    which, on further simplification and Young's inequality yields, 
    \begin{align}
    \nonumber 
      \frac{1}{N} \sum_{i=1}^N & \mathbb E|X_t^i  - X_t^{i,N}|^p 
      \leq  \, 2(p-1)(2p-3)  \int_0^t \frac{1}{N}\sum_{i=1}^N  \mathbb E|X_s^i - X_s^{i,N}|^{p} ds  \notag
     \\
   &  + \frac{p}{N} \sum_{i=1}^N \int_0^t \mathbb E|X_s^i - X_s^{i,N}|^{p-2} \Bigl\{ \Big\langle X_s^i - X_s^{i,N},  \frac{1}{N} \sum_{j=1}^N \tilde{b}(s,X_s^{i},X_s^j) - \frac{1}{N} \sum_{j=1}^N \tilde{b}(s,X_s^{i,N},X_s^{j,N}) \Big\rangle \nonumber
         \\
  & \qquad +2(p-1) \Big|  \frac{1}{N} \sum_{j=1}^N \tilde{\sigma}(s,X_s^{i},X_s^j)- \frac{1}{N} \sum_{j=1}^N \tilde{\sigma}(s,X_s^{i,N},X_s^{j,N})\Big|^2 \Big\} ds \notag
   \\
     & + \frac{p}{N^2} \sum_{i,j=1}^N  \mathbb E\int_0^t |X_s ^i  - X_s ^{i,N}|^{p-2}  \big\{ \langle X_s^i - X_s^{i,N},  f({X}^{i}_{s},{X}^{j}_{s}) - f({X}^{i,N}_{s},{X}^{j,N}_{s}) \rangle \nonumber
      \\
      & \qquad  + 2(p-1) |  g({X}^{i}_{s},{X}^{j}_{s}) - g({X}^{i,N}_{s},{X}^{j,N}_{s})  |^2 \big\} ds \nonumber
  \\
  &+ \int_0^t \frac{1}{N}\sum_{i=1}^N \Big|\int_{\mathbb R^d} \tilde{b}(s,X_s^{i},y)\mu_s^{X^i}(dy) - \frac{1}{N} \sum_{j=1}^N \tilde{b}(s,X_s^{i},X_s^j) \Big|^p ds \nonumber
         \\
   &  +4(p-1)\frac{1}{N} \sum_{i=1}^N \int_0^t\Big|\int_{\mathbb R^d} \tilde{\sigma}(s,X_s^{i},y)\mu_s^{X^i}(dy)  - \frac{1}{N} \sum_{j=1}^N \tilde{\sigma}(s,X_s^{i},X_s^j)\Big|^p  ds \nonumber
  \\
       &\, + \int_0^t  \frac{1}{N} \sum_{i=1}^N  \mathbb E \Big| \int_{\mathbb R^d} f(X_s^i, x) \mu_s^{X^i}(dx) - \frac{1}{N} \sum_{j=1}^N f({X}^{i}_{s},{X}^{j}_{s}) \Big|^p ds  \nonumber
    \\
     & + 4(p-1) \int_0^t \frac{1}{N}\sum_{i=1}^N \Bigl| \int_{\mathbb R^d} g(X_s^i,x) \mu_s^{X^i}(dx) -\frac{1}{N}\sum_{j=1}^N g({X}^{i}_{s},{X}^{j}_{s})\Bigl|^p ds  \nonumber
      \\
        =: &\, 2(p-1)(2p-3)  \int_0^t \frac{1}{N}\sum_{i=1}^N  \mathbb E|X_s^i - X_s^{i,N}|^{p} ds  + A_1 + A_2+ A_3+ A_4+ A_5+ A_6 \label{A1+.+A6},
   \end{align}
for all $t \in [0,T]$. 
To estimate $A_1$, by using Assumption \ref{as:DIPoC} and Young\textquotesingle{s} inequality, one obtains
  \begin{align}
    A_{1}:= & \, \frac{p}{N^2} \sum_{i, j=1}^N \int_0^t \mathbb E|X_s^i - X_s^{i,N}|^{p-2} \big\{ \big\langle X_s^i - X_s^{i,N},   \tilde{b}(s,X_s^{i},X_s^j) -  \tilde{b}(s,X_s^{i,N},X_s^{j,N}) \big\rangle \nonumber
         \\
  & \qquad +2(p-1) \big|   \tilde{\sigma}(s,X_s^{i},X_s^j)-  \tilde{\sigma}(s,X_s^{i,N},X_s^{j,N})\big|^2 \big\} ds \nonumber
  \\
  \leq & \, \frac{K}{N} \int_{0}^t \sum_{i=1}^N \mathbb E|X_s^i - X_s^{i,N}|^{p}ds \label{eq:A1}
\end{align}
for all $t \in [0,T]$. 
Using Assumption \ref{as:DIPoC}, Lemma \ref{L_rate_x^i}, $A_{2.3}$ can be estimated by
\begin{align}
    A_{2} :=&\, \frac{p}{N^2} \sum_{i, j=1}^N  \mathbb E\int_0^t |X_s ^i  - X_s ^{i,N}|^{p-2}  \big\{ \langle X_s^i - X_s^{i,N},  f({X}^{i}_{s},{X}^{j}_{s}) - f({X}^{i,N}_{s},{X}^{j,N}_{s}) \rangle \notag
      \\
     & + 2(p-1) |  g({X}^{i}_{s},{X}^{j}_{s}) - g({X}^{i,N}_{s},{X}^{j,N}_{s})  |^2 \big\} ds  \leq K \int_0 ^t \frac{1}{N} \sum_{i=1}^N \mathbb E |X_s^{i} - X_s^{i,N}|^p ds \label{A2}
\end{align}
for all $t \in [0,T]$.
For the estimation of $A_3$, notice that
\begin{align*}
     A_{3} := & \frac{1}{N} \sum_{i=1}^N \int_0^t \mathbb E  \Big| \int_{\mathbb R^d} \tilde{b}(s,X_s^{i},y)  \mu_s^{X^i}(dy) - \frac{1}{N} \sum_{j=1}^N \tilde{b}(s,X_s^{i},X_s^j) )\Big|^p ds 
    \\
     = &\, \frac{1}{N^{p+1}} \sum_{i=1}^N \int_0^t \mathbb E \Big|  \int_{\mathbb R^d} (\tilde{b}(s,X_s^{i},y) - \tilde{b}(s,X_s^{i},X_s^{i}) ) \mu_s^{X^i}(dy)\Big|^p ds  
    \\
    & + \frac{1}{N^{p+1}} \sum_{i=1}^N \int_0^t \mathbb E \Big|\sum_{j\neq i}^N \int_{\mathbb R^d} (\tilde{b}(s,X_s^{i},y) - \tilde{b}(s,X_s^{i},X_s^{j}) ) \mu_s^{X^i}(dy)\Big|^p ds 
\end{align*}
for all $t \in [0,T]$. Let \begin{align*}
         Y_s^{i,j} := \displaystyle \int_{\mathbb R^d} 
              \big(\tilde{b}(s,X_s^i,y)  -  \tilde{b}(s,{X}^{i}_{s},{X}^{j}_{s}) \big) \mu_s^{X^i}(dy)  
     \end{align*}
 for all $s\in [0,T]$, $i, j \in \{1, \ldots,N\}$.  
 Since particles are independently and identically distributed, $Y_s^{i,1}$,$\ldots$,$Y_s^{i,j-1}$, $Y_s^{i,j+1}$,$\ldots$, $Y_s^{i,N}$ are conditionally independent given $X^i_s$ and $\mathbb EY_s^{i,i}\equiv 0$ and $\mathbb E[Y_s^{i,j}|X_s^i]=0$ whenever $j \neq i$ for all $i, j \in \{1, \ldots, N\}$. Thus, due to the conditional Rosenthal-type inequality (see Lemma \ref{lemma:conditionalRosenthalIneq}),
\begin{align*}
    A_{3} 
     \leq & \frac{1}{N^{p+1}}\sum_{i=1}^N \int_0^t \mathbb E \int_{\mathbb R^d}  |\tilde{b}(s,X_s^i,y)  -  \tilde{b}(s,{X}^{i}_{s},{X}^{i}_{s}) |^p \mu_s^{X^i}(dy) ds 
     \\
     & + \frac{1}{N^{p+1}} \sum_{i=1}^N \int_0^t \sum_{j\neq i} \mathbb E \int_{\mathbb R^d}   | \tilde{b}(s,X_s^i, y)  -  \tilde{b}(s,{X}^{i}_{s},{X}^{j}_{s}) |^p \mu_s^{X^i}(dy) ds
    \\
    &  + \frac{1}{N^{p+1}} \sum_{i=1}^N\mathbb E\int_0^t \Bigl ( \sum_{j\neq i}  \mathbb E\Bigl[\int_{\mathbb R^d}   | (\tilde{b}(s,X_s^i, y)  -  \tilde{b}(s,{X}^{i}_{s},{X}^{j}_{s})|^2 \mu_s^{X^i}(dy) \Big| X_s^i \Bigr]\Bigl)^{p/2}  ds
\end{align*}
and then Assumption \ref{as:DIPoC} and Theorem \ref{E-U} yields
\begin{align}
    A_{3}
    & \leq \frac{K}{N^{p+1}}\sum_{i=1}^N \int_0^t \mathbb E \int_{\mathbb R^d}  |y-X_s^i|^p \mu_s^{X^i}(dy) ds + \frac{K}{N^{p+1}} \sum_{i=1}^N \int_0^t \sum_{j\neq i} \mathbb E \int_{\mathbb R^d} |y-X_s^j|^p \mu_s^{X^i}(dy) ds  \notag
    \\
    & \quad + \frac{K (N-1)^{p/2-1}}{N^{p+1}} \sum_{i=1}^N\mathbb E\int_0^t \sum_{j\neq i}  \mathbb E\int_{\mathbb R^d} |y-X_s^j|^p \mu_s^{X^i}(dy) ds  \leq KN^{-p/2} \label{eq:A3}
\end{align}
for all $t\in [0,T]$. Performing similar calculations, due to Remark \ref{rem:til:sig:lips}, one can get
\begin{align}
   A_{4} := \frac{4(p-1)}{N}\sum_{i=1}^N \int_0^t\mathbb E \Bigl|\int_{\mathbb R^d} \tilde{\sigma}(s,X_s^{i},y)\mu_s^{X^i}(dy) - \frac{1}{N} \sum_{j=1}^N \tilde{\sigma}(s,X_s^{i},X_s^j) \Bigl|^p ds  \leq KN^{-p/2} \label{eq:A4}
\end{align}
for all $t\in [0,T]$. For the estimation of $A_{5}$, using the similar arguments as for $Y_s^{i,j}$ and using conditional Rosenthal-type inequality, one can obtain
\begin{align*}
    A_{5}  : = & \, \frac{K}{N} \sum_{i=1}^N  \mathbb E\int_0^t \Bigl |\frac{1}{N} \sum_{j=1}^N \int_{\mathbb R^d} (f(X_s^i , x)  -  f({X}^{i}_{s},{X}^{j}_{s}) ) \mu_s^{X^i}(dx) \Bigl |^p ds
    \\
    & \, 
 \leq \frac{K}{N^{p+1}}\sum_{i=1}^N \int_0^t \mathbb E \int_{\mathbb R^d}  |f(X_s^i , x)  -  f({X}^{i}_{s},{X}^{i}_{s}) |^p \mu_s^{X^i}(dx) ds 
 \\
 & \quad + \frac{K}{N^{p+1}} \sum_{i=1}^N \int_0^t \sum_{j\neq i} \mathbb E \int_{\mathbb R^d}   | f(X_s^i, x)  -  f({X}^{i}_{s},{X}^{j}_{s}) |^p \mu_s^{X^i}(dx) ds \nonumber
   \\
  & \quad  + \frac{K}{N^{p+1}} \sum_{i=1}^N\mathbb E\int_0^t \Bigl ( \sum_{j\neq i}  \mathbb E\Bigl[\int_{\mathbb R^d}   | f(X_s^i, x)  -  f({X}^{i}_{s},{X}^{j}_{s})|^2 \mu_s^{X^i}(dx) \Big| X_s^i \Bigr]\Bigl)^{p/2}  ds \nonumber
\end{align*}
which, on using Remark \ref{rem:fg:grw:eu}, H\"older's inequality and Theorem \ref{E-U}, gives
\begin{align}
    A_{5} & \leq\frac{K}{N^{p+1}}\sum_{i=1}^N \int_0^t \mathbb E \int_{\mathbb R^d}  (1+|X_s^i - x|+|X^{i}_{s}-{X}^{j}_{s}|)^{p(q+1)} \mu_s^{X^i}(dx) ds \nonumber
    \\
   & \quad + \frac{K}{N^{p+1}} \sum_{i=1}^N \int_0^t \sum_{j\neq i} \mathbb E \int_{\mathbb R^d}   (1+|X_s^i - x|+|X^{i}_{s}-{X}^{j}_{s}|)^{p(q+1)} \mu_s^{X^i}(dx) ds \nonumber
   \\
  & \quad + \frac{K (N-1)^{p/2-1}}{N^{p+1}} \sum_{i=1}^N\mathbb E\int_0^t \sum_{j\neq i}  \mathbb E\int_{\mathbb R^d}   (1+|X_s^i - x|+|X^{i}_{s}-{X}^{j}_{s}|)^{p(q+1)} \mu_s^{X^i}(dx)   ds \leq   KN^{-p/2} \label{A5}
\end{align}
for all $t \in [0,T]$. Similarly, one obtains
\begin{align}
    A_{6} : = \frac{4(p-1)}{N} \sum_{i=1}^N  \mathbb E\int_0^t \Bigl |\frac{1}{N} \sum_{j=1}^N \int_{\mathbb R^d} (g(X_s^i , x)  -  g({X}^{i}_{s},{X}^{j}_{s}) ) \mu_s^{X^i}(dx) \Bigl |^p ds \leq   KN^{-p/2} \label{A6}
\end{align}
for all $t \in [0,T]$.

Finally, gathering all estimates from \eqref{eq:A1} to \eqref{A6} and substituting them in \eqref{A1+.+A6}, applying Gr\"onwall's lemma and the exchangeability of particles concludes the desired result. 
\end{proof}


\section{The tamed Euler scheme for the interacting particle system and its convergence rate}
\label{sec:Taming-in Euler scheme}
In order to introduce an explicit tamed Euler scheme for the interacting particle system (IPS) \eqref{IPS}, the  interval $[0,T]$ is divided into  $n$ subintervals, each of length $h=1/n$, i.e., $t_k = k /n=kh$ for all $k \in \{0,\ldots, nT \}$. Recall 
the constant $q>0$ from Assumption \ref{as:eu:f:g} (also from Assumption \ref{as:b:poly} given below) and define the following taming method for the coefficients $b$, $\sigma$, $f$ and $g$ as follows;   
\begin{align}
\label{taming-FiniteT}
   & (b^n, \sigma^n)(t,x,\mu)  := \frac{(b, \sigma)(t,x,\mu)}{1+n^{- 1/2}|x|^{2q}} \mbox{ and }  (f^n, g^n)(x,y)  := \frac{(f,g)(x,y)}{1+n^{-1/2}|x-y|^{2q}},  
\end{align}
respectively, for all $t \in [0,T]$, $x,y\in \mathbb R^d$ and $\mu \in \mathcal{P}_2(\mathbb R^d)$. 
In this article,  the  following tamed Euler scheme is proposed for IPS  \eqref{IPS} 
\begin{align}
\label{Eulercheme-FiniteT}
    \hat{X}_{t_{k+1}} ^{i,N,n}
     = \hat{X}_{t_{k}} ^{i,N,n} 
    & + \Bigl (b^n (t_{k},\hat{X}_{t_{k}} ^{i,N,n}, \hat{\mu}_{t_k}^{X,N,n}) + \frac{1}{N} \sum_{j=1}^N f^n(\hat{X}_{t_{k}} ^{i,N,n},\hat{X}_{t_{k}} ^{j,N,n}) \Bigl )h 
    \nonumber
    \\
    &  + \Bigl ( \sigma^n (t_{k},\hat{X}_{t_{k}} ^{i,N,n}, \hat{\mu}_{t_k}^{X,N,n})+ \frac{1}{N} \sum_{j=1}^N g^n(\hat{X}_{t_{k}} ^{i,N,n} ,\hat{X}_{t_{k}} ^{j,N,n}) \Bigl) \Delta W_{t_k}^i
\end{align}
almost surely with initial value  $X_0 ^{i,N,n} = X_0 ^i$ where $\Delta W_{t_k}^i= W_{t_{k+1}}^i-W_{t_k}^i$ for all $i \in \{1,\ldots,N\}$   and $\hat{\mu}_{t_k}^{X,N,n} = \displaystyle \frac{1}{N} \sum_{j=1}^N \delta_{\hat{X}_{t_k}^{j,N,n}}$.
Also, let  $k_n(t)=\lfloor nt \rfloor/ n$   and define the time-continuous version of the tamed Euler scheme \eqref{Eulercheme-FiniteT}  as 
\begin{align}{\label{scheme}}
    \hat{X}_t^{i,N,n} 
    = X_0^i 
    & + \int_0^t \Bigl (b^n(k_n(s),\hat{X}_{k_n(s)}^{i,N,n},\hat{\mu}_{k_n(s)}^{X,N,n}) + \frac{1}{N} \sum_{j=1}^N f^n(\hat{X}_{k_n(s)} ^{i,N,n},\hat{X}_{k_n(s)} ^{j,N,n}) \Bigl )ds \nonumber\\
    & + \int_0^t  \Bigl ( \sigma^n (k_n(s),\hat{X}_{k_n(s)} ^{i,N,n}, \hat{\mu}_{k_n(s)}^{X,N,n})+ \frac{1}{N} \sum_{j=1}^N g^n(\hat{X}_{k_n(s)} ^{i,N,n} ,\hat{X}_{k_n(s)} ^{j,N,n}) \Bigl) dW_{s}^i
\end{align}
almost surely for all $t \in [0,T]$ and $i \in \{1,\ldots,N\}$.\
\subsection{Moment bounds}
In order to establish the moment bound of the  scheme \eqref{scheme},   make further assumption about the coefficient $b(t,x,\mu)$. 
\begin{assumption}
\label{as:b:poly}
There exists a constant $L>0$ such that
    \begin{align*}
         |b(t,x,\mu)-b(t,x',\mu')| & \leq L\{(1+|x|+|x'|)^{q}|x-x'|+W^{(2)} (\mu,\mu')\}, 
    \end{align*}
    for   all  $t\in [0,T]$, $x,x' \in \mathbb R^d$ and $ \mu,\mu' \in \mathcal{P}_2(\mathbb R^d)$.
\end{assumption}
The following remark is an immediate consequence of the above assumption and Assumption \ref{as:eu:b:sig}. 
These polynomial Lipschitz assumptions on $b$ and $f$ are needed to show the sharp rate $1/2$ of convergence of the scheme \eqref{scheme}. 
\begin{remark}\label{rem:b:sig:gwt:tem}
  From Assumptions \ref{as:eu:b:sig} and \ref{as:b:poly}, there is a constant $K>0$  such that 
\begin{align*}
|\sigma(t,x,\mu)-\sigma(t,x',\mu')|^2 & \leq K\{(1+|x|+|x'|)^{q}|x-x'|^2+W^{(2)} (\mu,\mu')^2\}, 
\\
   |b(t,x,\mu)|  & \leq K\{ (1+|x|)^{q+1} + W^{(2)}(\mu, \delta_0) \},
   \\
|\sigma(t,x,\mu)|^2 & \leq K\{ (1+|x|)^{q+2} + W^{(2)}(\mu, \delta_0)^2 \},  
\end{align*}
    for all $t \in [0,T]$, $x,  x'\in \mathbb R^d$ and $\mu, \mu' \in \mathcal{P}_2(\mathbb R^d)$.
\end{remark}
Through the above remark, one identifies the polynomial growth in $b$ and $\sigma$ while the Remark \ref{rem:fg:grw:eu} gives polynomial growth in the kernels $f$ and $g$. 
These growth restrictions are useful for identifying the growth of their tamed versions $b^n$, $\sigma^n$, $f^n$ and $g^n$ which are listed in the following remark. 
\begin{remark}{\label{R_f}}
From Equation \eqref{taming-FiniteT}, Remarks \ref{rem:fg:grw:eu} and \ref{rem:b:sig:gwt:tem}, there is a constant $K>0$,  independent of $n\in \mathbb N$, such that 
    \begin{align*}
    |b^n (t,x, \mu)| & \leq K \min \big\{n^{1/4} (1+|x|) +W^{(2)}(\mu, \delta_0), |b(t,x, \mu)| \}, 
    \\
    |\sigma^n (t,x, \mu)|^2 & \leq K  \min\{ n^{{1}/{4}} (1+|x|)^2+W^{(2)}(\mu, \delta_0)^2 ,  |\sigma(t,x, \mu)|^2 \}, 
    \\
     |f^n (x,y)| & \leq K \min\{ n^{1/4} (1+|x-y|), |f(x,y)| \},
     \\
     |g^n (x,y)|^2 & \leq K \min \{ n^{1/4} (1+|x-y|^2) ,|g(x,y)|^2 \},
    \end{align*}
    for all $t \in [0,T], x, y\in \mathbb R^d$ and $\mu \in \mathcal{P}_2(\mathbb R^d)$. 
\end{remark}
Observe that the coercivity conditions about $(b,\sigma)$ and $(f,g)$ as mentioned in  Assumption \ref{as:eu:b:sig} and Remark \ref{rem:fg:grw:eu}, respectively, play a crucial role in the proof of moment estimates of IPS \eqref{IPS} (see Theorem \ref{thm:eu-ips} above). 
The following remark provides similar coercivity conditions on their tamed versions $(b^n,\sigma^n)$ and $(f^n,g^n)$, resptively and they are used in the proof of  moment bound of the scheme \eqref{scheme}, see Lemma \ref{higher MB} below. 
\begin{remark}\label{rem:bn:sign:gwt:tem} 
By using Assumption \ref{as:eu:b:sig},  Remark \ref{rem:fg:grw:eu} and Equation \eqref{taming-FiniteT}, there is a constant $K>0$ such that   
    \begin{align*}
     \langle x, b^n (t,x, \mu)\rangle + (p_0 -1)|\sigma^n (t,x, \mu )|^2 & \leq K \{ 1+ |x|^2  + W^{(2)} (\mu,\delta_0)^2 \}, 
     \\
      2 \langle x-y,f^n(x,y) \rangle + (p_0 -1) |g^n(x,y)|^2 & \leq K \{1 + |x-y|^2 \}, 
\end{align*} 
for all $t \in [0,T]$, $x,y\in \mathbb R^d$,  $\mu \in \mathcal{P}_2(\mathbb R^d)$ and $n \in \mathbb N$. 
\end{remark} 
The following remark shows that the tamed version $f^n$ of $f$ satisfies conditions similar  to the one given in Assumption \ref{as:anti-sys} for the function  $f$. 
This is being used in proving the moment bound of the scheme \eqref{scheme} in Lemma \ref{higher MB}.   .  
\begin{remark} \label{rem:anti-sys:tem}
    By Assumption \ref{as:anti-sys}, the tamed function $f^n$ defined in Equation \eqref{taming-FiniteT} satisfies the anti-symmetric property, i.e., $f^n(x,y)=-f^n(y,x)$ for all $x$, $y \in \mathbb R^d$. 
    Further, there exists a constant $K>0$ such that
        \begin{align*}
        (|x|^{p_0 -2}-|y|^{p_0 -2})\langle x+y,f^n(x,y) \rangle & \leq K (|x|^{p_0}+|y|^{p_0}),
        \\
        \langle x-y,f^n(x,y) \rangle + 2(p_0-1) |g^n(x,y)|^2 &\leq K (1+|x-y|^2), 
        \end{align*}
        for all $x,y \in \mathbb R^d$. 
 
\end{remark}
On the basis of the observations made in  Remarks \ref{rem:bn:sign:gwt:tem} and \ref{rem:anti-sys:tem}, one can deduce that  $(b^n, \sigma^n)$ and $(f^n, g^n)$, as defined  in Equation \eqref{taming-FiniteT},  are indeed  correct choices of taming $(b, \sigma)$ and $(f,g)$, respectively. 

In the following remark, we discuss the case when the coefficients are linear.    
\begin{remark}
  In case the coefficients $b$, $\sigma$, $f$ and $g$ are globally Lipschitz continuous, i.e., if $q\equiv 0$, then their tamed versions given in Equation \eqref{taming-FiniteT} are no longer required. 
  Indeed, one can take $b^n= b$, $\sigma^n = \sigma$, $f^n=f$ and $g^n=g$ in the scheme \eqref{scheme} to define the Euler scheme of IPS \eqref{IPS}. 
  Furthermore, if $\sigma$ and $g$ are constants, then to tame the coefficients $b$ and $f$, one can use the following forms
   \begin{align*}
      b^n(t,x,\mu) :=\frac{b(t,x,\mu)}{1+n^{-1}|x|^{4q}} \mbox{ and } f^n(x,y)  := \frac{f(x,y)}{1+n^{-1}|x-y|^{4q}},  
  \end{align*}
 respectively, for all $n\in \mathbb N$, $t\in [0,T]$, $x,y \in \mathbb R^d$ and $\mu \in \mathcal{P}_2(\mathbb R^d)$. 
 In such a case, the rate of strong convergence in $L^p$-norm can be shown to be $1.0$ by adapting the approach developed in this paper. 
\end{remark}
It should be noted from Remark \ref{R_f} that the growth of these tamed functions is linear, but the constants appearing on the right side depend on $n$, which explode when $n$ goes to infinity.  
Due to these reasons, new techniques are developed to establish moment bound of the scheme \eqref{scheme}. 
We begin by estimating one-step error of the scheme \eqref{scheme}  which subsequently plays a critical role in our convergence analysis. 
\begin{lemma}{\label{L, diff, X_t}}
Let the assumptions of Theorem \ref{E-U} hold. 
In addition, Assumption \ref{as:b:poly} is also satisfied. 
 Then,
\begin{align}
     \mathbb E \big |\hat{X}_t^{i,N,n}- \hat{X}_{k_n(t)}^{i,N,n}  \bigl|^{p_0} \leq K n^{-3p_0/8} \mathbb E \bigl\{1+\bigl|\hat{X}_{k_n(t)}^{i,N,n}\bigr|^{p_0} +\frac{1}{N} \sum_{j=1}^N \bigl|\hat{X}_{k_n(t)}^{j,N,n}\bigr|^{p_0}\bigr\} \notag
\end{align}
for all $t \in [0,T], i \in \{1,\ldots,N \}$ where $K>0$ is a constant independent of $n,N \in \mathbb N.$
\end{lemma}
This result is established using only the linear growth condition on the tamed coefficients given in Remark \ref{R_f}; it does not require Theorem \ref{thm:eu-ips} nor Assumption \ref{as:anti-sys}. 
    \begin{proof} 
By Equation \eqref{scheme}, one can write 
        \begin{align}
    \mathbb E \bigl|\hat{X}_t^{i,N,n} - \hat{X}_{k_n(t)}^{i,N,n}\bigr|^{p_0}  
     \leq  & K n^{-p_0}  \mathbb E \Big\{\big |b^n(k_n(t),\hat{X}_{k_n(t)}^{i,N,n},\hat{\mu}_{k_n(t)}^{X,N,n})\bigr|^{p_0} + \frac{1}{N} \sum_{j=1}^N \bigl|f^n(\hat{X}_{k_n(t)} ^{i,N,n},\hat{X}_{k_n(t)} ^{j,N,n})\big|^{p_0}  \Big\}\notag
    \\
    & + K n^{-p_0/2}  \mathbb E \Big\{\bigl|\sigma^n (k_n(t),\hat{X}_{k_n(t)} ^{i,N,n}, \hat{\mu}_{k_n(t)}^{X,N,n})\big|^{p_0}+ \frac{1}{N} \sum_{j=1}^N \bigl|g^n(\hat{X}_{k_n(t)} ^{i,N,n} ,\hat{X}_{k_n(t)} ^{j,N,n}) \bigl|^{p_0} \Big\}  \notag
\end{align}
for all $t \in [0,T]$ and $i \in \{1,\ldots,N\}$ and 
 Remark \ref{R_f} concludes the proof.
    \end{proof}
\begin{lemma}{\label{higher MB}}
Let the conditions of Theorem \ref{thm:eu-ips} be satisfied and let Assumption \ref{as:b:poly} hold. 
Then, 
  \begin{align*}
      \sup_{t \in [0,T]} ~ \sup_{i \in \{1,\ldots,N \}} \mathbb E |\hat{X}_t ^{i,N,n}|^{p_0}  \leq K
  \end{align*}
  where the constant $K>0$ is independent of $N,n \in \mathbb N$. 
\end{lemma}  
\begin{proof}
   From Equation \eqref{scheme}  and It\^{o}\textquotesingle{s} formula, 
   \begin{align*}
        \mathbb E |\hat{X}_t^{i,N,n}&|^{p_0}   = \mathbb E |X_0^{i}|^{p_0}  + p_0 \mathbb E \int_{0}^t  |\hat{X}_s^{i,N,n}|^{p_0-2}  \Big\langle \hat{X}_s^{i,N,n}, b^n (k_n(s),\hat{X}_{k_n(s)}^{i,N,n},\hat{\mu}_{k_n(s)}^{X,N,n}) + \frac{1}{N} \sum_{j=1}^N f^n (\hat{X}_{k_n(s)}^{i,N,n},\hat{X}_{k_n(s)}^{j,N,n}) \Big\rangle  ds 
        \\
        & + p_0 \mathbb E \int_{0}^t  |\hat{X}_s^{i,N,n}|^{p_0-2}  \Big\langle \hat{X}_s^{i,N,n}, \big(\sigma^n (k_n(s),\hat{X}_{k_n(s)}^{i,N,n},\hat{\mu}_{k_n(s)}^{X,N,n}) + \frac{1}{N} \sum_{j=1}^N g^n (\hat{X}_{k_n(s)}^{i,N,n},\hat{X}_{k_n(s)}^{j,N,n}) \big) dW^i_s\Big\rangle   
        \\
        & +  \frac{p_0(p_0-2)}{2} \int_{0}^t  \mathbb E |\hat{X}_s^{i,N,n}|^{p_0-4} \Big|\big(\sigma^n (k_n(s),\hat{X}_{k_n(s)}^{i,N,n},\hat{\mu}_{k_n(s)}^{X,N,n}) + \frac{1}{N} \sum_{j=1}^N g^n (\hat{X}_{k_n(s)}^{i,N,n},\hat{X}_{k_n(s)}^{j,N,n}) \big)^*\hat{X}_s^{i,N,n}\Big|^2 ds 
        \\
        & + \frac{p_0}{2} \mathbb E \int_{0}^t  |\hat{X}_s^{i,N,n}|^{p_0-2} \Big|\sigma^n (k_n(s),\hat{X}_{k_n(s)}^{i,N,n},\hat{\mu}_{k_n(s)}^{X,N,n}) + \frac{1}{N} \sum_{j=1}^N g^n (\hat{X}_{k_n(s)}^{i,N,n},\hat{X}_{k_n(s)}^{j,N,n}) \Big|^2 ds
  \end{align*}
  and then averaging over all particles yields, 
  \begin{align}
      \frac{1}{N}  \sum_{i=1}^N & \mathbb E |\hat{X}_t^{i,N,n}|^{p_0} \leq \frac{1}{N}  \sum_{i=1}^N \mathbb E |X_0^{i}|^{p_0} \notag
      \\
            & +  \frac{p_0}{N}  \sum_{i=1}^N \mathbb E \int_{0}^t  |\hat{X}_s^{i,N,n}|^{p_0-2}  \Big\{ \big\langle \hat{X}_{k_n(s)}^{i,N,n}, b^n (k_n(s),\hat{X}_{k_n(s)}^{i,N,n},\hat{\mu}_{k_n(s)}^{X,N,n}) \big\rangle + (p_0-1)\big|\sigma^n (k_n(s),\hat{X}_{k_n(s)}^{i,N,n},\hat{\mu}_{k_n(s)}^{X,N,n}) \big|^2 \Big\} ds \notag
           \\
        & +  \frac{p_0}{N}  \sum_{i=1}^N \mathbb E \int_{0}^t  |\hat{X}_s^{i,N,n}|^{p_0-2} \Big\{ \Big\langle \hat{X}_{k_n(s)}^{i,N,n}, \frac{1}{N} \sum_{j=1}^N f^n (\hat{X}_{k_n(s)}^{i,N,n},\hat{X}_{k_n(s)}^{j,N,n}) \Big\rangle + (p_0-1)  \frac{1}{N} \sum_{j=1}^N \big|g^n (\hat{X}_{k_n(s)}^{i,N,n},\hat{X}_{k_n(s)}^{j,N,n}) \big|^2 \Big\}ds \notag
        \\
           & +  \frac{p_0}{N}  \sum_{i=1}^N \mathbb E \int_{0}^t  |\hat{X}_s^{i,N,n}|^{p_0-2}  \Big\langle \hat{X}_s^{i,N,n}-\hat{X}_{k_n(s)}^{i,N,n}, b^n (k_n(s),\hat{X}_{k_n(s)}^{i,N,n},\hat{\mu}_{k_n(s)}^{X,N,n}) + \frac{1}{N} \sum_{j=1}^N f^n (\hat{X}_{k_n(s)}^{i,N,n},\hat{X}_{k_n(s)}^{j,N,n}) \Big\rangle  ds \notag
           \\
           =:& \frac{1}{N} \sum_{i=1}^N \mathbb E |X_0^{i}|^{p_0} + U_1+ U_2 + U_3 \label{U_1 + U_2 + U_3}
       \end{align}
   for all $t\in [0,T]$.  
For the estimation of $U_1$, one uses Remark \ref{rem:bn:sign:gwt:tem}, the inequality $\displaystyle W^{(2)}(\hat{\mu}_{k_n(s)}^{X,N,n},\delta_0)^2  \leq \frac{1}{N} \sum_{j=1}^N  \big|\hat{X}_{k_n(s)}^{j,N,n}|^{2}$ and  Young\textquotesingle{s} inequality to get the following estimate, 
\begin{align}
    U_1  : = & \, \frac{p_0}{N}  \sum_{i=1}^N \mathbb E \int_{0}^t  |\hat{X}_s^{i,N,n}|^{p_0-2}  \Big\{ \big\langle \hat{X}_{k_n(s)}^{i,N,n}, b^n (k_n(s),\hat{X}_{k_n(s)}^{i,N,n},\hat{\mu}_{k_n(s)}^{X,N,n}) \big\rangle + (p_0-1)\big|\sigma^n (k_n(s),\hat{X}_{k_n(s)}^{i,N,n},\hat{\mu}_{k_n(s)}^{X,N,n}) \big|^2 \Big\} ds \notag
    \\
    \leq  & \, K   \int_{0}^t  \frac{1}{N} \sum_{i=1}^N \mathbb E|\hat{X}_s^{i,N,n}|^{p_0-2}  \big\{1 + |\hat{X}_{k_n(s)} ^{i,N,n}|^2 + W^{(2)}(\hat{\mu}_{k_n(s)}^{X,N,n},\delta_0)^2   \big\} ds \notag
    \\
     \leq &  \, K+ K\int_{0}^t  \frac{1}{N} \sum_{i=1}^N \big\{ \mathbb E|\hat{X}_s^{i,N,n}|^{p_0} + \mathbb E|\hat{X}_{k_n(s)}^{i,N,n}|^{p_0} \big\} ds \leq K +K\int_{0}^t    \frac{1}{N} \sum_{i=1}^N \sup_{r \in [0, s]} \mathbb E|\hat{X}_r^{i,N,n}|^{p_0} ds  \label{U1}
\end{align}
   for all $t\in [0, T]$. 
   Also, for estimating $U_2$,
   \begin{align}
       U_2 : = &  p_0 \int_{0}^t  \frac{1}{N^2} \sum_{i,j=1}^N \mathbb E  |\hat{X}_s^{i,N,n}|^{p_0-2} \Big\{ \Big\langle \hat{X}_{k_n(s)}^{i,N,n},  f^n (\hat{X}_{k_n(s)}^{i,N,n},\hat{X}_{k_n(s)}^{j,N,n}) \Big\rangle + (p_0-1)   \big|g^n (\hat{X}_{k_n(s)}^{i,N,n},\hat{X}_{k_n(s)}^{j,N,n}) \big|^2 \Big\}ds \notag
       \\
       = & p_0 \int_{0}^t  \frac{1}{N^2} \sum_{i,j=1}^N \mathbb E  |\hat{X}_{k_n(s)}^{i,N,n}|^{p_0-2} \Big\{ \Big\langle \hat{X}_{k_n(s)}^{i,N,n},  f^n (\hat{X}_{k_n(s)}^{i,N,n},\hat{X}_{k_n(s)}^{j,N,n}) \Big\rangle + (p_0-1)   \big|g^n (\hat{X}_{k_n(s)}^{i,N,n},\hat{X}_{k_n(s)}^{j,N,n}) \big|^2 \Big\}ds \notag
       \\
        +  p_0 \int_{0}^t  \frac{1}{N^2} & \sum_{i,j=1}^N \mathbb E  \big\{|\hat{X}_s^{i,N,n}|^{p_0-2}- |\hat{X}_{k_n(s)}^{i,N,n}|^{p_0-2} \big\}\Big\{ \Big\langle \hat{X}_{k_n(s)}^{i,N,n},  f^n (\hat{X}_{k_n(s)}^{i,N,n},\hat{X}_{k_n(s)}^{j,N,n}) \Big\rangle + (p_0-1)   \big|g^n (\hat{X}_{k_n(s)}^{i,N,n},\hat{X}_{k_n(s)}^{j,N,n}) \big|^2 \Big\}ds \notag
   \end{align}
   for all $t \in [0, T]$. 
   Using Corollary \ref{lem:symm-growth}, Lemma \ref{L, diff, X_t}, Cauchy-Schwarz inequality and Young\textquotesingle{s} inequality,  one can obtain
\begin{align}
    U_2  \leq & K \int_{0}^t     \frac{1}{N} \sum_{i,j=1}^N \mathbb E|\hat{X}_{k_n(s)}^{i,N,n}|^{p_0} ds \nonumber
    \\
    &  + K \int_{0}^t   \mathbb E  \frac{1}{N^2} \sum_{i,j=1}^N n^{1/4}(|\hat{X}_s^{i,N,n}|^{p_0-3} + |\hat{X}_{k_n(s)}^{i,N,n}|^{p_0-3}) |\hat{X}_s^{i,N,n}-\hat{X}_{k_n(s)}^{i,N,n}| (1+|\hat{X}_{k_n(s)}^{i,N,n}|^2+|\hat{X}_s^{i,N,n}-\hat{X}_s^{j,N,n}|^2)ds \nonumber
    \\
     \leq & K +K\int_{0}^t    \frac{1}{N} \sum_{i=1}^N \sup_{r \in [0, s]} \mathbb E|\hat{X}_r^{i,N,n}|^{p_0} ds {\label{U2}}
\end{align}
 for all $t\in [0, T]$. 
 Further, $U_3$ can be written as
\begin{align}
    U_3: =&  \frac{p_0}{N} \sum_{i=1}^N  \int_{0}^t  \mathbb E \big|\hat{X}_s^{i,N,n}|^{p_0-2}  \Big\langle \hat{X}_s^{i,N,n}-\hat{X}_{k_n(s)}^{i,N,n}, b^n (k_n(s),\hat{X}_{k_n(s)}^{i,N,n},\hat{\mu}_{k_n(s)}^{X,N,n}) + \frac{1}{N} \sum_{j=1}^N f^n (\hat{X}_{k_n(s)}^{i,N,n},\hat{X}_{k_n(s)}^{j,N,n}) \Big\rangle  ds \nonumber
    \\ 
     = &  \frac{p_0}{N} \sum_{i=1}^N  \int_{0}^t  \mathbb E \big|\hat{X}_{k_n(s)}^{i,N,n}|^{p_0-2}  \Big\langle \hat{X}_s^{i,N,n}-\hat{X}_{k_n(s)}^{i,N,n}, b^n (k_n(s),\hat{X}_{k_n(s)}^{i,N,n},\hat{\mu}_{k_n(s)}^{X,N,n}) + \frac{1}{N} \sum_{j=1}^N f^n (\hat{X}_{k_n(s)}^{i,N,n},\hat{X}_{k_n(s)}^{j,N,n}) \Big\rangle  ds \nonumber
    \\
    &  + \frac{p_0}{N} \sum_{i=1}^N \int_{0}^t   \mathbb E (\big|\hat{X}_s^{i,N,n}|^{p_0-2}  -\big|\hat{X}_{k_n(s)}^{i,N,n}|^{p_0-2}) \Big\langle \hat{X}_s^{i,N,n}-\hat{X}_{k_n(s)}^{i,N,n}, b^n (k_n(s),\hat{X}_{k_n(s)}^{i,N,n},\hat{\mu}_{k_n(s)}^{X,N,n}) \nonumber\\
    & \qquad  \qquad + \frac{1}{N} \sum_{j=1}^N f^n (\hat{X}_{k_n(s)}^{i,N,n},\hat{X}_{k_n(s)}^{j,N,n}) \Big\rangle  ds := U_{31} + U_{32} {\label{U3}}
\end{align}
for all $t\in [0, T]$. 
By using Cauchy-Schwarz inequality, Young\textquotesingle{s} inequality and Remark \ref{R_f}, $U_{31}$ can be estimated for all $t \in [0, T]$ as 
\begin{align}
\nonumber 
 U_{31}:
 &= p_0 \int_{0}^t  \frac{1}{N} \sum_{i=1}^N \mathbb E \big|\hat{X}_{k_n(s)}^{i,N,n}|^{p_0-2}
 \\ 
 &
 \qquad \qquad \qquad \times \Big\langle \hat{X}_s^{i,N,n}-\hat{X}_{k_n(s)}^{i,N,n}, b^n (k_n(s),\hat{X}_{k_n(s)}^{i,N,n},\hat{\mu}_{k_n(s)}^{X,N,n}) + \frac{1}{N} \sum_{j=1}^N f^n (\hat{X}_{k_n(s)}^{i,N,n},\hat{X}_{k_n(s)}^{j,N,n}) \Big\rangle  ds \nonumber
    \\
    & =  p_0 \int_{0}^t  \frac{1}{N} \sum_{i=1}^N \mathbb E \big|\hat{X}_{k_n(s)}^{i,N,n}|^{p_0-2}  \Big\langle \int _{k_n(s)}^s (b^n (k_n(r),\hat{X}_{k_n(r)} ^{i,N,n}, \hat{\mu}_{k_n(r)}^{X,N,n})dr +  \frac{1}{N} \sum_{j=1}^N f^n (\hat{X}_{k_n(r)} ^{i,N,n},\hat{X}_{k_n(r)} ^{j,N,n}))dr, \nonumber\\
    &  
    \qquad \qquad \qquad 
    b^n (k_n(s),\hat{X}_{k_n(s)}^{i,N,n},\hat{\mu}_{k_n(s)}^{X,N,n}) + \frac{1}{N} \sum_{j=1}^N f^n (\hat{X}_{k_n(s)}^{i,N,n},\hat{X}_{k_n(s)}^{j,N,n}) \Big\rangle  ds \nonumber
    \\
    & \leq K n^{-1} \int_{0}^t  \frac{1}{N} \sum_{i=1}^N \mathbb E \big|\hat{X}_{k_n(s)}^{i,N,n}|^{p_0-2} \bigl \{| b^n (k_n(s),\hat{X}_{k_n(s)}^{i,N,n},\hat{\mu}_{k_n(s)}^{X,N,n})|^2 +  \frac{1}{N} \sum_{j=1}^N |f^n (\hat{X}_{k_n(s)}^{i,N,n},\hat{X}_{k_n(s)}^{j,N,n})|^2 \bigl \} \nonumber
    \\
    & \leq K +K\int_{0}^t    \frac{1}{N} \sum_{i=1}^N \sup_{r \in [0, s]} \mathbb E|\hat{X}_r^{i,N,n}|^{p_0} ds.  \label{U31}
\end{align}
To estimate $U_{32}$, due to Remark \ref{R_f} and Lemma \ref{L, diff, X_t}, one can observe that for all $t\in [0, T]$
 \begin{align}
 \nonumber 
    U_{32} 
    & :=  p_0 \int_{0}^t  \frac{1}{N} \sum_{i=1}^N \mathbb E (\big|\hat{X}_s^{i,N,n}|^{p_0-2}  -\big|\hat{X}_{k_n(s)}^{i,N,n}|^{p_0-2}) \Big\langle \hat{X}_s^{i,N,n}-\hat{X}_{k_n(s)}^{i,N,n}, b^n (k_n(s),\hat{X}_{k_n(s)}^{i,N,n},\hat{\mu}_{k_n(s)}^{X,N,n}) 
    \\ \nonumber 
    & \qquad \qquad \qquad + \frac{1}{N} \sum_{j=1}^N f^n (\hat{X}_{k_n(s)}^{i,N,n},\hat{X}_{k_n(s)}^{j,N,n}) \Big\rangle  ds
    \\ \nonumber 
    & \leq K \int_{0}^t  \frac{n^{1/4}}{N} \sum_{i=1}^N \mathbb E (\big|\hat{X}_s^{i,N,n}|^{p_0-3} 
    \\ \nonumber 
    & \qquad \qquad \qquad 
    +\big|\hat{X}_{k_n(s)}^{i,N,n}|^{p_0-3})|\hat{X}_s^{i,N,n}-\hat{X}_{k_n(s)}^{i,N,n}|^2 \Big\{(1+|\hat{X}_{k_n(s)}^{i,N,n}|) + \frac{1}{N} \sum_{i=1}^N (1+|\hat{X}_{k_n(s)}^{i,N,n}-\hat{X}_{k_n(s)}^{j,N,n}|) \Big\}ds \nonumber
    \\
    & \leq K +K\int_{0}^t    \frac{1}{N} \sum_{i=1}^N \sup_{r \in [0, s]} \mathbb E|\hat{X}_r^{i,N,n}|^{p_0} ds.
    \label{U32}
 \end{align}
  Substituting \eqref{U31} and \eqref{U32} in \eqref{U3}, one obtains the following: 
 \begin{align}
     U_3 \leq K +K\int_{0}^t    \frac{1}{N} \sum_{i=1}^N \sup_{r \in [0, s]} \mathbb E|\hat{X}_r^{i,N,n}|^{p_0} ds  \label{U3b}
 \end{align}
 for all $t\in [0, T]$.
 Finally, substituting estimates from \eqref{U1}, \eqref{U2} and \eqref{U3b} into \eqref{U_1 + U_2 + U_3} and then applying the Gr\"onwall\textquotesingle{s} lemma and exchangeability of the particles which hold due to exchangeability of the initial values (see \cite[Section 2.1.2]{carmona2018b}),  we get the desired result.
\end{proof}

\subsection{Convergence rate of the scheme}
In this section, the strong convergence rate of the tamed Euler scheme \eqref{scheme} in $L^p$-norm is established (see Theorem \ref{thm:rate}) for $p$ satisfying $ p \in [2,   p_1)$ and $p \in [2,  \frac{p_0}{3q +1}]$ where $q>0$ and $p_0 > 2(q+1)$ are from Assumptions \ref{assum_initial}, \ref{as:eu:b:sig} and \ref{as:b:poly}, respectively, and $p_1$ comes from Assumption \ref{as:mon:rate}  given below.

 \begin{lemma}{\label{L,scheme,diff}}
Let the assumptions of Lemma \ref{higher MB} be satisfied. 
Then, 
   \begin{align*}
      \underset{i\in \{1,\ldots, N\}}{\sup} ~ \underset{ t \in [0, T]}{\sup}   \mathbb E|\hat{X}_{t} ^ {i,N,n}-\hat{X}_{k_n(t)} ^ {i,N,n}|^{p} \leq K n^{-p/2}
   \end{align*}
    for all $p \in [2, \frac{p_0}{q+1}]$  where $K$ is a constant that does not depend on $n,N \in \mathbb N$.
\end{lemma}
\begin{proof}
The proof follows by  Remark \ref{R_f} and  Lemma \ref{higher MB}. 
\end{proof}

To continue further we need to introduce the following assumptions. 
\begin{assumption} \label{as:mon:rate}
    There exist constants $L>0$ and $p_1>2$ such that 
    \begin{align*}
     \langle x-x' , b(t,x, \mu) - b(t,x', \mu') \rangle +(p_1-1)|\sigma(t,x, \mu)-\sigma (t,x', \mu')|^2  & \, \leq L \{|x-x'|^2+ W^{(2)}(\mu, \mu')^2\},
     \\
     \langle (x  - y)-  (x'  - y'), f(x,y) -  f(x',y')  \rangle  + 2 (p_1-1) \, |  g(x,y) - g(x',y')  |^2   & \, \leq L |(x-y)-(x'-y')|^2, 
     \\
     |b(t,x, \mu) - b(t',x, \mu)| + |\sigma(t,x, \mu) - \sigma(t',x, \mu)| & \, \leq L|t-t'|^{1/2},
\end{align*}
for all $t,t'\in [0,T]$, $x, x' \in \mathbb R^d$ and $\mu, \mu' \in \mathcal{P}_2(\mathbb R^d)$. 
\end{assumption}
To establish the rate of convergence (see Theorem \ref{thm:rate} below), we first establish the following lemmas.

\begin{lemma}{\label{l,f}}
 Let Assumptions \ref{assum_initial}, \ref{as:eu:b:sig}, \ref{as:eu:f:g}, \ref{as:anti-sys} be satisfied for some $p_0>2(q+1)$, $q>0$; Assumption \ref{as:mon:rate} holds for some $p_1 >2$.
 Then, 
  \begin{align*}
     \frac{1}{N^2} & \sum_{i, j=1}^N   \mathbb E\int_0^t |X_s^{i,N} - \hat{X}_s^{i,N,n}|^{p-2}  \{ \langle X_s^{i,N} - \hat{X}_s^{i,N,n},  f(X_s^{i,N}, X_s^{j,N}) - f^n (\hat{X}^{i,N,n}_{k_n(s)},\hat{X}^{j,N,n}_{k_n(s)})\rangle \nonumber \\
       & \quad \quad + (p-1)|g(X_s^{i,N}, X_s^{j,N})-g^n (\hat{X}^{i,N,n}_{k_n(s)},\hat{X}^{j,N,n}_{k_n(s)})|^2 \} ds
      \leq K \mathbb E\int_0 ^t |X_s^{i,N} - \hat{X}_s ^{i,N,n}|^p ds + Kn^{-\frac{p}{2}}
 \end{align*}
 for all   $t \in [0,T]$,  $p \in [2,  p_1)$ and $p \in [2, \frac{p_0}{3q +1}]$  where $K$ is a constant independent of $n,N \in \mathbb N$.  
   \end{lemma}
We note that the ``$3q+1$''-constant for the domain of $p$ arises from  Equation \eqref{S3} below. 
 \begin{proof}
First, one can write 
     \begin{align}
        \frac{1}{N^2}  \sum_{i, j=1}^N &  \mathbb E\int_0^t |X_s^{i,N} - \hat{X}_s^{i,N,n}|^{p-2}  \big\{\big \langle X_s^{i,N} - \hat{X}_s^{i,N,n},  f(X_s^{i,N},X_s^{j,N}) - f^n (\hat{X}^{i,N,n}_{k_n(s)},\hat{X}^{j,N,n}_{k_n(s)})  \big \rangle \nonumber 
        \\
       & \qquad \qquad + (p-1)|g(X_s^{i,N}, X_s^{j,N})-g^n (\hat{X}^{i,N,n}_{k_n(s)},\hat{X}^{j,N,n}_{k_n(s)}))|^2 \big \} ds \nonumber
       \\
        \leq \, & \frac{1}{N^2} \sum_{i, j=1}^N  \mathbb E\int_0 ^t |X_s^{i,N}  - \hat{X}_s^{i,N,n}|^{p-2} \{ \langle X_s^{i,N}  - \hat{X}_s ^{i,N,n}, f(X_s^{i,N},X_s^{j,N}) -  f(\hat{X}_s^{i,N,n},\hat{X}_s^{j,N,n}) \rangle  \nonumber
       \\
       & \quad \quad + (p_1-1)|g(X_s^{i,N}, X_s^{j,N})-g(\hat{X}_s^{i,N,n},\hat{X}_s^{j,N,n})|^2 \} ds \nonumber 
       \\
       &  + \frac{1}{N^2} \sum_{i, j=1}^N  \mathbb E \int_0 ^t |X_s^{i,N}  - \hat{X}_s^{i,N,n}|^{p-2} \{ \langle X_s^{i,N}  - \hat{X}_s ^{i,N,n}, f(\hat{X}_s^{i,N,n},\hat{X}_s^{j,N,n})- f(\hat{X}_{k_n(s)}^{i,N,n},\hat{X}_{k_n(s)}^{j,N,n}) \rangle \nonumber
       \\
       & \qquad \qquad + K |g(\hat{X}_s^{i,N,n},\hat{X}_s^{j,N,n})-g(\hat{X}_{k_n(s)}^{i,N,n},\hat{X}_{k_n(s)}^{j,N,n})|^2\} ds \nonumber
       \\
        &  + \frac{1}{N^2} \sum_{i, j=1}^N  \mathbb E\int_0 ^t|X_s^{i,N}  - \hat{X}_s^{i,N,n}|^{p-2} \{ \langle X_s^{i,N}  - \hat{X}_s ^{i,N,n}, f(\hat{X}_{k_n(s)}^{i,N,n},\hat{X}_{k_n(s)}^{j,N,n})-f^n(\hat{X}_{k_n(s)}^{i,N,n},\hat{X}_{k_n(s)}^{j,N,n}) \rangle \nonumber
        \\
        & \qquad \qquad + K |g(\hat{X}_{k_n(s)}^{i,N,n},\hat{X}_{k_n(s)}^{j,N,n})-g^n(\hat{X}_{k_n(s)}^{i,N,n},\hat{X}_{k_n(s)}^{j,N,n})|^2  \} ds\nonumber\\
         =: \, &  S_1 + S_2 + S_3 \label{S1+S2+S3}
     \end{align}
     for all $t \in [0,T]$. 
    Following similar steps as in Lemma \ref{L_rate_x^i} under Assumption \ref{as:f:g:PoC-2} and \ref{as:mon:rate}, $S_1$ can be estimated by
\begin{align}
    S_1  := & \frac{1}{N^2} \sum_{i, j=1}^N  \mathbb E\int_0 ^t |X_s^{i,N}  - \hat{X}_s^{i,N,n}|^{p-2} \{ \langle X_s^{i,N}  - \hat{X}_s ^{i,N,n},  f(X_s^{i,N},X_s^{j,N}) -  f(\hat{X}_s^{i,N,n},\hat{X}_s^{j,N,n}) \rangle  \nonumber
    \\
    & \qquad \qquad + (p_1-1)|g(X_s^{i,N},X_s^{j,N})-g(\hat{X}_s^{i,N,n},\hat{X}_s^{j,N,n})|^2 \} ds \nonumber 
 \\
 & \leq  \frac{K}{N} \sum_{i=1}^N \mathbb E \int_0 ^t |X_s^{i,N} - \hat{X}_s ^{i,N,n}|^p ds \label{S1}
    \end{align}
    for all $t \in [0,T]$. 
    For the estimation of $S_2$, one uses the Cauchy-Schwarz inequality, Young\textquotesingle{s} inequality, Assumption \ref{as:eu:b:sig}, Remark \ref{rem:fg:grw:eu}, Lemmas \ref{higher MB} and \ref{L,scheme,diff} to obtain the following estimates, 
\begin{align}
   S_2  := & \frac{1}{N^2} \sum_{i,j=1}^N  \mathbb E \int_0 ^t |X_s^{i,N}  - \hat{X}_s^{i,N,n}|^{p-2} \big \{ \langle X_s^{i,N}  - \hat{X}_s ^{i,N,n}, f(\hat{X}_s^{i,N,n},\hat{X}_s^{j,N,n})- f(\hat{X}_{k_n(s)}^{i,N,n},\hat{X}_{k_n(s)}^{j,N,n}) \rangle \nonumber
   \\
   &  \qquad \qquad + K |g(\hat{X}_s^{i,N,n},\hat{X}_s^{j,N,n})-g(\hat{X}_{k_n(s)}^{i,N,n},\hat{X}_{k_n(s)}^{j,N,n})|^2 \big\} ds \nonumber
   \\
    \leq &  \frac{K}{N^2} \sum_{i,j=1}^N \mathbb E \int_0 ^t |X_s^{i,N}  - \hat{X}_s^{i,N,n}|^{p-1} |f(\hat{X}_s^{i,N,n},\hat{X}_s^{j,N,n})- f(\hat{X}_{k_n(s)}^{i,N,n},\hat{X}_{k_n(s)}^{j,N,n}| ds \nonumber
   \\
   & \qquad \qquad +  \frac{K}{N^2} \sum_{i, j=1}^N  \mathbb E \int_0 ^t |X_s^{i,N}  - \hat{X}_s^{i,N,n}|^{p-2} |g(\hat{X}_s^{i,N,n},\hat{X}_s^{j,N,n})- g(\hat{X}_{k_n(s)}^{i,N,n},\hat{X}_{k_n(s)}^{j,N,n}|^2 ds \nonumber
   \\
   \leq &  \,\frac{K}{N} \sum_{i=1}^N \mathbb E \int_0 ^t |X_s^{i,N} - \hat{X}_s ^{i,N,n}|^p ds  +  \frac{K}{N^2} \sum_{i, j=1}^N  \int_0 ^t \bigl \{\mathbb E\bigl[1 + |\hat{X}_s^{i,N,n}-\hat{X}_s^{j,N,n}|^{2pq} + |\hat{X}_{k_n(s)}^{i,N,n}-\hat{X}_{k_n(s)}^{j,N,n}|^{2pq} \bigl] \bigl \}^{1/2} \nonumber
   \\
   & \qquad \qquad  \bigl \{\mathbb E \bigl[|\hat{X}_s^{i,N,n}- \hat{X}_{k_n(s)}^{i,N,n}|^{2p} +|\hat{X}_s^{j,N,n}- \hat{X}_{k_n(s)}^{j,N,n}|^{2p} \bigl] \bigl\}^{1/2} ds \nonumber
   \\
   \leq & \frac{K}{N} \sum_{i=1}^N \mathbb E \int_0 ^t |X_s^{i,N} - \hat{X}_s ^{i,N,n}|^p ds  + Kn^{-\frac{p}{2}}\label{S2}
\end{align}
for all $t \in [0,T]$. 

 For $S_3$, using the Cauchy-Schwarz inequality, Young\textquotesingle{s} inequality, Equation \eqref{taming-FiniteT}, Remark \ref{rem:fg:grw:eu}, Lemma  \ref{higher MB},  one gets
\begin{align}
       S_3 & =  \frac{1}{N^2} \sum_{i, j=1}^N \mathbb E\int_0 ^t|X_s^{i,N}  - X_s^{i,N,n}|^{p-2} \{ \langle X_s^{i,N}  - \hat{X}_s ^{i,N,n}, f(\hat{X}_{k_n(s)}^{i,N,n},\hat{X}_{k_n(s)}^{j,N,n})-f^n(\hat{X}_{k_n(s)}^{i,N,n},\hat{X}_{k_n(s)}^{j,N,n}) \rangle \nonumber
       \\
        & \qquad \qquad + K |g(\hat{X}_{k_n(s)}^{i,N,n},\hat{X}_{k_n(s)}^{j,N,n})-g^n(\hat{X}_{k_n(s)}^{i,N,n},\hat{X}_{k_n(s)}^{j,N,n})|\} ds\nonumber
        \\
        & \leq \frac{K}{N^2} \sum_{i, j=1}^N  \mathbb E\int_0 ^t |X_s^{i,N}  - X_s^{i,N,n}|^{p-1} |f(\hat{X}_{k_n(s)}^{i,N,n},\hat{X}_{k_n(s)}^{j,N,n})-f^n(\hat{X}_{k_n(s)}^{i,N,n},\hat{X}_{k_n(s)}^{j,N,n})| ds \nonumber
        \\
        & \qquad \qquad + \frac{K}{N^2} \sum_{i, j=1}^N  \mathbb E\int_0 ^t |X_s^{i,N}  - X_s^{i,N,n}|^{p-2} |g(\hat{X}_{k_n(s)}^{i,N,n},\hat{X}_{k_n(s)}^{j,N,n})-g^n(\hat{X}_{k_n(s)}^{i,N,n},\hat{X}_{k_n(s)}^{j,N,n})|^2 ds\nonumber 
        \\
        & \leq \frac{K}{N} \sum_{i=1}^N  \mathbb E\int_0 ^t |X_s^{i,N}  - X_s^{i,N,n}|^{p}ds + \frac{Kn^{-\frac{p}{2}}}{N^2} \sum_{i, j=1}^N  \mathbb E\int_0 ^t \frac{1+|\hat{X}_{k_n(s)}^{i,N,n}-\hat{X}_{k_n(s)}^{j,N,n})|^{(3q+1)p}}{(1+n^{-\frac{1}{2}}|\hat{X}_{k_n(s)}^{i,N,n}-\hat{X}_{k_n(s)}^{j,N,n}|^{2q})^{p}}  ds \nonumber
        \\ 
        & \leq \frac{K}{N} \sum_{i=1}^N \mathbb E\int_0 ^t | X_s^{i,N}  - \hat{X}_s ^{i,N,n}|^2 ds + Kn^{-\frac{p}{2}} \label{S3},
\end{align}
for all $t \in [0,T]$.  
Thus, plugging the estimates from \eqref{S1}, \eqref{S2} and \eqref{S3} into \eqref{S1+S2+S3} yields the desired result.
 \end{proof}

The rate of strong convergence of the tamed Euler scheme \eqref{scheme} of IPS \eqref{IPS} in $L^p$-norm is shown below. 
\begin{theorem}\label{thm:rate}  

{Let Assumptions \ref{assum_initial}, \ref{as:eu:b:sig}, \ref{as:eu:f:g}, \ref{as:anti-sys},\ref{as:b:poly} and \ref{as:mon:rate} be satisfied for some $p_0>2(q+1)$ with $q>0$; Assumption \ref{as:mon:rate} holds with a $p_1>2$. Then, the tamed Euler scheme \eqref{scheme} converges to IPS \eqref{IPS} in $L^p$-sense with rate $1/2$, \textit{i.e},
\begin{align*}
       \underset{i\in \{1,\ldots,N\}}{\sup} ~ \underset{t \in [0,T]}{\sup}  \mathbb E|X_t^{i,N} - \hat{X}_t^{i,N,n}|^p \leq Kn^{-p/2}
    \end{align*}
all $p \in [2, p_1)$ and $p \in [2, \frac{p_0}{3q +1}]$ where $K$ is a positive constant that does not depend on $n,N \in \mathbb N$.  }

\end{theorem}
\begin{proof}
Applying It\^o's formula and Jensen\textquotesingle{s} inequality, 
    \begin{align*}
          \mathbb E | & X_t^{i,N} - \hat{X}_t^{i,N,n}|^p  
         \\
         & \leq  \, p   \mathbb E \int_0^t |X_s^{i,N} - \hat{X}_s^{i,N,n}|^{p-2} \langle X_s^{i,N} - \hat{X}_s^{i,N,n}, b(s,X_s^{i,N},\mu_s^{X,N}) - b^n(k_n(s),\hat{X}_{k_n(s)}^{i,N,n},\hat{\mu}_{k_n(s)}^{X,N,n})\rangle ds 
         \\
         &  + p \mathbb E\int_0^t |X_s^{i,N} - \hat{X}_s^{i,N,n}|^{p-2}  \Big\langle X_s^{i,N} - \hat{X}_s^{i,N,n},\frac{1}{N} \sum_{j=1}^N  \big\{\ f(X_s^{i,N}, X_s^{j,N}) - f^n (\hat{X}^{i,N,n}_{k_n(s)},\hat{X}^{j,N,n}_{k_n(s)}\big) \big\}  \Big\rangle ds 
         \\
         & + p(p-1) \mathbb E \int_0^t |X_s^{i,N} - \hat{X}_s^{i,N,n}|^{p-2} \big|{\sigma} (s,X_s^{i,N},\mu_s^{X,N}) - {\sigma} ^n(k_n(s),\hat{X}_{k_n(s)}^{i,N,n},\hat{\mu}_{k_n(s)}^{X,N,n}) \big|^2ds
         \\
         &  +  p(p-1) \mathbb E \int_0^t |X_s^{i,N} - \hat{X}_s^{i,N,n}|^{p-2} \frac{1}{N} \sum_{j=1}^N \big|g(X_s^{i,N}, X_s^{j,N})-g^n (\hat{X}^{i,N,n}_{k_n(s)},\hat{X}^{j,N,n}_{k_n(s)})\big|^2 ds
    \end{align*}
    for all $t \in [0,T]$. 
    Now, averaging over all particles and then by using Young\textquotesingle{s} inequality,  one obtains, 
    \begin{align*}
       & \frac{1}{N} \sum_{i=1}^N \mathbb E  |X_t^{i,N}  - \hat{X}_t^{i,N,n}|^p
       \\
       & \leq  \frac{p}{N} \sum_{i=1}^N \mathbb E\int_0^t  |X_s^{i,N} - \hat{X}_s^{i,N,n}|^{p-2} \big\{ \langle X_s^{i,N} - \hat{X}_s^{i,N,n}, b(s,X_s^{i,N},\mu^{X,N}_s) - b(s,\hat{X}_{s}^{i,N,n},\hat{\mu}_{s}^{X,N,n}) \rangle 
       \\
        & \qquad \qquad \qquad + (p_1-1)|\sigma (s,X_s^{i,N},\mu^{X,N}) -\sigma(s,\hat{X}_{s}^{i,N,n},\hat{\mu}_{s}^{X,N,n}|^2 \big\}  ds 
        \\
        & \quad +  \frac{p}{N} \sum_{i=1}^N \mathbb E \int_0^t |X_s^{i,N} - \hat{X}_s^{i,N,n}|^{p-2} \langle X_s^{i,N} - \hat{X}_s^{i,N,n}, b(s,\hat{X}_s^{i,N,n},\hat{\mu}_s^{X,N,n}) - b^n (k_n(s),\hat{X}_{k_n(s)}^{i,N,n},\hat{\mu}_{k_n(s)}^{X,N,n}) \rangle ds 
        \\
        & \quad +  \frac{K}{N} \sum_{i=1}^N \mathbb E \int_0^t |X_s^{i,N} - \hat{X}_s^{i,N,n}|^{p-2} |\sigma (s,\hat{X}_s^{i,N,n},\hat{\mu}_s^{{X,N,n}}) -\sigma^n(k_n(s),\hat{X}_{k_n(s)}^{i,N,n},\hat{\mu}_{k_n(s)}^{X,N,n})|^2 ds 
        \\
       & \quad + \frac{p}{N^2} \sum_{i, j=1}^N  \mathbb E\int_0^t |X_s^{i,N} - \hat{X}_s^{i,N,n}|^{p-2}  \{ \langle X_s^{i,N} - \hat{X}_s^{i,N,n},  f(X_s^{i,N}, X_s^{j,N}) - f^n (\hat{X}^{i,N,n}_{k_n(s)},\hat{X}^{j,N,n}_{k_n(s)})\rangle 
       \\
       & \qquad \qquad + (p-1)|g(X_s^{i,N},X_s^{j,N})-g^n (\hat{X}^{i,N,n}_{k_n(s)},\hat{X}^{j,N,n}_{k_n(s)}))|^2 \} ds := B_1 + B_2 + B_3 + B_4
    \end{align*}
    for all $t \in [0,T]$. 
By Assumption \ref{as:mon:rate}, $B_1$ is estimated by, 
\begin{align}
    B_1:= &  \frac{p}{N} \sum_{i=1}^N \mathbb E\int_0^t  |X_s^{i,N} - \hat{X}_s^{i,N,n}|^{p-2} \big\{ \big\langle X_s^{i,N} - \hat{X}_s^{i,N,n}, b(s,X_s^{i,N},\mu^{X,N}_s) - b(s,\hat{X}_{s}^{i,N,n},\hat{\mu}_{s}^{X,N,n}) \big\rangle  \notag
       \\
        & \qquad \qquad + (p_1-1)|\sigma (s,X_s^{i,N},\mu^{X,N}) -\sigma(s,\hat{X}_{s}^{i,N,n},\hat{\mu}_{s}^{X,N,n}|^2 \big\}  ds  \nonumber
        \\
        \leq & \frac{K}{N} \sum_{i=1}^N \mathbb E\int_0^t  |X_s^{i,N} - \hat{X}_s^{i,N,n}|^{p}ds \label{es:B1}
\end{align}
for all $t \in [0,T]$. 
Also, Cauchy-Schwarz inequality and Young\textquotesingle{s} inequality yields the following estimate for $B_2$, 
\begin{align}
    B_2: = &  \frac{p}{N} \sum_{i=1}^N \mathbb E \int_0^t |X_s^{i,N} - \hat{X}_s^{i,N,n}|^{p-2} \langle X_s^{i,N} - \hat{X}_s^{i,N,n}, b(s,\hat{X}_s^{i,N,n},\hat{\mu}_s^{X,N,n}) - b^n (k_n(s),\hat{X}_{k_n(s)}^{i,N,n},\hat{\mu}_{k_n(s)}^{X,N,n}) \rangle ds \nonumber
    \\
    \leq & \frac{K}{N} \sum_{i=1}^N \mathbb E\int_0^t  |X_s^{i,N} - \hat{X}_s^{i,N,n}|^{p}ds +  \frac{K}{N} \sum_{i=1}^N \mathbb E  \int_0^t  |b (s,\hat{X}_{s} ^ {i,N,n}, \hat{\mu}_s ^{X,N,n})- b(k_n(s),\hat{X}_{k_n(s)} ^ {i,N,n}, \hat{\mu}_{k_n(s)} ^{X,N,n})|^p ds \nonumber
    \\
   & \qquad  + \frac{K}{N} \sum_{i=1}^N \mathbb E \int_0^t |b(k_n(s),\hat{X}_{k_n(s)} ^ {i,N,n},  \hat{\mu}_{k_n(s)} ^{X,N,n})- b^n (k_n(s),\hat{X}_{k_n(s)} ^ {i,N,n}, \hat{\mu}_{k_n(s)} ^{X,N,n})|^p ds \notag 
   \\
   =: & \, \frac{K}{N} \sum_{i=1}^N \mathbb E\int_0^t  |X_s^{i,N} - \hat{X}_s^{i,N,n}|^{p}ds + B_{2.1}+ B_{2.2} \label{es:B21+B22}
\end{align}
 for all $t \in [0,T]$. 
 For $B_{2.1}$, by Assumption \ref{as:b:poly}, \ref{as:mon:rate}, Cauchy-Schwarz inequality, Lemma \ref{higher MB} and Lemma \ref{L,scheme,diff}, one has
\begin{align}
   B_{2.1} :=&     \frac{K}{N} \sum_{i=1}^N \mathbb E  \int_0^t  |b (s,\hat{X}_{s} ^ {i,N,n}, \hat{\mu}_s ^{X,N,n})- b(k_n(s),\hat{X}_{k_n(s)} ^ {i,N,n}, \hat{\mu}_{k_n(s)} ^{X,N,n})|^p ds \nonumber  
     \\
     \leq &  \frac{K}{N} \sum_{i=1}^N \mathbb E \int_0^t(1 +|\hat{X}_{s} ^ {i,N,n}| + |\hat{X}_{k_n(s)} ^ {i,N,n}| )^{pq}|\hat{X}_{s} ^ {i,N}-\hat{X}_{k_n(s)} ^ {i,N,n}|^p ds + K \mathbb E \int_0^t W^{(2)} ( \hat{\mu}_s ^{X,N,n},\hat{\mu}_{k_n(s)} ^{X,N,n})^p ds + Kn^{-\frac{p}{2}}   \nonumber 
    \\
   \leq & \frac{K}{N} \sum_{i=1}^N \int_0 ^t \Big\{\mathbb E\Bigl[(1 + |\hat{X}_{s} ^ {i,N,n}|^{2pq} + |\hat{X}_{k_n(s)} ^ {i,N,n}|^{2pq})\Bigl] \mathbb E\Bigl[|\hat{X}_{s} ^ {i,N,n}-\hat{X}_{k_n(s)} ^ {i,N,n}|^{2p} \Bigl] \Big\}^{1/2} ds \nonumber 
  \\
   & \qquad \qquad +   \frac{K}{N}\sum_{j=1}^N \mathbb E \int_0^t |\hat{X}_{s} ^ {j,N,n}-\hat{X}_{k_n(s)} ^ {j,N,n}|^p ds + Kn^{-\frac{p}{2}}  \leq  Kn^{-\frac{p}{2}} \label{es:B21}
\end{align}
for all $t \in [0, T]$. 
Further, one can estimate $B_{2.2}$ by applying Equation \eqref{taming-FiniteT}, Remark \ref{rem:b:sig:gwt:tem} and Lemma \ref{higher MB}, 
\begin{align}
    B_{2.2}:= & \frac{K}{N} \sum_{i=1}^N \mathbb E \int_0^t |b(k_n(s),\hat{X}_{k_n(s)} ^ {i,N,n},  \hat{\mu}_{k_n(s)} ^{X,N,n})- b^n (k_n(s),\hat{X}_{k_n(s)} ^ {i,N,n}, \hat{\mu}_{k_n(s)} ^{X,N,n})|^p ds \nonumber
   \\
     \leq & \frac{K}{N} \sum_{i=1}^N n^{-\frac{p}{2}} \mathbb E \int_0^t\frac{|\hat{X}_{k_n(s)} ^{i,N,n}|^{2pq}}{(1+n^{-\frac{1}{2}} |\hat{X}_{k_n(s)} ^{i,N,n}|^{2q})^p }|b(k_n(s),\hat{X}_{k_n(s)} ^ {i,N,n}, \hat{\mu}_{k_n(s)} ^{X,N,n})|^p ds\nonumber 
    \\
     \leq & \frac{K}{N} \sum_{i=1}^N n^{-\frac{p}{2}} \mathbb E \int_0^t|\hat{X}_{k_n(s)} ^{i,N,n}|^{2pq} \Bigl(1+|\hat{X}_{k_n(s)} ^{i,N,n}|^{q+1} + \Bigl(\frac{1}{N} \sum_{j=1}^N |\hat{X}_{k_n(s)} ^{j,N,n}|^2 \Bigl )^{\frac{1}{2}}  \Bigl )^p ds \nonumber
    \\
       \leq & \frac{K}{N} \sum_{i=1}^N n^{-\frac{p}{2}} \int_0^t \Big\{\mathbb E\Bigl[|\hat{X}_{k_n(s)} ^{i,N,n}|^{4pq} \Bigl] \mathbb E \Bigl[1+|\hat{X}_{k_n(s)} ^{i,N,n}|^{2p(q+1)}+\frac{1}{N} \sum_{j=1}^N |\hat{X}_{k_n(s)} ^{j,N,n}|^{2p} \Bigl]\Big\}^{1/2}   \leq Kn^{-\frac{p}{2}} \label{es:B22}
\end{align}
for all $t \in [0,T]$.
Thus, substituting estimates from Equations \eqref{es:B21} and \eqref{es:B22} into equation \eqref{es:B21+B22} yields,  
\begin{align}  \label{es:B2}
    B_2 \leq \frac{K}{N} \sum_{i=1}^N \mathbb E\int_0^t  |X_s^{i,N} - \hat{X}_s^{i,N,n}|^{p}ds +K n^{-\frac{p}{2}} 
\end{align}
for all $t \in [0,T]$.
The term $B_3$ can be estimated by adopting similar reasoning as have been used in $B_2$, which yields
\begin{align}
   B_3:= &  \frac{K}{N} \sum_{i=1}^N \mathbb E \int_0^t |X_s^{i,N} - \hat{X}_s^{i,N,n}|^{p-2} |\sigma (s,\hat{X}_s^{i,N,n},\hat{\mu}_s^{{X,N,n}}) -\sigma^n(k_n(s),\hat{X}_{k_n(s)}^{i,N,n},\hat{\mu}_{k_n(s)}^{X,N,n})|^2 ds  \nonumber
   \\
  \leq &  \frac{K}{N} \sum_{i=1}^N \mathbb E\int_0^t  |X_s^{i,N} - \hat{X}_s^{i,N,n}|^{p}ds +K n^{-\frac{p}{2}} \label{es:B3}
\end{align}
for all $t \in [0,T]$. 
Moreover, by Lemma \ref{l,f}, $B_4$ is estimated as
\begin{align}
   B_4:= & \frac{1}{N} \sum_{i=1}^N \mathbb E  \int_0^t |X_s^{i,N} - \hat{X}_s^{i,N,n}|^{p-2} \langle X_s^{i,N} - \hat{X}_s^{i,N,n},\frac{1}{N} \sum_{j=1}^N  (f(X_s^{i,N}, X_s^{j,N}) - f^n (\hat{X}^{i,N,n}_{k_n(s)},\hat{X}^{j,N,n}_{k_n(s)}))  \rangle \nonumber
   \\
   & \qquad \qquad + (p-1)|g(X_s^{i,N},X_s^{j,N})-g^n (\hat{X}^{i,N,n}_{k_n(s)},\hat{X}^{j,N,n}_{k_n(s)}))|^2 \} ds \notag
       \\
    \leq &  \frac{K}{N} \sum_{i=1}^N \mathbb E\int_0^t  |X_s^{i,N} - \hat{X}_s^{i,N,n}|^{p}ds +K n^{-\frac{p}{2}} \label{es:B4}
\end{align}
for all $t \in [0,T]$. 
Thus, collecting the estimates from \eqref{es:B1},  \eqref{es:B3} and \eqref{es:B4}, one gets
 \begin{align*}
       \frac{1}{N} \sum_{i=1}^N & \mathbb E  |X_t^{i,N}  - \hat{X}_t^{i,N,n}|^p   \leq \frac{K}{N} \sum_{i=1}^N \mathbb E\int_0^t  |X_s^{i,N} - \hat{X}_s^{i,N,n}|^{p}ds +K n^{-\frac{p}{2}}
\end{align*}
for all $t \in [0,T]$.
The proof is concluded using the Gr\"onwall\textquotesingle{s} lemma and the exchangeability of the particles. 
\end{proof}
\section{Ergodicity}
\label{sec:Ergodicity}

In this section, we examine the long-time behavior of the solution of MV-SDE \eqref{Mckean}, the IPS \eqref{IPS} and the tamed Euler scheme \eqref{eq:TamedEulerforErgodicity} (defined below), i.e.,  we demonstrate the existence of a unique invariant measure in each of these cases.
  Moreover, we show that the invariant measure of the tamed Euler scheme converges to the invariant measure of IPS \eqref{IPS} in the Wasserstein $W^{(2)}$-distance when the mesh size tends to zero. 

The results in Sections \ref{subsec:Ergod-MV-SDE} and \ref{subsec:Ergod-IPS}, for the MV-SDE \eqref{Mckean} and IPS \eqref{IPS} respectively, are inspired by the methodology from \cite{chen2025} and at points the presentation is streamlined. The results in Section \ref{sec:ergodicity tamed Euler scheme} establishing the ergodicity of the tamed Euler scheme, are to the best of our knowledge novel to the literature (even to less non-superlinear growth settings than those studied in this work) and the methodology combines \cite{chen2025}, \cite{bao2024geometric} and to a lesser extent \cite{yuanping2024explicit}. 
Lastly, in terms of presentation of our results, the previous sections hold only for $T<\infty$, thus part of the work here involves modifying the assumptions in order to allow $T\to \infty$. 
\smallskip

\textit{For presentation purposes}, we write this section in a self contained fashion with its own assumptions that involve a unavoidable repetition of the earlier assumptions but with sharper constant dependencies. 
For instance, Assumptions \ref{as:X0:erg}, \ref{as:Er Sol}  and \ref{as:Er Sol mon} are a sharpened, and slightly stronger, version of the well-posedness and PoC rate assumptions in Sections \ref{sec:main:welpos} and \ref{sec:IPS-and-PoC}. The ergodicity results for the taming scheme are presented in Section \ref{sec:ergodicity tamed Euler scheme} with their own version of the assumptions.

\subsection{Ergodicity of the McKean--Vlasov SDE}
\label{subsec:Ergod-MV-SDE}
To study the ergodic property of MV-SDE \eqref{Mckean} we take inspiration from \cite{Wang2018-DDSDE-Landau-type} and \cite{chen2025}, and   make the following assumptions. 
Just as in Sections \ref{sec:main:welpos} and \ref{sec:IPS-and-PoC}, take $q>0$, $\ell>2(q+1)$ and $p_0 \geq \ell$; note that $\ell>2$. In contrast to Sections \ref{sec:main:welpos}  and \ref{sec:IPS-and-PoC}, our arguments require a new running parameter $\ell$ relating (as $p_0$) to the integrability of the initial distribution and whose role will be made clear below. 
To make the section more self-contained, we reproduce Assumption \ref{assum_initial} as Assumption \ref{as:X0:erg}. 
\begin{assumption}
\label{as:X0:erg}
$\mathbb E |X_0|^{p_0} < \infty$. 
\end{assumption}
\begin{assumption}\label{as:Er Sol}
The function $f$ satisfies $f(x,y)= -f(y,x)$ for all $x, y \in \mathbb R^d$. 
Also, there exist constants $\hat{L}_{b\sigma}^{(1)},\hat{L}_{fg}^{(1)}, \hat{L}_{f}^{(1)} \in \mathbb R $ and $\hat{L}_{b\sigma} ,\hat{L}_{b\sigma}^{(2)}, \hat{L}_{fg} >0$ such that
    \begin{align*}
        \langle x,b(t,x,\mu)\rangle  + (p_0-1)|\sigma(t,x,\mu)|^2 & \leq \hat{L}_{b\sigma} +\hat{L}_{b\sigma}^{(1)}|x|^2 + \hat{L}_{b\sigma}^{(2)} W^{(2)} (\mu,\delta_0)^2,   
        \\
        (|x|^{p_0-2} - |y|^{p_0-2}) \langle x+y,f(x,y) \rangle & \leq \hat{L}_{f}^{(1)} (|x|^{p_0}+|y|^{p_0}),
        \\
        \langle x -y,f(x,y) \rangle + 2 (p_0-1)|g(x,y)|^2 & \leq \hat{L}_{fg} + \hat{L}_{fg}^{(1)} |x-y|^2,
    \end{align*}
    for all $t \in [0,\infty),x,y \in \mathbb R^d$ and $\mu \in \mathcal{P}_2(\mathbb{R}^2)$.
\end{assumption}
\begin{assumption}{\label{as:Er Sol mon}}
    There exist constants $L_{b\sigma}^{(1)}$, $L_{fg}^{(1)} \in \mathbb R$  and $L_{b\sigma}^{(2)} >0$ such that  
    \begin{align*}
       \langle x-x',b(t,x,\mu)-b(t,x',\mu')\rangle  + |\sigma(t,x,\mu)-\sigma(t,x',\mu')|^2  \leq &\ L_{b\sigma}^{(1)}|x-x'|^2 + L_{b\sigma}^{(2)} W^{(2)} (\mu,\mu')^2,  
       \\
         \langle (x-x') -(y-y'),f(x,y) - f(x',y') \rangle + {(p_0-1)} |g(x,y)-g(x',y')|^2 \leq &\  L_{fg}^{(1)} |(x-x')-(y-y')|^2,
\\
         |f(x,y) -  f(x',y')|^2  \leq  \, L_{f}^{(1)} \big(1+|x-y|^{2q} & + |x'-y'|^{2q} \big)|(x-x')-(y-y')|^2,
\end{align*} 
     for all $t \in [0,\infty)$, $x,x',y,y' \in \mathbb R^d$ and $\mu,\mu' \in \mathcal{P}_2(\mathbb R^d)$.
     The maps $b$ and $\sigma$ are jointly continuous. 
\end{assumption}
The following theorem establishes the ergodic properties of MV-SDE \eqref{Mckean} under the set of assumptions mentioned above. 
For this purpose, the solution of  MV--SDE \eqref{Mckean} with initial value $X_0$ is denoted by $X:=\{X_t\}_{t \geq 0}$ and the flow of its marginal laws by $\{P_{t} \mu_0\}_{t\geq 0} := \{\mu_t^{X}\}_{t\geq 0}$.   
Here,  $\mu_0:=\mu_0^{X}$ is the law of $X_0$.  
Moreover, if the initial law is $\nu_0$, then $Y:=\{Y_t\}_{t \geq 0}$ and $\{P_{t} \nu_0\}_{t\geq 0} := \{\mu_t^{Y}\}_{t\geq 0}$ denote the solution of MV--SDE \eqref{Mckean} and the flow of its marginal laws, respectively. 
\begin{theorem}{\label{Th:Erog sol}}
Take $q>0$, $\ell>2(q+1)$ and some $p_0 \geq \ell$ with $p_0\geq 5$. 
Let Assumptions \ref{as:X0:erg},  \ref{as:Er Sol} and \ref{as:Er Sol mon} be satisfied. 
Then, the following statements are true. 
    \newline
    \textnormal{(A)} Set $\rho_1 := \ell\big\{\hat{L}_{b\sigma}^{(1)}+\hat{L}_{b\sigma}^{(2)}+ 2\hat{L}_{fg}^{(1)+} +\frac{\hat{L}_{f}^{(1)}}{2}+\frac{3(\ell-2)}{2\ell}\big\}$.  If $\mu_0 \in \mathcal{P}_\ell(\mathbb R^d)$,  then for every $t \in [0,\infty)$,   
    \begin{align*}
        W^{(\ell)} (P_{t} \mu_0, \delta_0)^\ell \leq e^{{\rho_1}t} W^{(\ell)} (\mu_0,\delta_0)^\ell + \frac{2 (\hat{L}_{b\sigma})^{\ell/2}(e^{\rho_1 t} -1)}{\rho_1} \mathbbm{1}_{{\{\rho_1} \neq 0\}} + \bigl(2 (\hat{L}_{b\sigma})^{\ell/2} + (\hat{L}_{fg})^{\ell/2} \bigl) t  \mathbbm{1}_{{\{\rho_1}=0\}}.
    \end{align*}
    Moreover, if $\rho_1<0$, then $\displaystyle \sup_{t \in [0, \infty)}\mathbb E|X_t|^{\ell} \leq K$ for some constant $K>0$.  
    \newline 
   \textnormal{(B)}  
   Set $\rho_2 := 2L_{b\sigma}^{(1)} +4L_{fg}^{(1)+}+4L^{(2)}_{b \sigma}+1$.
   Then,  for every  $\mu_0,\nu_0 \in \mathcal{P}_\ell(\mathbb R^d)$ and $t \in [0, \infty)$, the following statements are true:
        \begin{align*}
             W^{(2)}(P_{t} \mu_0, P_{t} \nu_0)^2  &\leq 6 e^{\rho_2 t} \bigl(W^{(2)}(\mu_0,\delta_0)^2 + W^{(2)}(\delta,\nu_0)^2 \bigl).
        \end{align*}
\noindent
   \textnormal{(C)} Let $b$ and $\sigma$ be independent of time and choose $\rho_1 <0$ and $\rho_2 <0$ as defined above in statements    \textnormal{(A)} and    \textnormal{(B)}, respectively, under the additional restriction that $p_0,\ell$ satisfy: $2\ell -2 > 2(q+1)$ and $p_0 \geq  2\ell -2$.
   Then there exists a unique invariant probability measure $\bar \mu \in \mathcal{P}_{\ell}(\mathbb R^d)$ to MV--SDE \eqref{Mckean} in the following sense: for any $r\in[2,\ell]$, we have 
    \begin{align*}
         W^{(r)} (P_{t} \bar \mu, \bar \mu) =0 \quad \mbox{for all } t \in [0, \infty) \qquad \mbox{ and }\qquad 
         \textrm{for all } \hat{\nu} \in \mathcal{P}_{2\ell-2}(\mathbb R^d) \quad    \lim_{t \rightarrow \infty} W^{(r)} (P _{t} \hat{\nu}, \bar \mu) =0. 
        \end{align*}    
\end{theorem}
\begin{proof}
(A) 
Using It\^o's formula, Assumption \ref{as:Er Sol},  Lemma \ref{lem:sym-gwt:erg} and Young'{s} inequality, one has
\begin{align*}
   & e^{-\rho_1 t}  \mathbb E |X_t|^\ell 
   \\
   & \leq \,  \mathbb E |X_0|^{\ell} - \rho_1 \mathbb E \int_{0}^t e^{-\rho_1 s} |{X}_s|^{\ell} ds   
   + \ell \, \mathbb E \int_{0}^t e^{-\rho_1 s} |{X}_s|^{\ell-2}  \bigl\{ \big\langle {X}_s, b (s,{X}_s,{\mu}_s^X) \big\rangle + (\ell-1) \big|\sigma (s,{X}_s,{\mu}_s^X) \big|^2 \bigl\} ds
           \\
  & \quad +  \ell \, \mathbb E \int_{0}^t e^{-\rho_1 s} |{X}_s|^{\ell-2} \Big\{ \Big\langle {X}_s, \int_{\mathbb R^d} f ({X}_s,y) \mu_s^X(dy) \Big\rangle + (\ell-1) \int_{\mathbb R^d} \bigl|g ({X}_s,y) \big|^2 \mu_s^X(dy)\Big\}ds
\\
 & \leq  \mathbb E |X_0|^{\ell} - \rho_1 \mathbb E \int_{0}^t e^{-\rho_1 s} |{X}_s|^{\ell} ds +  \ell \,  \mathbb E \int_{0}^t e^{-\rho_1 s} |{X}_s|^{\ell-2} \big\{\hat{L}_{b\sigma} + \hat{L}_{b\sigma}^{(1)}|{X}_s|^2 + \hat{L}_{b\sigma}^{(2)} W^{(2)} (\mu_s^X,\delta_0)^2 \big\}
     \\
     & \quad + \ell \int_{0}^t \int_{\mathbb R^d} \int_{\mathbb R^d} e^{-\rho_1 s} |x|^{\ell-2} \big\{ \langle x,f(x,y) \rangle + (\ell -1) |g(x,y)|^2 \big\} \mu_s^X(dx)\mu_s^X(dy)ds
\\
  & \leq  W^{(\ell)}(\mu_0, \delta_0)^\ell + \Big(- \rho_1 +\Big\{\hat{L}_{b\sigma}^{(1)}+\hat{L}_{b\sigma}^{(2)}+ 2(\hat{L}_{fg}^{(1)})^+ +\frac{\hat{L}_{f}^{(1)}}{2}+\frac{3(\ell-2)}{2\ell}\Big\}\ell\Big)  \mathbb E \int_{0}^t e^{-\rho_1 s} |{X}_s|^{\ell} ds 
  \\
  & \quad + \bigl(2 (\hat{L}_{b\sigma})^{\ell/2} + (\hat{L}_{fg})^{\ell/2} \bigl) \int_0^t e^{-\rho_1 s} ds
\end{align*}
for all $t \in [0,T]$, and performing calculations similar to those above with $\rho_1 =0$ completes the proof.
It can also be noted from the above inequality that the bounds of $\ell$-th moment of the solution process does not depend on time $t$. 
\newline

\noindent 
\noindent
(B) Recall that $\{X_t\}_{t \geq 0}$ and $\{Y_t\}_{t \geq 0}$  are solutions of MV-SDE \eqref{Mckean} with initial values $X_0 \sim \mu_0$ and $Y_0 \sim \nu_0$, respectively,  where $\mu_0,\nu_0 \in \mathcal{P}_\ell (\mathbb R^d)$ with $\ell>2(q+1)$. 
In addition, their IPS are $\{X_t^{i,N}\}_{t \geq 0}$ and $\{Y_t^{i,N}\}_{t \geq 0}$, respectively.
Then, one observes that 
\begin{align}
    W^{(2)}(P_{t} \mu_0, P_{t} \nu_0)^2  \leq &\, \mathbb E |X_t^i - Y_t ^i|^2 \leq  3 \{\mathbb E |X_t^i-X_t^{i,N}|^2 + \mathbb E |X_t^{i,N}-Y_t^{i,N}|^2 + \mathbb E |Y_t^i-Y_t^{i,N}|^2\} {\label{E:1+E:2+E:3}}
\end{align}
for all $t \in [0,T]$. 
For the first term on the right side of Equation \eqref{E:1+E:2+E:3}, applying It\^o\textquotesingle{s} formula  and performing similar calculations as done in Equation \eqref{eq:ito:poc:dd}  and then using Assumption \ref{as:Er Sol mon}, one obtains  
\begin{align*}
   \frac{1}{N}\sum_{i=1}^N \mathbb E e^{-\rho_2 t }|X_t^i-X_t^{i,N}|^2  \leq &  \, \big\{2L_{b\sigma}^{(1)} +4L_{fg}^{(1)+}+1 - \rho_2\big\}  \int_0^t e^{-\rho_2 s }  \frac{1}{N}\sum_{i=1}^N \mathbb E |X_s^i-X_s^{i,N}|^2 ds 
   \\
   &  +  2L_{b\sigma}^{(2) }\int_0^t e^{-\rho_2 s } \frac{1}{N}\sum_{i=1}^N \mathbb E W^{(2)} (\mu_s^{X^i}, \mu_s^{X,N})^2 ds
   \\
   &   + \frac{1}{N} \sum_{i=1}^N \int_0^t  e^{-\rho_2 s } \mathbb E \Bigl | \frac{1}{N}\sum_{j=1}^N \int_{\mathbb R^d} (f(X_s^i , x)  -  f({X}^{i}_{s},{X}^{j}_{s}) ) \mu_s^{X^i}(dx) \Bigl |^2  ds
   \\
    &  + \frac{4}{N^3} \sum_{i=1}^N  \int_0^t e^{-\rho_2 s } \Bigl | \sum_{j=1}^N \mathbb E\int_{\mathbb R^d} (g(X_s^i , x)  -  g({X}^{i}_{s},{X}^{j}_{s}) ) \mu_s^{X^i}(dx) \Bigl |^2 ds
\end{align*}
for all $t \in [0,T]$. 
To estimate the third and fourth terms on the right side of the above equation, one perform similar calculations as done Equations \eqref{eq:f:rate} and \eqref{eq:g:rate}. 
However,  constants $K>0$ appearing therein do not depend on time $t$ due to part (A) of this theorem.  
Thus, 
\begin{align*}
 e^{-\rho_2 t }  \frac{1}{N}\sum_{i=1}^N \mathbb E |X_t^i-X_t^{i,N}|^2  \leq &  \, \big\{2L_{b\sigma}^{(1)} +4L_{fg}^{(1)+} +4L_{b\sigma}^{(2) } +1 - \rho_2\big\}  \int_0^t e^{-\rho_2 s }  \frac{1}{N}\sum_{i=1}^N \mathbb E |X_s^i-X_s^{i,N}|^2 ds 
   \\
   &  +  4L_{b\sigma}^{(2) }\int_0^t e^{-\rho_2 s }  \mathbb E W^{(2)} \Bigl(\frac{1}{N}\sum_{j=1}^N \delta_{X_s^j}, \mu_s ^{X} \Bigl)^2 ds
  + \frac{K}{N} \int_0^t  e^{-\rho_2 s } ds 
\end{align*}
where $K>0$  depends only on $L_f^{(1)}$ and $L_{fg}^{(1)}$. 
Moreover,  as $\rho_2 = 2L_{b\sigma}^{(1)} +4L_{fg}^{(1)+}+4L^{(2)}_{b \sigma}+1$ and the particles are exchangeable, one has 
\begin{align*}
         \mathbb E |X_t^i-X_t^{i,N}|^2 & \leq   4L_{b\sigma}^{(2) }\int_0^t e^{-\rho_2 (t-s) }  \mathbb E W^{(2)} \Bigl(\frac{1}{N}\sum_{j=1}^N \delta_{X_s^j}, \mu_s ^{X} \Bigl)^2 ds + \frac{K (e^{\rho_2 t}-1)}{N \rho_2} 
\end{align*}
which by using \eqref{rate:empi}
 leads to 
\begin{align}{\label{E:1}}
    \lim_{N \to \infty}  \mathbb E |X_t^i-X_t^{i,N}|^2 =0
\end{align}
for all $t \in [0,T]$.

For the estimation of second term on the right side of Equation \eqref{E:1+E:2+E:3},  apply It\^o\textquotesingle{s} formula to write, 
\begin{align*}
    e^{-\rho_2 t}\mathbb E|X_t ^{i,N}   - Y_t ^{i,N}|^2 = &\, \mathbb E|X_0 ^{i}   - Y_0 ^{i}|^2 -\rho_2  \int_0 ^t e^{-\rho_2 s} \mathbb E |X_s ^{i,N}   - Y_s ^{i,N}|^2 ds
    \\
    &  + 2 \mathbb E \int_0 ^t e^{-\rho_2 s}\big\{\langle X_s ^{i,N}   - Y_s ^{i,N}, b(s,X_s ^{i,N}, \mu_s^{X,N}) - b(s,Y_s ^{i,N}, \mu_s^{Y,N})\rangle 
    \\
    & \qquad + | \sigma(s,X_s ^{i,N}, \mu_s^{X,N}) - \sigma(s,Y_s ^{i,N}, \mu_s^{Y,N})|^2 \big\} ds
    \\
    &  + 2 \mathbb E \int_0 ^t e^{-\rho_2 s}\Big\langle X_s ^{i,N}   - Y_s ^{i,N}, \frac{1}{N} \sum_{j=1}^N \big\{f(X_s ^{i,N},X_s ^{j,N})-f(Y_s ^{i,N},Y_s ^{j,N})\big\}\Big\rangle ds 
    \\
    &  +2  \mathbb E \int_0^t e^{-\rho_2 s} \Bigl |\frac{1}{N} \sum_{j=1}^N \big\{g(X_s ^{i,N},X_s ^{j,N})-g(Y_s ^{i,N},Y_s ^{j,N})\big\} \Bigl |^2 ds
\end{align*}
which on using Assumption \ref{as:Er Sol mon} and Jensen\textquotesingle{s} inequality yields
\begin{align*}
&   e^{-\rho_2 t} \frac{1}{N} \sum_{i=1}^N \mathbb E|X_t ^{i,N}   - Y_t ^{i,N}|^2 
\\
&
\leq  \, \frac{1}{N} \sum_{i=1}^N \mathbb E|X_0 ^{i}   - Y_0 ^{i}|^2 + \{2L_{b\sigma}^{(1)}+2L_{b\sigma}^{(2)}- \rho_2 \} \int_0 ^t e^{-\rho_2 s} \frac{1}{N} \sum_{i=1}^N  \mathbb E |X_s ^{i,N}   - Y_s ^{i,N}|^2 ds
    \\
   & \qquad  + \frac{1}{N^2} \sum_{i,j=1}^N 
   \, \mathbb E \int_0 ^t e^{-\rho_2 s}  \big\{ \langle (X_s ^{i,N}   - Y_s ^{i,N})-(X_s ^{j,N}   - Y_s ^{j,N}),f(X_s ^{i,N},X_s ^{j,N})-f(Y_s ^{i,N},Y_s ^{j,N}) \rangle 
    \\
    & \qquad + 2 |g(X_s ^{i,N},X_s ^{j,N})-g(Y_s ^{i,N},Y_s ^{j,N})|^2 \big\} ds
    \\
    & \leq   \, \frac{1}{N} \sum_{i=1}^N \,\mathbb E|X_0 ^{i}   - Y_0 ^{i}|^2 + \{2L_{b\sigma}^{(1)}+2L_{b\sigma}^{(2)}+4L_{fg}^{(1)+}- \rho_2 \} \int_0 ^t e^{-\rho_2 s} \frac{1}{N} \sum_{i=1}^N  \mathbb E |X_s ^{i,N}   - Y_s ^{i,N}|^2 ds
    \\
  &  \leq  \,  \frac{1}{N} \sum_{i=1}^N  \mathbb E|X_0 ^{i}   - Y_0 ^{i}|^2 .
\end{align*}
 Now using the exchangeability of the particles we obtain 
    \begin{align}\label{E:2}
        \mathbb E|X_t ^{i,N}   - Y_t ^{i,N}|^2 
       \leq e^{\rho_2 t} \mathbb  E|X_0 ^{i} - Y_0 ^{i}|^2 \leq 2 e^{\rho_2 t} \bigl(\mathbb  E|X_0 ^{i}|^2 +  \mathbb  E| Y_0 ^{i}|^2 \bigl)
    \end{align}
for all $t \in [0,T]$.
By following similar calculations as in \eqref{E:1}, one can get
\begin{align}{\label{E:3}}
   \lim_{N  \to \infty} \mathbb E|Y_t^i -Y_t ^{i,N}|^2 = 0
\end{align}
for all $t \in [0,T]$.
Thus, by gathering the estimates \eqref{E:1}, \eqref{E:2} and \eqref{E:3} into Equation \eqref{E:1+E:2+E:3} we conclude the part (B) of the theorem.
\newline \newline \noindent
(C) We comment first on the constants $q,\ell,p_0$ that were initially assumed $q>0$, $\ell>2(q+1)$ and some $p_0 \geq \ell$; note that since $q>0$ it follows naturally $\ell>2$. The additional restriction $2\ell -2 > 2(q+1)$ is stronger than $\ell>2(q+1)$ since $\ell>2$ at least. The condition $p_0 \geq  2\ell -2$ suffices for $P_{\cdot}\mu_0$ to make sense for any $\mu_0\in \mathcal{P}_{(2\ell -2)} (\mathbb R^d) $. The reason for this is laid bare in \eqref{eq:l:drop}. 

First, notice that for any $\mu_0 \in \mathcal{P}_\ell(\mathbb R^d)$ we have that $ P_{t+s}\mu_0= P_{t}( P_{s}\mu_0)$ hold as $b$ and $\sigma$ are independent of time. Also, due to  statement (B), one can obtain,  
\begin{align*}
     W^{(2)} (P_{t}\mu_0, P_{t+s}\mu_0) 
     & 
     =   W^{(2)} \big ( P_{t}\mu_0, P_{t}(P_{s}\mu_0) \big)
     \leq  \sqrt{6} e^{\rho_2 t /2} \sup_{s \geq 0} \bigl(W^{(2)} (\mu_0,\delta_0) + W^{(2)}(\delta_0,P_{s}\mu_0) \bigl) < \infty
\end{align*}
where the last inequality holds due to statement (A)  with $\rho_1 <0$ and hence, following from $\rho_2<0$ we have 
\begin{align*}
\lim_{t \to \infty} \sup_{s \geq 0}    W^{(2)} (P_{t}\mu_0, P_{t+s}\mu_0) =0.
\end{align*}
Since $(\mathcal{P}_2(\mathbb R^d),W^{(2)})$ is a complete space, therefore, there exists a $\mu \in \mathcal{P}_2(\mathbb R^d)$ such that 
\begin{align*}
    \lim_{t \rightarrow \infty}W^{(2)} (P_{t}\mu_0,\mu)=0\qquad \textrm{for all $\mu_0 \in \mathcal{P}_\ell (\mathbb R^d)$. }
\end{align*}

\color{black}
With $\mu \in \mathcal{P}_2(\mathbb R^d)$, we do not have sufficient integrability to ensure that $P_t \mu$ makes sense via the well-posedness Theorem \ref{E-U} (and its assumption \ref{assum_initial} with $p_0>2(q+1)$). 
We next show that $\{P_t \mu_0\}_{t \geq 0}$ is Cauchy in $(\mathcal{P}_\ell(\mathbb R^d),W^{(\ell)})$ and for that we increase the integrability of $\mu_0$ as to leverage a Cauchy-Schwarz inequality argument. Concretely, by the Cauchy-Schwarz inequality we have for $s>0$
\begin{align}
\nonumber 
     W^{({\ell})} \big(P_{t}\mu_0, P_{t+s}\mu_0\big)^{{\ell}}
    = W^{({\ell})} \big( P_{t}\mu_0, P_{t}(P_s\mu_0) \big)^{{\ell}} 
    & \leq \mathbb E |X_t - Y_t|^{{\ell}}
    \\
    &
    \leq \Big (\mathbb E |X_t - Y_t|^{2} \mathbb E |X_t - Y_t|^{2{\ell}-2} \Big)^{1/2} 
    \leq K e^{\rho_2 t /2} 
    \label{eq:l:drop}
\end{align}
where $\{X_t \}_{t \geq 0}$ and $\{Y_t \}_{t \geq 0}$ are the solutions of MV--SDE \eqref{Mckean} with initial laws $\mu_0$ and $P_s\mu_0$, respectively, and with  $\mu_0,P_s\mu_0 \in\mathcal{P}_{2\ell -2}(\mathbb R^d)$. Hence, $\{P_t \mu_0\}_{t \geq 0}$ is Cauchy in $(\mathcal{P}_\ell(\mathbb R^d),W^{(\ell)})$ which implies $\bar \mu \in \mathcal{P}_\ell(\mathbb R^d)$.
Further, \cite[Remark 6.12]{villani2009OT} yields,  
\begin{align*}
    W^{(2)} (P_s\bar \mu,\bar \mu)
    & 
    \leq \liminf_{t \rightarrow \infty} W^{(2)} \bigl(P_s(P_t\mu_0),P_t\mu_0\bigl) 
    \leq 
    \liminf_{t \rightarrow \infty}     W^{(2)} (P_{t+s}\mu_0,\bar \mu) 
    + \liminf_{t \rightarrow \infty}   W^{(2)} \bigl(\bar \mu, P_t\mu_0\bigl) =  0
\end{align*}
for all $s \geq 0$. 
Thus, $\bar \mu$ is an invariant measure of MV--SDE \eqref{Mckean}. 
To close the proof, for any $\hat{\nu} \in \mathcal{P}_{2\ell-2}(\mathbb R^d)$  we have from the earlier results that 
\begin{align*}
    W^{(2)} (\bar{P}_t^N \hat{\nu},\bar \mu)  
    \leq 
    W^{(2)} (\bar{P}_t^N \hat{\nu},\bar{P}_t^N\bar \mu) 
     +  W^{(2)} (\bar{P}_t^N \bar \mu, \bar \mu)
     \leq 
     \sqrt{6} e^{\rho_2 t /2} \bigl(W^{(2)}(\hat{\nu},\delta_0)^2 + W^{(2)}(\delta_0,\bar \mu)^2 \bigl).
\end{align*}
In turn, this leads to 
\begin{align*}
    \lim_{ t \to \infty} W^{(2)} (P_t\hat{\nu},\bar  \mu)  =0\qquad \textrm{for any $\hat{\nu} \in \mathcal{P}_{2\ell-2}(\mathbb R^d)$.}
\end{align*}
Thus, the invariant measure $\bar \mu$ is also the ergodic limit of the marginal flows of MV-SDE \eqref{Mckean}. 
The arguments just used in combination with those used for \eqref{eq:l:drop} allow to straightforward establish the convergence and invariance in $W^{(r)}$ for any $r\in[2,\ell]$.   
\end{proof}
\subsubsection{The Vlasov kernel case}
In the arguments of the proof above, the invariant measure $\mu \in \mathcal{P}_{\ell} (\mathbb R^d)$ is found by starting the MV-SDE flow  $\hat{\nu}_0 \in \mathcal{P}_{2\ell -2} (\mathbb R^d)$.
This is happening due to the calculations involved in Equation \eqref{eq:l:drop}. 
However, if we consider the Vlasov kernel class discussed in Section \ref{sec:PoC-ind}, i.e., MV-SDE \eqref{Mckean} with coefficients given in Equation \eqref{eq:nu-barsig-inp}, then it can be shown that the invariant measure and $\hat{\nu}_0$ possess the $p_0$-th moment directly. 
The corollary below articulates this. Before stating it,  we make the following additional assumptions on the coefficients $\tilde{b}$ and $\tilde{\sigma}$ in line with the set up of Section \ref{sec:PoC-ind}. Concretely, 

\begin{assumption}{\label{as:ergo:DI:b:sig}}
There exist constants $\tilde{L}_{b\sigma}>0$, $\tilde{L}_{b\sigma}^{(2)} >0$ and $\tilde{L}_{b\sigma}^{(1)}$ $ \in \mathbb R$ such that
    \begin{align*}
        \langle x,\tilde{b}(t,x,y)\rangle  + (p_0-1)|\tilde{\sigma}(t,x,y)|^2 & \leq \tilde{L}_{b\sigma} + \tilde{L}_{b\sigma}^{(1)}|x|^2 + \tilde{L}_{b\sigma}^{(2)} |y|^2, 
    \end{align*}
    for all $t \in [0,\infty)$ and $x, y \in \mathbb R^d$. 
\end{assumption}

\begin{assumption}{\label{as:ergo:DI:mon:b:sig}}
There exist constants $\tilde{L}_{b\sigma}^{(4)}>0$, $\tilde{L}_b >0$ and $\tilde{L}_{b\sigma}^{(3)}$ $ \in \mathbb R$ such that
    \begin{align*}
        \langle x-x', \tilde{b}(t,x,y) - \tilde{b}(t,x',y') \rangle  +2(p_0-1) |\tilde{\sigma}(t,x,y) - \tilde{\sigma}(t,x',y')|^2 & \leq \tilde{L}_{b\sigma}^{(3)}|x-x'|^2+\tilde{L}_{b\sigma}^{(4)}|y-y'|^2,
         \\
         |\tilde{b}(t,x,y) - \tilde{b}(t,x,y')|  & \, \leq \tilde{L}_b |y-y'|,
    \end{align*}
    for all  $t \in [0,\infty)$ and $x, x', y, y' \in \mathbb R^d$.
\end{assumption}

\begin{assumption}{\label{as:ergo:DI:mon:f:g}}
There exist constants $\tilde{L}_{fg}^{(1)} \in \mathbb R$ and  $\tilde{L}_f^{(1)} >0$ such that
    \begin{align*}
         \langle (x  - y)-  (x'  - y'), f(x,y) -  f(x',y')  \rangle  + 4 (p_0-1) \, |  g(x,y) - g(x',y')  |^2   \leq  & \, \tilde{L}_{fg}^{(1)} |(x-y)-(x'-y')|^2, 
         \\
         (|x-x'|^{p_0-2}-|y-y'|^{p_0-2})  \langle  (x+y) -(x'+y'),f(x,y)  - f(x',y')\rangle \leq &\, \tilde{L}_{f}^{(1)}(|x-x'|^{p_0}+ |y-y'|^{p_0}), 
    \end{align*}
    for all $x$, $x'$, $y$,  $y' \in \mathbb R^d$.
\end{assumption}

\begin{corollary}{\label{Th:DI:Erog sol}}
Let $q>0$ and  $p_0 > \max\{2(q+1),5\}$. 
    Let Assumptions \ref{as:X0:erg}, \ref{as:Er Sol}, \ref{as:ergo:DI:b:sig}, \ref{as:ergo:DI:mon:b:sig}, \ref{as:ergo:DI:mon:f:g} be satisfied. 
    Then, the following statements hold. 
    \newline
    \textnormal{(A)} Let $\tilde\rho_1 := {p_0} \bigl(\tilde{L}_{b\sigma}^{(1)}+\tilde{L}_{b\sigma}^{(2)}+ 2(\hat{L}_{fg}^{(1)})^+ +\frac{\hat{L}_{f}^{(1)}}{2}+\frac{3({p_0}-2)}{2{p_0}}\bigl)$ and $\mu_0 \in \mathcal{P}_{p_0}(\mathbb R^d)$.
    Then for every $t \in [0,\infty)$,   
    \begin{align*}
        W^{({p_0})} (P_{t} \mu_0, \delta_0)^{p_0} \leq e^{{\tilde\rho_1}t} W^{({p_0})} (\mu_0,\delta_0)^{p_0} + \frac{2 (\tilde{L}_{b\sigma})^{{p_0}/2}(e^{\tilde\rho_1 t} -1)}{\tilde\rho_1} \mathbbm{1}_{{\rho_1} \neq 0} + \bigl(2 (\hat{L}_{b\sigma})^{\ell/2} + (\hat{L}_{fg})^{\ell/2} \bigl) t  \mathbbm{1}_{{\tilde\rho_1}=0}.
    \end{align*}
    Thus, if $\tilde\rho_1<0$, then $\displaystyle \sup_{t \in [0, \infty)}\mathbb E|X_t|^{{p_0}} \leq K$ where $K>0$ do not depend on time $t$.  
    \newline 
   \textnormal{(B)}  
   Let $\tilde\rho_2 := 2({p_0}-1) (2{p_0} -3) + \frac{{p_0}}{2} (2\tilde{L}_{b\sigma}^{(3)} + 2\tilde{L}_{b\sigma}^{(4)} + 4\tilde{L}_{fg}^{(1)+} + \tilde{L}_{f}^{(1)})$.
   Then,  for every  $\mu_0,\nu_0 \in \mathcal{P}_{p_0}(\mathbb R^d)$ and $t \in [0, \infty)$
        \begin{align*}
             W^{({p_0})}(P_{t} \mu_0, P_{t} \nu_0)^{p_0}  &\leq 2\cdot3^{{p_0} -1} e^{\tilde\rho_2 t} \bigl( W^{({p_0})}(\mu_0,\delta_0)^{p_0}+ W^{({p_0})}(\delta_0,\nu_0)^{p_0} \bigl).
        \end{align*}
        \newline 
   \textnormal{(C)} Let $\tilde b$ and $\tilde \sigma$ be independent of time, and $\tilde\rho_1 $ and $\tilde\rho_2 <0$
   Then,   there exists a unique probability measure $\tilde\mu \in \mathcal{P}_{p_0}(\mathbb R^d)$  of such that    
    \begin{align*}
            W^{({p_0})} (P_{t} \tilde\mu, \tilde\mu) =0 \quad \mbox{for all } t \in [0, \infty) \mbox{ and }  \lim_{t \rightarrow \infty} W^{({p_0})} (P _{t} \hat{\nu}, \tilde\mu) =0
        \end{align*}
    for all  $\hat{\nu} \in \mathcal{P}_{{p_0}}(\mathbb R^d) $.
\end{corollary}

\begin{proof}
(A) It can be verified that Assumption \ref{as:Er Sol}  holds due to Assumption \ref{as:ergo:DI:b:sig} with  $\hat{L}_{b\sigma}= \tilde{L}_{b\sigma}$, $\hat{L}_{b\sigma}^{(1)}=\tilde{L}_{b\sigma}^{(1)}$ and $\hat{L}_{b\sigma}^{(2)}=\tilde{L}_{b\sigma}^{(2)}$ where one takes $b(t,x,\mu)=\displaystyle \int_{\mathbb R^d} \tilde{b}(t,x,y)\mu(dy)$ and $\sigma(t,x,\mu)=\displaystyle \int_{\mathbb R^d} \tilde{\sigma}(t,x,y)\mu(dy)$. 
Thus, statement (A) of this corollary follows  from the statement (A) of Theorem \ref{Th:Erog IPS}.
\newline

\noindent
 
 \noindent
(B) 
Notice that, 
\begin{align}
    W^{(p_0)}(P_{t} \mu, P_{t} \nu)^{p_0}  \leq &\, \mathbb E |X_t^i - Y_t ^i|^{p_0} \leq  3^{{p_0}-1} \{\mathbb E |X_t^i-X_t^{i,N}|^{p_0} + \mathbb E |X_t^{i,N}-Y_t^{i,N}|^{p_0} + \mathbb E |Y_t^i-Y_t^{i,N}|^{p_0} \} {\label{R:1+R:2+R:3}}
\end{align}
for all $t \in [0,T]$. For the first term on the right side of Equation \eqref{R:1+R:2+R:3}, applying It\^o\textquotesingle{s} formula, one can obtain from \eqref{A1+.+A6} 
   \begin{align*}
   \nonumber 
      \frac{1}{N} \sum_{i=1}^N 
      & \mathbb E e^{-\tilde{\rho}_1 t}|X_t^i  - X_t^{i,N}|^{p_0}
      \\
      \leq 
       &  
      \,  \Bigl(-\tilde{\rho}_1 + 2({p_0}-1) (2{p_0} -3) + \frac{{p_0}}{2} (2\tilde{L}_{b\sigma}^{(3)} + 2\tilde{L}_{b\sigma}^{(4)} + 4\tilde{L}_{fg}^{(1)+} + \tilde{L}_{f}^{(1)}) \Bigl)
      \frac{1}{N} \sum_{i=1}^N \int_0^t e^{-\tilde{\rho}_1 s} \mathbb E|X_s^i  - X_s^{i,N}|^{p_0} ds  \notag
     \\
  &+ \frac{1}{N}\sum_{i=1}^N \int_0^t  e^{-\tilde{\rho}_1 s}\Big|\int_{\mathbb R^d} \tilde{b}(s,X_s^{i},y)\mu_s^{X^i}(dy) - \frac{1}{N} \sum_{j=1}^N \tilde{b}(s,X_s^{i},X_s^j) \Big|^{p_0} ds \nonumber
         \\
   &  +4({p_0}-1)\frac{1}{N} \sum_{i=1}^N \int_0^t e^{-\tilde{\rho}_1 s} \Big|\int_{\mathbb R^d}  \tilde{\sigma}(s,X_s^{i},y)\mu_s^{X^i}(dy)  - \frac{1}{N} \sum_{j=1}^N \tilde{\sigma}(s,X_s^{i},X_s^j)\Big|^{p_0}  ds \nonumber
  \\
       &\, + \int_0^t e^{-\tilde{\rho}_1 s}  \frac{1}{N} \sum_{i=1}^N  \mathbb E \Big| \int_{\mathbb R^d} f(X_s^i, x) \mu_s^{X^i}(dx) - \frac{1}{N} \sum_{j=1}^N f({X}^{i}_{s},{X}^{j}_{s}) \Big|^{p_0} ds  \nonumber
    \\
     & + 4({p_0}-1) \frac{1}{N}\sum_{i=1}^N \int_0^t e^{-\tilde{\rho}_1 s} \Bigl| \int_{\mathbb R^d} g(X_s^i,x) \mu_s^{X^i}(dx) -\frac{1}{N}\sum_{j=1}^N g({X}^{i}_{s},{X}^{j}_{s})\Bigl|^{p_0} ds 
   \end{align*}
  for all $t \in [0,T]$. Notice that, last four terms have been estimated in \ref{eq:A3}, \ref{eq:A4}, \ref{A5} and \ref{A6}. Therefore, by using exchangeability of particles
  \begin{align*}
      \mathbb E|X_t^i  - X_t^{i,N}|^{p_0} 
      \leq 
      & 
      \Bigl(-\tilde{\rho}_1 + 2({p_0}-1) (2{p_0} -3) + \frac{{p_0}}{2} (2\tilde{L}_{b\sigma}^{(3)} + 2\tilde{L}_{b\sigma}^{(4)} + 4\tilde{L}_{fg}^{(1)+} + \tilde{L}_{f}^{(1)}) \Bigl)
      \\
      & \qquad \qquad \qquad \qquad \times \int_0^t e^{-\tilde{\rho}_1 (t-s)} \mathbb E |X_s^i - X_s^{i,N}|^{{p_0}} ds + \frac{K (e^{\tilde{\rho}_1 t}-1)}{N^{{p_0} /2} \tilde{\rho}_1} 
 \end{align*}
 for all $t \in [0,T]$, where constant $K$ depends on $\tilde{L}_b^{(1)}$, $\tilde{L}_{b\sigma}^{(3)}$, $\tilde{L}_{b\sigma}^{(4)}$, $L_f^{(1)}$ and $\tilde{L}_{fg}^{(1)}$.
Thus, 
\begin{align}{\label{R:1}}
    \lim_{N \to \infty}  \mathbb E |X_t^i-X_t^{i,N}|^2 =0.
\end{align}
and by similar calculations, one obtains
\begin{align}{\label{R:3}}
    \lim_{N \to \infty}  \mathbb E |Y_t^i-Y_t^{i,N}|^2 =0
\end{align} 
Now, for the second term on the right side of Equation \eqref{R:1+R:2+R:3}, applying It\^o\textquotesingle{s} formula, one obtains 
\begin{align*}
\nonumber
        \frac{1}{N} \sum_{i=1}^N e^{-\tilde{\rho}_1 t} 
        & \mathbb E| X_t^{i,N}-Y_t^{i,N}|^{p_0} 
       \\ \nonumber 
       & \leq 
       \, \frac{1}{N} \sum_{i=1}^N \mathbb E|X_0^{i,N}-Y_0^{i,N}|^{p_0} - \tilde{\rho}_1 \frac{1}{N} \sum_{i=1}^N \int_0^t e^{-\tilde{\rho}_1 s} \mathbb E|X_s^{i,N}-Y_s^{i,N}|^{p_0} ds \nonumber
       \\
       & + \frac{{p_0}}{N^2} \sum_{i,j=1}^N \mathbb E \int_0^t e^{-\tilde{\rho}_1 s} |X_s^{i,N}-Y_s^{i,N}|^{{p_0}-2} \big\{ \big\langle X_s^{i,N}-Y_s^{i,N}, \tilde{b}(s,X_s^{i,N},X_s^{j,N}) - \tilde{b}(s,X_s^{i,N},X_s^{j,N})\big\rangle \nonumber
         \\
         & \qquad \qquad \qquad  +({p_0}-1) \big|\tilde{\sigma}(s,X_s^{i,N},X_s^{j,N}) - \tilde{\sigma}(s,X_s^{i,N},X_s^{j,N}) \bigl|^2 \big\} ds \nonumber
         \\
         &  +\frac{{p_0}}{N^2} \sum_{i,j=1}^N  \mathbb E\int_0^t e^{-\tilde{\rho}_1 s} |X_s^{i,N}-Y_s^{i,N}|^{{p_0}-2} \Bigl \{\big \langle X_s^{i,N}-Y_s^{i,N}, f(X_s^{i,N},X_s^{j,N})  - f({X}^{i,N}_{s},{X}^{j,N}_{s}) \big \rangle \nonumber
         \\
        & \qquad \qquad \qquad  + ({p_0}-1) \bigl|g(X_s^{i,N},X_s^{j,N})  - g({X}^{i,N}_{s},{X}^{j,N}_{s}) \bigl|^2 \bigl \}ds \nonumber
   \end{align*}
which gives for all $t \in [0,T]$
\begin{align*}
     & \mathbb E| X_t^{i,N}-Y_t^{i,N}|^{p_0} 
     \\
     &
     \leq \, e^{-\tilde{\rho}_1 t} \mathbb E|X_0^{i,N}-Y_0^{i,N}|^{p_0} + \bigl( - \tilde{\rho}_1 + \frac{{p_0}}{2} (2\tilde{L}_{b\sigma}^{(3)} + 2\tilde{L}_{b\sigma}^{(4)} + 4\tilde{L}_{fg}^{(1)+} + \tilde{L}_{f}^{(1)}) \bigl) \int_0^t e^{\tilde{\rho}_1 (t-s)} \mathbb E|X_s^{i,N}-Y_s^{i,N}|^{p_0} ds \nonumber
       \\
      & \leq e^{\tilde{\rho}_1 t} \mathbb E|X_0^{i,N}-Y_0^{i,N}|^{p_0}
      \\
       & \leq 2e^{\tilde{\rho}_1 t} \bigl( \mathbb E|X_0^{i,N}|^{p_0} + \mathbb E|Y_0^{i,N}|^{p_0} \bigl)
\end{align*}
Finally, by gathering all the estimates and injecting them in \eqref{R:1+R:2+R:3} concludes.
\newline

\noindent
(C) Since $\tilde{b}$ and $\tilde{\sigma}$ are independent of time, therefore using the first statement and second statement with $\tilde{\rho}_1 <0$
\begin{align*}
     W^{({p_0})} (P_{t}\mu_0, P_{t+s}\mu_0) 
     & =   W^{({p_0})} \big(P_{t}\mu_0, P_{t}(P_{s}\mu_0)\big)
     \\
     & \leq  \sqrt{2}\cdot 3^{{p_0} -1 /2}e^{\tilde\rho_2 t/2} \sup_{s \geq 0}\bigl(W^{(p_0)} (\mu_0,\delta_0) + W^{(p_0)}(\delta_0,P_{s}\mu_0) \bigl) < \infty
\end{align*}
which implies 
\begin{align*}
    \lim_{t \to \infty} \sup_{ s\geq 0} W^{({p_0})} (P_{t}\mu_0, P_{t+s}\mu_0) & =   0
\end{align*}
as $\tilde{\rho}_2 <0$. Thus, there exist $\tilde{\mu} \in \mathcal{P}_{p_0} (\mathbb R^d)$ such that and $\lim_{t \to \infty} W^{{p_0}} (P_t \mu_0, \tilde{\mu}) =0$. Further, following a similar approach as in Part (C) of Theorem \ref{Th:Erog sol} one can conclude the result.
\end{proof}


\subsection{Ergodicity of the interacting particle system}
\label{subsec:Ergod-IPS}
We now establish the ergodicity of the interacting particle system (IPS) \eqref{IPS} and we recall that the IPS system \eqref{IPS} is exchangeable and thus $\textrm{Law}(X_t^{i,N})=\textrm{Law}(X_t^{j,N})$ for any $t,i,j$.  
If the initial laws of the non IPS \eqref{NIPS} are $\mu_0^{X^{i}}= \mu_0$ and $\mu_0^{Y^{i}}= \nu_0$ for any $i$, then we use $\{X_t^{i,N}\}_{t \geq 0}$ and $\{Y_t^{i,N}\}_{t \geq 0}$ to denote the solutions of IPS \eqref{IPS} with $\mathbb R^d$-valued $\mathscr{F}_0$-measurable i.i.d.~random initial values $X_0^{i,N}=X_0^i$ and $Y_0^{i,N}=Y_0^i$, respectively.
Their flows of marginals are denoted by $\{P_{t}^{i,N} \mu_0\}_{t\geq 0} := \{\mu_t^{X^{i,N}}\}_{t \geq 0}$  and $\{P_{t}^{i,N} \nu_0\}_{t\geq 0} := \{\mu_t^{Y^{i,N}}\}_{t\geq 0}$ for $i \in \{1,\ldots,N \}$, respectively. This construction akin to that in the previous section and as in \cite{yuanping2024explicit}.

As we are looking for the invariant distribution of the IPS system, which is a measure on $(\bR^d)^N$, we introduce the joint law of the IPS as 
$\bar P^N_t \nu_0^{\otimes N}:= \textrm{Law}\big(  (X^{1,N}_t,\ldots,X^{N,N}_t) \big)$ where $\textrm{Law}(X^{i,N}_0)=\nu_0$ for all $i=1,\ldots,N$ and $\nu_0^{\otimes N}$ is the $N$-tensorised initial distribution $\nu_0 \in \cP(\bR^d)$.
In view of the results in Sections \ref{sec:IPS-and-PoC} and \ref{subsec:Ergod-MV-SDE}, one expects that as $t\to \infty$ the $N$-particle flow $\bar P^N_t \nu_0^{\otimes N}$ converges approximately to the $N$-tensorized invariant measure $\mu ^{\otimes N}$ of the MV-SDE \eqref{Mckean} (and at some exponential rate in $t$ uniformly over $N$) but, in particular, that $\textrm{Law}(X^{i,N}_t)\to \mu$ weakly as $t\to \infty$ as $N\to \infty$ (for any $i$). 
\color{black}

\begin{theorem}{\label{Th:Erog IPS}}
    Let the assumptions of Theorem \ref{Th:Erog sol} hold and consider $\rho_1$ and $\rho_2$ as given therein.  Then, the following statements hold.
    \newline
    \textnormal{(A)} Let $\mu_0 \in \mathcal{P}_\ell(\mathbb R^d)$. 
    Then, for any $t \in [0,\infty)$ and $i \in \{1,\ldots, N\}$,  
    \begin{align*}
         W^{(\ell)} (P_{t}^{i, N} \mu_0, \delta_0)^\ell \leq e^{{\rho_1}t} W^{(\ell)} (\mu_0,\delta_0)^\ell + \frac{2 (\hat{L}_{b\sigma})^{\ell/2}(e^{\rho_1 t} -1)}{\rho_1} \mathbbm{1}_{\{\rho_1 \neq 0\}} +\bigl(2 (\hat{L}_{b\sigma})^{\ell/2} + (\hat{L}_{fg})^{\ell/2} \bigl) t  \mathbbm{1}_{\{\rho_1=0\}}. 
    \end{align*}
     Further, if $\rho_1 <0$, then $\displaystyle \sup_{i\in \{1,\ldots,N\}} \sup_{t \in [0, \infty)}\mathbb E|X^{i, N}_t|^{\ell} \leq  K$ where $K$ does not depend on $N$ (or time). 
    \newline 
   \textnormal{(B)} If $\mu_0,\nu_0 \in \mathcal{P}_\ell(\mathbb R^d),$ then, for every $t \in [0,\infty)$ and $i \in \{1,\ldots, N\}$, the following contraction holds 
   \begin{align*}
             \sup_{i\in \{1,\ldots,N\}} W^{(2)}(P_{t}^{i,N} \mu_0, P_{t}^{i,N} \nu_0)^2  
             &
             \leq 
             2e^{\rho_2 t} \bigl(W^{(2)}(\mu_0,\delta_0)^2 + W^{(2)}(\delta,\nu_0)^2 \bigl).
        \end{align*}
   \newline 
   \textnormal{(C)} 
   Let $b$, $\sigma$ be independent of time and $\rho_1,\rho_2 <0$  under the restriction that $p_0,\ell$ satisfy $2\ell -2 > 2q+2$ and $p_0 \geq  2\ell -2$. 
   Then, there exists a unique invariant probability measure $\bar \mu^N \in \mathcal{P}_\ell \big( (\mathbb R^d)^N \big)$ of IPS \eqref{IPS} in the following sense (using the notation introduced just above the theorem): for any $r\in[2,\ell]$ we have 
\begin{align*}
    W^{(r)} (\bar P^N _{t} \bar \mu^N, \bar \mu^N) =0\ \ \mbox{ for all } t \in [0,\infty) 
    \quad \textrm{and}\quad\textrm{for all $\hat{\nu}^{\otimes N} \in \mathcal{P}_{2\ell-2}(\mathbb (\bR^d)^N)$} 
    \quad 
    \lim_{t  \to \infty} W^{(r)} (\bar P^{N}_{t} \hat{\nu}^{\otimes N}, \bar \mu^N) =0. 
\end{align*}
\end{theorem}
\begin{proof}
This proof is established by directly adapting calculations carried out previously in Section \ref{subsec:Ergod-MV-SDE}. 
\newline  \noindent
(A) 
The proof follows by adapting arguments similar to those used in the proof of statement (A) of Theorem \ref{Th:Erog sol} where Corollary \ref{lem:symm-growth} is used instead of Lemma  \ref{lem:sym-gwt:erg}. The uniformity of $K$ over the particle number $N$ is straightforward to establish (see also the arguments used to prove Theorem \ref{thm:eu-ips}).  
\newline \newline \noindent
(B) 
The proof follows by straightforward work with Equation \eqref{E:2}.   
 \newline \newline \noindent
(C) 
Let us start by observing that statements (A) and  (B) holds for $\mu_0 ^{\otimes N} \in \mathcal{P}_{\ell}(\mathbb (R^d)^N)$ with additional dependence on $N$ but as $N$ is fixed for IPS, one gets the invariant distribution to the IPS. Following similar steps as in Part (C) of Theorem \ref{Th:Erog sol}, one can get the desired result where one needs to carry the argument over the IPS joint law flow $(\bar P^N_t \hat \nu^{\otimes N})_t $ leveraging the results established in part (A) and (B) for the IPS's marginals---critically, the constants may depend on $N$ but $N$ is fixed for the argument as one is establishing the invariant distribution to the IPS (and not yet recovering the invariant distribution of the original MV-SDE).  
Let $\mu_0 \in \mathcal{P}_\ell(\mathbb R^d)$, then  $\mu_0 ^{\otimes N} \in \mathcal{P}_{\ell}(\mathbb (\bR^d)^N)$ and from the earlier notation $\bar{P}_t^N \mu_0 ^{\otimes N}=\textrm{Law}(X_t^{1,N}, \ldots, X_t^{N,N})$ and  $\bar{P}_0^N \mu_0 ^{\otimes N}=\mu_0 ^{\otimes N}$. 
From part (A) we have for any $t \in[0,\infty)$ 
\begin{align}
    W^{(\ell)}(\bar{P}_t^N \mu_0^{\otimes N}, \delta_0^{\otimes N})^\ell 
    & 
    = \mathbb E \Big(\sum_{i=1}^N   |X_t^{i,N}|^2 \Big)^{\ell/2} \leq  N^{\ell/2-1}\sum_{i=1}^N \mathbb E|X_t^{i,N}|^\ell \notag
    \\
& \leq N^{\ell/2} \Big(e^{{\rho_1}t} W^{(\ell)} (\mu_0,\delta_0)^\ell + \frac{2 (\hat{L}_{b\sigma})^{\ell/2}(e^{\rho_1 t} -1)}{\rho_1} \mathbbm{1}_{{\rho_1} \neq 0} +\bigl(2 (\hat{L}_{b\sigma})^{\ell/2} + (\hat{L}_{fg})^{\ell/2} \bigl) t  \mathbbm{1}_{{\rho_1}=0}  \Big). 
\label{eq:ergo:IPS,dirac}
\end{align}
Similarly, from part (B), one obtains for any $t \in[0,\infty)$ 
\begin{align}
    W^{2}(\bar{P}_t^N \mu_0^{\otimes N},\bar{P}_t^N \nu_0^{\otimes N})^2 
    & 
    \leq \mathbb E \Big(\sum_{i=1}^N   |X_t^{i,N}-Y_t^{i,N}|^2 \Big)  
    \leq 2 N e^{\rho_2 t} \bigl(W^{(2)}(\mu_0,\delta_0)^2 + W^{(2)}(\delta_0,\nu_0)^2 \bigl). \label{eq:ergo:IPS,diff}
\end{align}
Now, as $b$ and $\sigma$ are independent of time and $\rho_1 <0,$ therefore from \eqref{eq:ergo:IPS,dirac} and \eqref{eq:ergo:IPS,diff}, one can see that for $\mu_0 \in \mathcal{P}_\ell (\mathbb R^d)$ 
    \begin{align*}
     W^{(2)} \big(\bar{P}_{t}^N \mu_0^{\otimes N}, \bar{P}_{t+s}^N \mu_0^{\otimes N}\big) ^2
     & 
     = 
     W^{(2)} \big(\bar{P}_{t}^N \mu_0^{\otimes N}, \bar{P}_{t}^N(\bar{P}_{s}^N \mu_0^{\otimes N})\big)^2
     \\
     &
     \leq 
     2 e^{\rho_2 t/2} \sup_{s \geq 0} \bigl(N\cdot W^{(2)}(\mu_0,\delta_0)^2 
     + W^{(2)}(\delta_0^{\otimes N},\bar{P}_{s}^N \mu_0^{\otimes N})^2 \bigl) 
     < \infty
\end{align*}
which in turn gives due to due to $\rho_2 <0$, 
\begin{align*}
    \lim_{t \to \infty} \sup_{s \geq 0} W^{(2)} (\bar{P}_{t}^N \mu_0^{\otimes N}, \bar{P}_{t+s} ^N\mu_0^{\otimes N}) =0,
\end{align*}
and hence we have a Cauchy sequence. Since $\big(  \mathcal{P}_2\big( (\mathbb R^d)^N\big), W^{(2)} \big)$ is complete metric space, therefore, there exists a $\bar \mu^{N} \in \mathcal{P}_2((\mathbb R^d)^N)$ such that 
\begin{align*}
    \lim_{t \rightarrow \infty}
    W^{(2)} (\bar{P}_{t}^N\mu_0^{\otimes N},\bar \mu^{N})=0
    \qquad \textrm{for all $\mu_0 \in \mathcal{P}_\ell ( \mathbb R^d )$. }
\end{align*}
Thus, to show that $\bar \mu^{N} \in \mathcal{P}_\ell ((\mathbb R^d)^N)$, it is easy to see that given $\mu_0 \in \mathcal{P}_{2\ell-2} ( \mathbb R^d )$ 
    \begin{align}
\nonumber 
     W^{({\ell})} \big(\bar{P}_{t}^N \mu_0^{\otimes N}, \bar{P}_{t+s}^N \mu_0^{\otimes N} \big)^{{\ell}} 
     &
    = 
    W^{({\ell})} \big( \bar{P}_{t}^N \mu_0^{\otimes N}, \bar{P}_{t}^N (\bar{P}_s^N \mu_0^{\otimes N}) \big)^{{\ell}} 
    \leq \mathbb E \Bigl(\sum_{i=1}^N |X_t^{i,N} - Y_t^{i,N}|^{{2}} \Bigl)^{\ell/2} 
    \\
    & \leq N^{\ell/2 -1} \sum_{i=1}^N \Big (\mathbb E |X_t^{i,N} - Y_t^{i,N}|^{2} \mathbb E |X_t^{i,N} - Y_t^{i,N}|^{2{\ell}-2} \Big)^{1/2} 
    \leq K N^{\ell/2 -1} e^{\rho_2 t /2}, \notag
\end{align}
where  $\text{Law}\bigl((X_0^{1,N},\ldots,X_0^{N,N})\bigl)= \mu_0^{\otimes N}$ and $\text{Law}\bigl((Y_0^{1,N},\ldots,Y_0^{N,N})\bigl)= \bar{P}_s^N \mu_0^{\otimes N}$ , respectively, and $\mu_0^{\otimes N},\bar{P}_s^N \mu_0^{\otimes N} \in\mathcal{P}_{2\ell -2}((\mathbb R^d)^N)$. Hence, $\{\bar{P}_t^N \mu_0^{\otimes N} \}_{t \geq 0}$ is Cauchy in $\big( \mathcal{P}_\ell\big( (\mathbb R^d)^N\big), W^{(\ell)}\big)$ which implies the existence of a limiting measure $\bar \mu^{ N} \in \mathcal{P}_\ell((\mathbb R^d)^N)$. Furthermore, \cite[Remark 6.12]{villani2009OT} yields,  
\begin{align*}
    W^{(2)} (\bar{P}^N_s \mu^{\otimes N},\bar \mu^{N}) 
    &  
    \leq \liminf_{t \rightarrow \infty} 
    W^{(2)} \bigl(\bar{P}^N_s(\bar{P}^N_t \mu_0^{\otimes N}),\bar{P}^N_t \mu_0^{\otimes N} \bigl) 
    \\
    & 
    \leq 
    \liminf_{t \rightarrow \infty} W^{(2)} (\bar{P}^N_{t+s}\mu_0^{\otimes N}, \bar \mu^{N}) 
    +
    \liminf_{t \rightarrow \infty} W^{(2)} \bigl(\bar \mu^{N}, \bar{P}^N_t\mu_0^{\otimes N} \bigl) 
    =  0
\end{align*}
for all $s \geq 0$. 
Thus, $\bar \mu^{N}$ is an invariant measure of IPS \eqref{IPS}. 
Finally, for any $\hat{\nu} \in \mathcal{P}_{2\ell-2}(\mathbb R^d)$ 
\begin{align*}
    W^{(2)} ( \bar P_t \hat{\nu}^{\otimes N},\bar \mu^{N}) 
    & \leq 
    W^{(2)} (P_t\hat{\nu}^{\otimes N},P_t \bar \mu^{N})
    +  W^{(2)} (\bar P_t \mu^{N}, \bar \mu^{N}) 
    \\
    & \leq \sqrt{2} e^{\rho_2 t/2} 
    \bigl(  \sqrt{N}\cdot W^{(2)}(\hat \nu , \delta_0)
              +  W^{(2)}(\delta_0^{\otimes N},\bar \mu^N )\bigl) 
\end{align*}
which leads to the (uniqueness within the class) conclusion 
\begin{align*}
    \lim_{ t \to \infty} W^{(2)} (\bar P_t\hat{\nu}^{\otimes N},\bar \mu^{N})  =0
    \qquad \textrm{for any $\hat{\nu} \in \mathcal{P}_{2\ell-2}( \mathbb R^d)$.}
\end{align*}
We conclude the proof with an intuitive reasoning for a clarification we did not find in the literature. As mentioned, the IPS \eqref{IPS} forms an exchangeable system across time and one may ask if the attractor/invariant distribution $\bar \mu^N$ over $(\bR^d)^N$ also forms an exchangeable system. For any permutation $\pi$ of $\{1,\ldots,N\}$ we have for any $t\geq 0$
\begin{align*}
\bar P^N_t ( \hat \nu^{\otimes N} )
= 
\textrm{Law}\big(  (X^{1,N}_t,\ldots,X^{N,N}_t) \big)
=
\textrm{Law}\big(  (X^{\pi(1),N}_t,\ldots,X^{\pi(N),N}_t) \big) 
=: \bar P^{\pi(N)}_t ( \hat \nu^{\otimes N} )
\end{align*}
where the superscript $\pi(N)$ denotes the permutation deployed to the particle system, noting that the tensorised initial condition $\hat \nu ^{\otimes N}$ is invariant under permutations, i.e., $\hat \nu^{\pi(\otimes N)}=\hat \nu ^{\otimes N}$.

Via the existence of a limiting distribution we have (in the appropriate Wasserstein distance) 
\begin{align*}
\bar \mu^N = \lim _{t \rightarrow \infty} \bar P^N_t(\nu^{\otimes N})=\lim _{t \rightarrow \infty} \bar P^N_t(\nu^{\pi(\otimes N)}).
\end{align*} 
On the other hand, by the permutation-equivariance of the Wasserstein distance  
\begin{align*}
\lim _{t \rightarrow \infty} \bar P_t^N ( \nu^{\pi(\otimes N)} ) 
=\lim _{t \rightarrow \infty}  \bar P_t^{\pi(N)} (\nu^{\otimes N}) 
  =\pi\left(\lim _{t \rightarrow \infty} \bar P_t^N \nu\right)
=\bar \mu^{\pi(N)}, 
\end{align*}
where in the second equality we use that the coordinate permutation map is an isometry under the Wasserstein distance. 

Hence $ \bar \mu^{\pi(N)}=\bar \mu^N$ for all permutations \(\pi\), i.e. \(\mu\) is invariant under all coordinate permutations, which is exactly exchangeability for the limit measure. 
\end{proof}

The following corollary is a direct consequence of Corollary \ref{Th:DI:Erog sol}, which states that the invariant measure $\tilde{\mu}^{i,N}$ of the interacting particle system corresponding to MV--SDE \eqref{Mckean} with the coefficients given in Equation \eqref{eq:nu-barsig-inp}, and $\hat{\nu}_0$ has $p_0$-th moment.

\begin{corollary}{\label{Th:DI:IPS:Erog}}
    Let assumptions of Corollary \ref{Th:DI:IPS:Erog} be satisfied with $\tilde{\rho}_1$ and $\tilde{\rho}_2$ as given therein. 
    Then, the following statements hold. 
    \newline
    \textnormal{(A)} Let $\mu_0 \in \mathcal{P}_{p_0}(\mathbb R^d)$,  then for every $t \in [0,\infty)$,   
    \begin{align*}
        W^{({p_0})} (P_{t}^{i,N} \mu_0, \delta_0)^{p_0} \leq e^{{\tilde\rho_1}t} W^{({p_0})} (\mu_0,\delta_0)^{p_0} + \frac{2 (\tilde{L}_{b\sigma})^{{p_0}/2}(e^{\tilde\rho_1 t} -1)}{\tilde\rho_1} \mathbbm{1}_{{\rho_1} \neq 0} + \bigl(2 (\hat{L}_{b\sigma})^{\ell/2} + (\hat{L}_{fg})^{\ell/2} \bigl) t  \mathbbm{1}_{{\tilde\rho_1}=0}.
    \end{align*}
    Thus, if $\tilde\rho_1<0$, then $\displaystyle \sup_{i=1,\ldots,N} \sup_{t \in [0, \infty)}\mathbb E|X_t^{i,N}|^{{p_0}} \leq K$ where $K>0$ do not depend on $N$ (or time $t$).  
    \newline 
   \textnormal{(B)}  
   Let $\mu_0,\nu_0 \in \mathcal{P}_{p_0}(\mathbb R^d)$ and $t \in [0, \infty)$. Then for every $i \in \{1,\ldots,N \}$,
        \begin{align*}
             W^{({p_0})}(P_{t}^{i,N} \mu_0, P_{t}^{i,N} \nu_0)^2  &\leq 2^{{p_0}}3^{{p_0} -1} e^{\tilde\rho_2 t} \bigl(W^{({p_0})}(\mu_0,\delta_0)^{p_0}+ W^{({p_0})}(\delta_0,\nu_0)^{p_0}\bigl).
        \end{align*}
   \textnormal{(C)} Let $\tilde b$ and $\tilde \sigma$ be independent of time, and $\tilde\rho_1 $, $\tilde\rho_2 <0$. 
   Then,   there exists a unique probability measure $\tilde\mu^N \in \mathcal{P}_{p_0}\big( (\mathbb R^d)^N\big )$  of such that    
    \begin{align*}
            W^{({p_0})} (\bar P_{t}^{N} \tilde\mu^N, \tilde\mu^N) =0 \quad \mbox{for all } t \in [0, \infty) \quad \mbox{ and } \quad  \textrm{for any  $\hat{\nu} \in \mathcal{P}_{{p_0}}(\mathbb R^d) $} \quad 
\lim_{t \rightarrow \infty} W^{({p_0})} (\bar P _{t}^{N} \hat{\nu}^{\otimes N}, \tilde\mu) =0.
        \end{align*}
\end{corollary}
\subsection{Ergodicity of the tamed Euler scheme}
\label{sec:ergodicity tamed Euler scheme}
This section is narrowly focused on establishing the ergodicity of the tamed Euler scheme at fixed number of particles $N$ and at fixed time-step discretisation $h=1/n$, and in a setting where the diffusion coefficient is a non-diagonal constant as in \cite{yuanping2024explicit,schuh2024conditions} and in contrast to many other works \cite{bao2024geometric,Brehier2023MR4655541,MR4803778,ottobrecrisanangeli2025}. 
We will work with a taming definition, see Equation \eqref{taming-ergodic} below, which is slightly different from the construction in Section \ref{sec:Taming-in Euler scheme}. This is a bespoke definition for the ergodic case of the class of superlinear growth particle systems we tackle---we need to renew its definition given in Section \ref{sec:Taming-in Euler scheme} (see Equations \eqref{taming-FiniteT} and \eqref{Eulercheme-FiniteT}) to accommodate technical changes to the time-grid and a variation on the taming functions themselves -- compare Equations \eqref{taming-FiniteT} and \eqref{taming-ergodic}. Comparatively to the finite time case of Section \ref{sec:Taming-in Euler scheme}, there we took a stronger taming in order to give a shorter proof of the moment bounds, otherwise the proof here would be very lengthy. 
The results of Section \ref{sec:Taming-in Euler scheme} can be shown to hold with the taming of Equation \eqref{taming-ergodic}. 

We state a new variant of the assumptions to clearly specify the role of the involved constants in the main theorem, provide the scheme and taming, and then comment holistically on the changes. A surprising element to this approach is that under taming (at fixed $n,N$), one can work under $\cP_2(\bR^d)$ only. The details are given below.
\medskip

Throughout let $q>0$ indicating polynomial growth of functions, fix the particle system size to be $N$, and define a fixed time step $h:=1/n$ for some fixed $n\in \bN$ and the corresponding time grid $t_k=k/n=hk$ for $k\in \bN_0$.  
\begin{assumption} \label{as:sch:gr:erg}
There exist constants $\hat{L}_{b\sigma}^{(1)}>0$, $ \hat{L}_{b\sigma}^{(2)}>0$, $\hat{L}_{fg}^{(1)} > 0$ and $q>0$ such that 
    \begin{align*}
        \langle x,b(t,x,\mu)\rangle  + |\sigma(t,x,\mu)|^2 & \leq -\hat{L}_{b\sigma}^{(1)}(1+|x|^{q})|x|^2 + \hat{L}_{b\sigma}^{(2)} W^{(2)} (\mu,\delta_0)^2, 
        \\
         \langle x -y,f(x,y) \rangle + 2|g(x,y)|^2 & \leq -\hat{L}_{fg}^{(1)} (1 + |x-y|^{q} )|x-y|^2,
    \end{align*}
    for all $t \in [0,\infty)$, $x$, $y \in \mathbb R^d$ and $\mu \in \mathcal{P}_2(\mathbb R^d)$. 
\end{assumption}
\begin{assumption}{\label{as:diff:b:f:erg}}
There exist constants $L_b^{(1)}>0$ and $L_f^{(1)} > 0$   such that 
    \begin{align*}
    |b(t,x,\mu)  -b(t,x',\mu')|^2 & \leq L_{b} ^{(1)}(1+|x|^{2q} + |x'|^{2q})|x-x'|^2 + L_{b} ^{(2)} W^{(2)} (\mu,\mu')^2,
    \\
    |f(x,y)  - f(x',y')|^2 & \leq L_{f}^{(1)} (1+|x-y|^{2q} + |x'-y'|^{2q}) |(x-x')-(y-y')|^2,
\end{align*}
for all $t \in [0,\infty)$, $x$, $x'$, $y$, $y' \in \mathbb R^d$ and $\mu$, $\mu' \in \mathcal{P}_2(\mathbb R^d)$. 
Also, $\displaystyle \sup_{t \in [0, \infty)} |b(t,0,0)| <\infty$ and $\displaystyle \sup_{t \in [0, \infty)} |\sigma(t,0,0)| <\infty$. 
\end{assumption}
For a fixed $n \in \mathbb N$, define $t_k:=k/n$ for any $k\in \mathbb N$ and 
 redefine the tamed coefficients as,  
\begin{align}
\label{taming-ergodic}
   & (b^n, \sigma^n)(t,x,\mu)  := \frac{(b, \sigma)(t,x,\mu)}{1+n^{- 1/2}|x|^{q}}
   \qquad \mbox{ and }\qquad   (f^n, g^n)(x,y)  := \frac{(f,g)(x,y)}{1+n^{-1/2}|x-y|^{q}},
\end{align}
for all $t $, $x$, $y\in \mathbb R^d$ and $\mu \in \mathcal{P}_2(\mathbb R^d)$ with $q>0$ as given in Assumption \ref{as:sch:gr:erg}. 

We use the same notation for taming, in Equation \eqref{taming-FiniteT} and in Equation  \eqref{taming-ergodic}, which should not cause confusion as they are defined in different contexts. 
 Similarly,  $\hat{X}^{i,N,n}$ is used in both cases to denote the corresponding tamed Euler scheme, see Equations \eqref{scheme} and \eqref{eq:TamedEulerforErgodicity}.  
 For the purpose of ergodicity, we take fixed step-size $h:=1/n$ and  consider the following  tamed Euler scheme for IPS  \eqref{IPS},
\begin{align}\label{eq:TamedEulerforErgodicity}
    \hat{X}_{t_{k+1}} ^{i,N,n}
     = 
     \hat{X}_{t_{k}} ^{i,N,n} 
    & + \big (b^n (t_{k},\hat{X}_{t_{k}} ^{i,N,n}, \hat{\mu}_{t_k}^{X,N,n}) + \frac{1}{N} \sum_{j=1}^N f^n(\hat{X}_{t_{k}} ^{i,N,n},\hat{X}_{t_{k}} ^{j,N,n}) \big )h \nonumber
    \\
    & + \big ( \sigma^n (t_{k},\hat{X}_{t_{k}} ^{i,N,n}, \hat{\mu}_{t_k}^{X,N,n})+ \frac{1}{N} \sum_{j=1}^N g^n(\hat{X}_{t_{k}} ^{i,N,n} ,\hat{X}_{t_{k}} ^{j,N,n}) \big) \Delta W_{t_k}^i,
\end{align}
almost surely for all $i \in \{1,\ldots, N\}$ and $k \in \mathbb N_0$ with  initial law $\mu_0 \in \mathcal{P}_2(\mathbb R^d)$ of $\mathscr{F}_0$-measurable initial random variable $X_0 ^{i,N,n} = X_0 ^i$ where  $\Delta W_{t_k}^i:= W_{t_{k+1}}^i-W_{t_k}^i$    and $\hat{\mu}_{t_k}^{X,N,n} = \displaystyle \frac{1}{N} \sum_{j=1}^N \delta_{\hat{X}_{t_k}^{j,N,n}}$.

The main result of this section is Theorem \ref{Th:Erg:Sch}. 

\subsubsection{Additional assumptions, remarks and preliminary results}
We start with some direct implications from the previous assumptions. 
\begin{remark}
\label{R:gr:tam:erg}
By Assumption \ref{as:diff:b:f:erg}, for all  $t \in [0,\infty)$, $x$, $y \in \mathbb R^d$ and $\mu \in \mathcal{P}_2(\mathbb R^d)$, 
    \begin{align*}
        |b(t,x,\mu)|^2 & \leq L_b^{(1)} (1+|x|^{2q})|x|^2 + L_b^{(2)} W^{(2)}(\mu, \delta_0)^2+ L_{b}^*, 
        \\
        |f(x,y)|^2 &\leq L_{f}^{(1)} (1+|x-y|^{2q}) |x-y|^2+ L_f^*,
        \\
       |b^n(t,x,\mu)|^2 & \leq \min \bigl\{ \bigl(n(L_b^{(1)} |x|^2  + L_{b}^*) +L_b^{(2)} W^{(2)} (\mu,\delta_0)^2 \bigl), |b(t,x,\mu)| \bigl\},
       \\
       |f^n(x, y)|^2 & \leq \min\bigl\{n \bigl(L_f^{(1)}|x-y|^2 + L_f^* \bigl), |f(x,y)| \bigl\},
    \end{align*}
    where $L_b^*:=2\displaystyle \sup_{t \in [0, \infty)}|b(t,0,\delta_0)|^2>0$ and $L_f^*:=2|f(0,0)|^2>0$.
\end{remark}
\begin{remark}[A clarifying example regarding the assumptions, taming and ergodicity]
One of the main difficulties with the taming method and ergodicity, is the loss of the initial dissipativity property of the coefficients due to taming (see \cite{Brehier2023MR4655541,MR3843836,bao2024geometric}). 

In our case, the assumptions of Section \ref{sec:ergodicity tamed Euler scheme} require extra ingredients to ensure ergodicity. As way of example and for $x \in \mathbb R^d$, the function $-|x|^2 x -x$ satisfies all the assumptions listed in Section \ref{sec:ergodicity tamed Euler scheme}, however the function $-|x|^2 x$ is not covered in our setup as it fails to satisfy Assumption \ref{as:sch:gr:erg}. 
\end{remark}
\begin{lemma}\label{lem:gr:sch:erg}
    Let the antisymmetric property of $f$ in Assumption \ref{as:Er Sol} be satisfied and let Assumption \ref{as:sch:gr:erg} hold. 
    Then,  
    \begin{align*}
         \langle x,b^n(t,x,\mu)\rangle  + |\sigma^n(t,x,\mu)|^2 & \leq -\hat{L}_{b\sigma}^{(1)} |x|^2 + \hat{L}_{b\sigma}^{(2)} W^{(2)} (\mu,\delta_0)^2,
         \\
         \frac{1}{N^2} \sum_{i,j=1}^N \big\{\langle x^i,f^n(x^i,x^j) \rangle + |g^n(x^i,x^j)|^2 \big\} 
         & 
         \leq  -\frac{L_{fg}^{(1)}}{N^2} \sum_{i,j=1}^N |x^i - x^j|^2,
    \end{align*}
    for all $n \in \mathbb N$, $t \in [0,\infty)$, $x$, $x^1, \ldots, x^N \in \mathbb R^d$ and $\mu \in \mathcal{P}_2(\mathbb{R}^2)$.
\end{lemma}
\begin{proof}
 Using Assumption \ref{as:sch:gr:erg}, one can notice that 
    \begin{align*}
       \langle x,  b^n(t,x,\mu)&\,\rangle    + |\sigma^n(t,x,\mu)|^2  \leq \frac{1}{1+n^{-1/2}|x|^{q}} \big\{\langle x,b(t,x,\mu)\rangle  + |\sigma(t,x,\mu)|^2\big\}
       \\
        & \leq \frac{-\hat{L}_{b\sigma}^{(1)}(1+|x|^{q})|x|^2 + \hat{L}_{b\sigma}^{(2)} W^{(2)} (\mu,\delta_0)^2}{1+n^{-1/2}|x|^{q}}  \leq - \hat{L}_{b\sigma}^{(1)}\frac{1+|x|^{q}}{1+n^{-1/2}|x|^{q}}|x|^2 + \hat{L}_{b\sigma}^{(2)} W^{(2)} (\mu,\delta_0)^2
    \end{align*}
for all $t \in [0,\infty)$, $x \in \mathbb R^d$,  and $\mu \in \mathcal{P}_2(\mathbb R^d)$ and then inequality $1+n^{-1/2}|x|^q \leq 1+|x|^q$  completes the proof of the first statement. 
Also, by  anti-symmetric property of $f$ given in  Assumption \ref{as:Er Sol} and  Assumption \ref{as:sch:gr:erg}, one has
    \begin{align*}
        \frac{1}{N^2} \sum_{i,j=1}^N & \big\{\langle x^i,f^n(x^i,x^j) \rangle + |g^n(x^i,x^j)|^2 \big\} = \frac{1}{2N^2} \sum_{i,j=1}^N \big\{\langle x^i-x^j,f^n(x^i,x^j) \rangle + 2|g^n(x^i,x^j)|^2 \big\}
        \\
        & \leq \frac{1}{2N^2} \sum_{i,j=1}^N \frac{1}{1+n^{-1/2}|x^i -x^j|^q} \big\{-L_{fg}^{(1)}(1 + |x^i-x^j|^q )|x^i - x^j|^2 \big\}
    \end{align*}   
    for all $n \in \mathbb N$, $x$, $x^1,\ldots, x^N  \in \mathbb R^d$  and $\mu \in \mathcal{P}_2(\mathbb R^d)$. Using that $L^{(1)}\geq 0$ and the inequality mentioned above, concludes the proof of the second statement. 
\end{proof}
\begin{assumption}
\label{as: Er sch b f}
There exist constants $L_{b\sigma}^{(1)}>0$, $L_{b\sigma}^{(2)}>0$  and $L_{fg}^{(1)}>0$ such that
    \begin{align*}
       \big\langle x-x'
       ,b(t,x,\mu)-b(t,x',\mu')\big\rangle  & + 2|\sigma(t,x,\mu)-\sigma(t,x',\mu')|^2  
       \\
        & \leq -L_{b\sigma}^{(1)}(1+|x|^q + |x'|^q)|x-x'|^2 + L_{b\sigma}^{(2)} W^{(2)} (\mu,\mu')^2,   
       \\
         \big\langle (x-x') -(y-y')
         ,f(x,y) - f(x',y') \big\rangle & + 4|g(x,y)-g(x',y')|^2 
         \\
         & \leq -L_{fg}^{(1)} (1 + |x-y|^q + |x'-y'|^q)|(x-x')-(y-y')|^2,
    \end{align*} 
     for all $t \in [0,\infty)$, $x,x',y,y' \in \mathbb R^d$ and $\mu,\mu' \in \mathcal{P}_2(\mathbb R^d)$.
\end{assumption}
\begin{assumption}{\label{as: Er sch x b f}}
There exist constants $L_{b\sigma} ^{(3)}>0$, $L_{b\sigma} ^{(4)}>0$, $L_{b\sigma} ^{(5)}>0$, $L_{fg}^{(2)}>0$ and $ L_{fg}^{(2)} >0$ such that 
    \begin{align*}
       \langle x-x'&,b(t,x,\mu) |x'|^q -b(t,x',\mu') |x|^q \rangle  + 2|\sigma(t,x,\mu)|x'|^q-\sigma(t,x',\mu')|x|^q|^2 \\
        & \leq \bigl \{L_{b\sigma} ^{(3)}(1+|x|^q + |x'|^q) - L_{b\sigma} ^{(4)}|x|^q|x'|^q  \bigl \}|x-x'|^2  +  L_{b\sigma}^{(5)}W^{(2)} (\mu,\mu')^2,\\
        \langle (x-x') &-(y-y'),f(x,y)|x'-y'|^q - f(x',y')|x-y|^q \rangle + 4|g(x,y)|x'-y'|^q -g(x',y')|x-y|^q|^2 \\
         & \leq \bigl \{ L_{fg}^{(2)}(1+|x-y|^q + |x'-y'|^q) - L_{fg}^{(3)} |x-y|^q|x'-y'|^q \bigl\} |(x-x')-(y-y')|^2,  
    \end{align*}
     for all $t \in [0,\infty)$, $x,x',y,y' \in \mathbb R^d$ and $\mu,\mu' \in \mathcal{P}_2(\mathbb R^d)$.
\end{assumption}
\begin{assumption}{\label{as: Er sch gr b f}}
There exist constants $L_{b} ^{(3)}>0$, $L_{b} ^{(4)}>0$ and $L_{f} ^{(2)} >0$ such that 
    \begin{align*}
        \big|b(t,x,\mu)|x'|^q  & -b(t,x',\mu')|x|^q\big |^2 
        \\
        & \leq L_{b} ^{(3)}(1+|x|^{2q} + |x'|^{2q} + |x|^{2q}|x'|^{2q})|x-x'|^2 + L_{b} ^{(4)} W^{(2)} (\mu,\mu')^2,
        \\
        \big |f(x,y)|x'-y'|^q & - f(x',y')|x-y|^q\big|^2 
        \\
        & \leq L_{f}^{(2)} (1+|x-y|^{2q} + |x'-y'|^{2q} + |x-y|^{2q}|x'-y'|^{2q}) |(x-x')-(y-y')|^2,
    \end{align*}
     for all $t \in [0,\infty)$, $x,x',y,y' \in \mathbb R^d$ and $\mu,\mu' \in \mathcal{P}_2(\mathbb R^d)$.
\end{assumption} 
Assumptions \ref{as: Er sch b f} and \ref{as: Er sch x b f} are crucial as they ensure that the tamed coefficients are preserving some aspect of dissipativity condition while Assumption \ref{as: Er sch gr b f} is needed for the contraction of scheme as the time iterations progress, see statement (B) of Theorem \ref{Th:Erg:Sch}.
\begin{lemma}{\label{lem: Er sch b mono}}
Let Assumptions \ref{as:diff:b:f:erg}, \ref{as: Er sch b f}, \ref{as: Er sch x b f} and \ref{as: Er sch gr b f} be satisfied. 
Then
    \begin{align*}
         \langle x-x',b^n(t,x,\mu)  - b^n(t,x',\mu')\rangle  \, & + |\sigma^n(t,x,\mu)-\sigma^n(t,x',\mu')|^2  
         \\
         & \leq - \big \{ L_{b\sigma}^{(1)}/2 \wedge L_{b\sigma}^{(4)} \big \}|x-x'|^2 + \big\{L_{b\sigma}^{(2)}+L_{b\sigma}^{(5)}\big\} W^{(2)} (\mu,\mu')^2, 
         \\
        \langle (x-x') -(y-y'),f^n(x,y) - f^n(x',y') \rangle & + 2|g^n(x,y)-g^n(x',y')|^2 
        \\
        & \leq 
        - \big \{ L_{fg}^{(1)}/2 \wedge L_{fg}^{(3)} \big \} |(x-x')-(y-y')|^2, 
        \\
        |b^n(t,x,\mu)-b^n(t,x',\mu')|^2 & \leq 4n\{L_{b} ^{(1)} \vee L_{b} ^{(3)}\} |x-x'|^2
             + 2\{L_b^{(2)} + L_{b} ^{(4)}\}W^{(2)} (\mu,\mu')^2
             \\
         |f^n(x,y)-f^n(x',y')|^2 & \leq 4n \{L_{f}^{(1)} \vee L_{f}^{(2)}\}|(x-x')-(y-y')|^2
    \end{align*}
 for all $t \in [0,\infty)$, $x,y,x',y'\in \mathbb R^d$, $\mu,\mu' \in \mathcal{P}_2(\mathbb R^d)$ and  $h=(1/n) \in (0, h^*)$ where $h^*:= \big(\frac{L_{b\sigma}^{(1)}}{2L_{b\sigma}^{(3)}}\big)^2 \wedge \bigl (\frac{L_{fg}^{(1)}}{2L_{fg}^{(2)}} \bigl )^2$. 
\end{lemma}
    \begin{proof}
    Due to Equation \eqref{taming-ergodic},  observe that 
     \begin{align*}
     \langle x-x',& b^n(t,x,\mu)- b^n(t,x',\mu')\rangle  + |\sigma^n(t,x,\mu)-\sigma^n(t,x',\mu')|^2 
     \\
         =& \, \Bigl \langle  x-x' , \frac{b(t,x,\mu)}{1+n^{-1/2}|x|^q} -\frac{b(t,x',\mu')}{1+n^{-1/2}|x'|^q} \Bigl \rangle  +  \Bigl |\frac{\sigma(t,x,\mu)}{1+n^{-1/2}|x|^q} -\frac{\sigma(t,x',\mu')}{1+n^{-1/2}|x'|^q} \Bigl|^2 
         \\ 
          \leq & \, \frac{ \langle x-x', b(t,x,\mu)-b(t,x',\mu') \rangle + 2 |\sigma(t,x,\mu)-\sigma(t,x',\mu')|^2 }{1+n^{-1/2}(|x|^q + |x'|^q) + n^{-1} |x|^q |x'|^q} \\
         & + n^{-1/2} \frac{ \langle x-x',b(t,x,\mu) |x'|^q -b(t,x',\mu') |x|^q \rangle  + 2|\sigma(t,x,\mu)|x'|^q-\sigma(t,x',\mu')|x|^q|^2}{1+n^{-1/2}(|x|^q + |x'|^q) + n^{-1} |x|^q |x'|^q}
     \end{align*}
     which on using Assumptions \ref{as: Er sch b f} and \ref{as: Er sch x b f} yields the following estimates, 
     \begin{align*}
      \langle x-x'&,b^n(t,x,\mu)-b^n(t,x',\mu')\rangle  + |\sigma^n(t,x,\mu)-\sigma^n(t,x',\mu')|^2
      \\
         & \leq \frac{-L_{b\sigma}^{(1)}(1+|x|^q + |x'|^q)|x-x'|^2 + L_{b\sigma}^{(2)} W^{(2)} (\mu,\mu')^2 } {1+n^{-1/2}(|x|^q + |x'|^q) + n^{-1} |x|^q |x'|^q} 
         \\
        & \quad + n^{-1/2} \frac{\bigl \{L_{b\sigma} ^{(3)}(1+|x|^q + |x'|^q) - L_{b\sigma} ^{(4)}|x|^q|x'|^q  \bigl\}|x-x'|^2  +  L_{b\sigma}^{(5)}W^{(2)} (\mu,\mu')^2}{1+n^{-1/2}(|x|^q + |x'|^q) + n^{-1} |x|^q |x'|^q}  
        \\
         & \leq - \frac{\bigl\{\big(L_{b\sigma}^{(1)}- n^{-1/2}L_{b\sigma}^{(3)}\big)(1+|x|^q + |x'|^q)+n^{-1/2}L_{b\sigma}^{(4)} |x|^q|x'|^q \bigl\}|x-x'|^2}{1+n^{-1/2}(|x|^q + |x'|^q) + n^{-1} |x|^q |x'|^q} 
         \\
         & \qquad\qquad + \big\{L_{b\sigma}^{(2)}+L_{b\sigma}^{(5)}\big\} W^{(2)} (\mu,\mu')^2 
         \\
         & \leq - \Bigl\{ \frac{L_{b\sigma}^{(1)}}{2} \wedge L_{b\sigma}^{(4)} \Bigl \}|x-x'|^2 + \big\{L_{b\sigma}^{(2)}+L_{b\sigma}^{(5)}\big\} W^{(2)} (\mu,\mu')^2 
     \end{align*}
     for all $t \in [0,\infty)$, $x,x',y,y' \in \mathbb R^d$ and $\mu,\mu' \in \mathcal{P}_2(\mathbb R^d)$. In the last inequality, one uses  $\displaystyle n^{-1/2} = h^{1/2} \leq \frac{L_{b\sigma}^{(1)}}{2L_{b\sigma}^{(3)}} $ and the inequality $\displaystyle \frac{n^{-1/2} + a}{1+a} \geq \frac{n^{-1/2} + n^{-1/2} a}{1+a} = n^{-1/2}$ for all constant $a>0$.
   Further, using Assumptions \ref{as: Er sch b f} and \ref{as: Er sch x b f}, one can obtain
    \begin{align*}
        \langle & (x-x') -(y-y'),f^n(x,y) - f^n(x',y') \rangle + 2|g^n(x,y)-g^n(x',y')|^2 
        \\
        & \leq \frac{\langle (x-x')-(y-y'),f(x,y) - f(x',y') \rangle + 4|g(x,y)-g(x',y')|^2}{1+n^{-1/2}(|x-y|^q+ |x'-y'|^q)+n^{-1}|x-y|^q |x'-y'|^q} 
        \\
        & \qquad + n^{-1/2}\frac{ \langle (x-x')-(y-y'),f(x,y)|x'-y'|^q - f(x',y')|x-y|^q \rangle + 4|g(x,y)|x'-y'|^q -g(x',y')|x-y|^q|^2}{1+n^{-1/2}(|x-y|^q+ |x'-y'|^q)+n^{-1}|x-y|^q |x'-y'|^q}
        \\
        & \leq \frac{ -L_{fg}^{(1)} (1 + |x-y|^q + |x'-y'|^q) |(x-x')-(y-y')|^2 }{1+n^{-1/2}(|x-y|^q+ |x'-y'|^q) +n^{-1}|x-y|^q |x'-y'|^q}\\
        & \qquad + n^{-1/2}  \frac{ \big\{L_{fg}^{(2)}(1+|x-y|^q + |x'-y'|^q) - L_{fg}^{(3)} |x-y|^q|x'-y'|^q \big\}  |(x-x')-(y-y')|^2}{1+n^{-1/2}(|x-y|^q+ |x'-y'|^q)+n^{-1}|x-y|^q |x'-y'|^q}
        \\
        & \leq - \Bigl \{ \frac{L_{fg}^{(1)}}{2} \wedge L_{fg}^{(3)} \Bigl \} |(x-x')-(y-y')|^2
    \end{align*}
  for all $x,x',y,y' \in \mathbb R^d$.  Furthermore, using Assumptions \ref{as:diff:b:f:erg} and \ref{as: Er sch gr b f}, it is easy to see that
         \begin{align*}
             |b^n  (t,x,\mu)& -b^n(t,x',\mu')|^2 \leq \frac{2|b(t,x,\mu)-b(t,x',\mu')|^2 + 2n^{-1}|b(t,x,\mu) |x'|^q - b(t,x',\mu') |x|^q |^2}{(1+n^{-1/2}(|x|^q + |x'|^q) + n^{-1} |x|^q |x'|^q)^2}\\
             & \leq \frac{2L_{b} ^{(1)}(1+|x|^{2q} + |x'|^{2q})|x-x'|^2 + 2L_{b} ^{(2)} W^{(2)} (\mu,\mu')^2  }{(1+n^{-1/2}(|x|^q + |x'|^q) + n^{-1} |x|^q |x'|^q)^2}\\
             & \qquad + n^{-1} \frac{2 L_{b} ^{(3)}(1+|x|^{2q} + |x'|^{2q} +  |x|^{2q}|x'|^{2q})|x-x'|^2 + 2L_{b} ^{(4)} W^{(2)} (\mu,\mu')^2 }{(1+n^{-1/2}(|x|^q + |x'|^q) + n^{-1} |x|^q |x'|^q)^2}
             \\
             & \leq \frac{4\{L_{b} ^{(1)} \vee L_{b} ^{(3)}\}\bigl (1+|x|^{2q} + |x'|^{2q} + n^{-1} |x|^{2q} |x'|^{2q} \bigl )|x-x'|^2}{1+n^{-1}(|x|^{2q} + |x'|^{2q}) + n^{-2} |x|^{2q} |x'|^{2q}}  + 2(L_b^{(2)} + L_{b} ^{(4)})W^{(2)} (\mu,\mu')^2
             \\
             & \leq 4n\{L_{b} ^{(1)} \vee L_{b} ^{(3)}\} |x-x'|
             + 2\{L_b^{(2)} + L_{b} ^{(4)}\} W^{(2)} (\mu,\mu')^2
         \end{align*}
     for all $t \in [0,\infty)$, $x,x',y,y' \in \mathbb R^d$ and $\mu,\mu' \in \mathcal{P}_2(\mathbb R^d)$. Finally, using Assumptions \ref{as:diff:b:f:erg} and {\ref{as: Er sch gr b f}}, one can observe
\begin{align*}
     |f^n &(x,y) -f^n(x',y')|^2  \leq \frac{2 |f(x,y)-f(x',y')|^2 + 2n^{-1}|f(x,y)|x-y|^q-f(x',y')|x'-y'|^q|^2}{(1+n^{-1/2}(|x-y|^q+|x'-y'|^q)+n^{-1}|x-y|^q|x'-y'|^q)^2}\\
     & \leq \frac{2L_{f}^{(1)} (1+|x-y|^{2q} + |x'-y'|^{2q}) |(x-x')-(y-y')|^2 }{(1+n^{-1}(|x-y|^{2q}+|x'-y'|^{2q})+n^{-2}|x-y|^{2q}|x'-y'|^{2q})}\\
     & \quad + n^{-1} \frac{2L_{f}^{(2)} (1+|x-y|^{2q} + |x'-y'|^{2q} + |x-y|^{2q}|x'-y'|^{2q}) |(x-x')-(y-y')|^2}{(1+n^{-1}(|x-y|^{2q}+|x'-y'|^{2q})+n^{-2}|x-y|^{2q}|x'-y'|^{2q})}
     \\
     & \leq 4n\{L_{f}^{(1)} \vee L_{f}^{(2)}\}|(x-x')-(y-y')|^2
\end{align*}
for all $x,x',y,y' \in \mathbb R^d$ which completes the proof.
    \end{proof}

\subsubsection{The main result}

We now introduce a final element of notation to use in the main result on the ergodic properties of the tamed Euler scheme \eqref{eq:TamedEulerforErgodicity}. 
Let $X_0$ and $Y_0$ be $\mathbb R^d-$valued  $\mathscr{F}_0$-measurable random variables. 
If $\mu_0 := \mu^X_0$, then the solution of scheme \eqref{eq:TamedEulerforErgodicity},  is denoted by $\{\hat{X}_{t_k} ^{i,N,n}\}_{k \in \mathbb N_0}$. 
Similarly, $\{\hat{Y}_{t_k} ^{i,N,n}\}_{k \geq 1}$ represents the solution of the scheme \eqref{eq:TamedEulerforErgodicity} with an initial value $Y_0$ having law $\nu_0:= \nu_0^Y$. 
Moreover, we define $P_{t_k}^{i, N, n} \mu_0 := \mu^{\hat{X}^{i,N,n}}_{t_k}$ and $P_{t_k}^{i, N, n} \nu_0 := \nu^{\hat{Y}^{i,N,n}}_{t_k}$ where $\mu^{\hat{X}^{i,N,n}}_{t_k}$ and $\nu^{\hat{Y}^{i,N,n}}_{t_k}$ are laws of $\hat{X}^{i,N,n}_{t_k}$  and $\hat{Y}^{i,N,n}_{t_k}$ for any $k \in \mathbb N_0$ and $i \in \{1,\ldots,N\}$, respectively. 
 
As in the previous section, we are looking for the invariant distribution of the tamed Euler scheme associated to the IPS system which is a measure on $(\bR^d)^N$. We introduce the joint law of the tamed Euler scheme as 
$\bar P^{N,n}_{t_k} \mu_0^{\otimes N}:= \textrm{Law}\big(  (\hat X^{1,N,n}_{t_k},\ldots,\hat X^{N,N,n}_{t_k}) \big)$  where $\textrm{Law}(\hat X^{i,N,n}_0)=\mu_0$ (or alternatively $\bar P^{N,n}_{0} \mu_0^{\otimes N}=\mu_0^{\otimes N}$) for all $i=1,\ldots,N$ and $\mu_0^{\otimes N}$ is the $N$-tensorised initial distribution $\mu_0 \in \cP_\cdot(\bR^d)$.  
\color{black}

\begin{theorem}{\label{Th:Erg:Sch}}
Fix $N,n\in \bN$, $h=1/n$ and $q>0$. Let Assumption \ref{as:X0:erg} hold for some $p_0\geq 2$. 
Let the antisymmetric property of $f$ in Assumption \ref{as:Er Sol} holds.
In addition, let Assumptions \ref{as:sch:gr:erg} to \ref{as: Er sch gr b f} be satisfied. 
Assume the constants $L_b^*,L_f^*$ defined in Remark \ref{R:gr:tam:erg} satisfy $L_b^*>0$ and $L_f^*>0$. 
Recall  $h^*>0$ as defined in Lemma \ref{lem: Er sch b mono}, and set the constants $\hat{\rho}_1,\hat{\rho}_2$ as 
\begin{align*}
\hat{\rho}_1 & := \hat{L}_{b\sigma}^{(1)} - \hat{L}_{b\sigma}^{(2)} - 4L_{fg}^{(1)} -L_{b}^{(1)} - L_b^{(2)} -4L_f^{(1)}>0
\\
\hat{\rho}_2 & := \big(L_{b\sigma}^{(1)}/2 \wedge L_{b\sigma}^{(4)} \big)-\big(L_{b\sigma}^{(2)}+L_{b\sigma}^{(5)} \big) +2 \big (L_{fg}^{(1)}/2 \wedge L_{fg}^{(3)} \big) - 4 (L_b^{(1)} \vee L_b^{(3)}) - 2 (L_b^{(2)} \vee L_b^{(4)}) - 16(L_f^{(1)} \vee L_f^{(2)}).  
\end{align*}
Then, the following statements are true. 
\newline \noindent 
 \textnormal{(A)} For every initial measure $\mu_0 \in \mathcal{P}_\ell(\mathbb R^d)$,  
 $k \in \mathbb N_0$ and  $ i \in \{1,\ldots,N \}$, one has,  
\begin{align*}
    W^{(2)} (P_{t_{k}}^{i,N,n}\mu_0,\delta_0)^2 \leq e^{-2 \hat{\rho}_1t_{k}} W^{(2)}(\mu_0, \delta_0)^2 + \frac{L_b^*+L_f^*}{ \hat{\rho}_1} \big\{1-(1- 2h \hat{\rho}_1)^{k}\big\}. 
\end{align*}
Moreover, $\displaystyle \sup_{i\in \{1,\ldots,N\}} \sup_{k \in \mathbb N_0} \mathbb E|\hat{X}_{t_{k}}^{i,N,n}|^2 \leq K$ whenever $ h \in (0, 1/2\hat{\rho}_1)$ where $K>0$ is independent of $N,n,k,h$.

\noindent
   \textnormal{(B)} 
   For all $\mu_0, \nu_0 \in \mathcal{P}_\ell(\mathbb R^d)$, $k \in \mathbb N_0$ and $h  \in (0,\min \{h^*,1/2\hat{\rho}_1\}) $  one has  
       \begin{align*}
       W^{(2)}(P^{i,N,n}_{t_{k}}\mu_0,P^{i, N,n}_{t_{k}}\nu_0)^2 \leq 2 e^{-2\hat{\rho}_2t_{k}}
       \big(W^{(2)}(\mu_0,\delta_0)^2 + W^{(2)}(\delta_0,\nu_0)^2 \big). 
    \end{align*}  
    
 \noindent
    \textnormal{(C)}  
    Assume that $b$ and $\sigma$ are independent of time.  
    If $\hat{\rho}_1, \hat{\rho}_2 >0$ and $h \in (0,h^* \wedge 1/(2\hat{\rho}_1))$, then there exists a unique invariant measure      $\bar {\mu}^{N,n} \in \mathcal{P}_2 \big(  (\mathbb R^d)^N\big)$ to the tamed Euler scheme \eqref{eq:TamedEulerforErgodicity} in the following sense  
    \begin{align*}
     W^{(2)}(\bar P_{t_k}^{ N,n} \bar {\mu}^{N,n}, \bar {\mu}^{N,n})  = 0 \quad \mbox{ for all } k \in \mathbb N_0 \quad    \mbox{ and } \quad 
    \textrm{for all $\hat{\nu} \in \mathcal{P}_{2}(\mathbb R^d)$} 
        \quad 
        \lim_{k \rightarrow \infty}  
        W^{(2)}(\bar P_{t_k}^{N,n} \hat{\nu}^{\otimes N}, \bar {\mu}^{N,n}) = 0.
\end{align*}
\end{theorem}
\begin{proof}
(A)  The result holds trivially for $k=0$. 
 For $k \in \mathbb N_0$, it can be  seen that
        \begin{align*}
         \mathbb E|\hat{X}_{t_{k+1}}^{i,N,n}|^2  = &\, \mathbb E |\hat{X}_{t_{k}}^{i,N,n}|^2 + 2h \mathbb E \Bigl\langle\hat{X}_{t_{k}}^{i,N,n},b^n(\hat{X}_{t_{k}}^{i,N,n},\hat{\mu}_{t_k}^{X,N,n}) + \frac{1}{N}\sum_{j=1}^N f^n(\hat{X}_{t_{k}}^{i,N,n},\hat{X}_{t_{k}}^{j,N,n}) \Bigl\rangle 
         \\
         &  + h^2 \mathbb E \Bigl| b^n(\hat{X}_{t_{k}}^{i,N,n},\hat{\mu}_{t_k}^{X,N,n}) + \frac{1}{N}\sum_{j=1}^Nf^n(\hat{X}_{t_{k}}^{i,N,n},\hat{X}_{t_{k}}^{j,N,n})\Bigl|^2 
         \\
         &  + h \mathbb E \Bigl| \sigma^n(\hat{X}_{t_{k}}^{i,N,n},\hat{\mu}_{t_k}^{X,N,n}) + \frac{1}{N}\sum_{j=1}^N g^n(\hat{X}_{t_{k}}^{i,N,n},\hat{X}_{t_{k}}^{j,N,n})\Bigl|^2 
    \end{align*}
for any   $i \in \{1,\ldots,N \}$. 
On averaging  both sides, one gets
\begin{align}
       \frac{1}{N} \sum_{i=1}^N \mathbb E|\hat{X}_{t_{k+1}}^{i,N,n}|^2  \leq & \, \frac{1}{N} \sum_{i=1}^N \mathbb E |\hat{X}_{t_{k}}^{i,N,n}|^2 + \frac{2h}{N} \sum_{i=1}^N \mathbb E \bigl\{\langle\hat{X}_{t_{k}}^{i,N,n},b^n(\hat{X}_{t_{k}}^{i,N,n},\hat{\mu}_{t_k}^{X,N,n}) \rangle + |\sigma^n(\hat{X}_{t_{k}}^{i,N,n},\hat{\mu}_{t_k}^{X,N,n})|^2\bigl\} \notag
         \\
       &  + \frac{2h}{N^2} \sum_{i,j=1}^N \mathbb E \bigl\{\langle\hat{X}_{t_{k}}^{i,N,n}, f^n(\hat{X}_{t_{k}}^{i,N,n},\hat{X}_{t_{k}}^{j,N,n}) \rangle + |g^n(\hat{X}_{t_{k}}^{i,N,n},\hat{X}_{t_{k}}^{j,N,n})|^2 \bigl\} \notag
       \\
       &  + \frac{2h^2}{N} \sum_{i=1}^N \mathbb E | b^n(\hat{X}_{t_{k}}^{i,N,n},\hat{\mu}_{t_k}^{X,N,n})|^2 + \frac{2h^2}{N^2}\sum_{i,j=1}^N \mathbb E |f^n(\hat{X}_{t_{k}}^{i,N,n},\hat{X}_{t_{k}}^{j,N,n})|^2 \label{eq:dis:ito}
\end{align}
which on using Lemma \ref{lem:gr:sch:erg} and Remark \ref{R:gr:tam:erg} gives
\begin{align*}
    \frac{1}{N} \sum_{i=1}^N & \mathbb E|\hat{X}_{t_{k+1}}^{i,N,n}|^2 \leq \frac{1}{N} \sum_{i=1}^N \mathbb E |\hat{X}_{t_{k}}^{i,N,n}|^2 
    -\hat{L}_{b\sigma}^{(1)} \frac{2h}{N} \sum_{i=1}^N \mathbb E  |\hat{X}_{t_{k}}^{i,N,n}|^2 + 2h \hat{L}_{b\sigma}^{(2)} W^{(2)} (\hat{\mu}_{t_k}^{X,N,n},\delta_0)^2 
    \\
    & \quad - \frac{2h L_{fg}^{(1)}}{N^2} \sum_{i,j=1}^N \mathbb E |\hat{X}_{t_{k}}^{i,N,n}-\hat{X}_{t_{k}}^{j,N,n}|^2 
    + L_b^{(1)} n \frac{2h^2}{N} \sum_{i=1}^N \mathbb E |\hat{X}_{t_{k}}^{i,N,n}|^2 + 2h L_b^{(2)} W^{(2)} (\hat{\mu}_{t_k}^{X,N,n},\delta_0)^2 + 2hL_b^*
    \\
    & \quad + \frac{2h^2 n}{N^2} \sum_{i,j=1}^N \mathbb E \big\{L_f^{(1)}|\hat{X}_{t_{k}}^{i,N,n}-\hat{X}_{t_{k}}^{j,N,n}|^2 + L_f^*\big\}
\end{align*}
and after further simplification and  iterations, one obtains the following estimates, 
\begin{align*}
   \frac{1}{N} \sum_{i=1}^N \mathbb E |\hat{X}_{t_{k+1}}^{i,N,n}|^2  \leq &\, \bigl\{1 - 2h(\hat{L}_{b\sigma}^{(1)} - \hat{L}_{b\sigma}^{(2)} - 4L_{fg}^{(1)} -L_{b}^{(1)} - L_b^{(2)} -4L_f^{(1)}) \bigl\} \frac{1}{N} \sum_{i=1}^N \mathbb E |\hat{X}_{t_{k}}^{i,N,n}|^2  + 2h(L_b^*+L_f^*)
    \\
     = & \bigl(1- 2h \hat{\rho}_1 \bigl) \frac{1}{N} \sum_{i=1}^N\mathbb E |\hat{X}_{t_{k}}^{i,N,n}|^2  + 2h(L_b^*+L_f^*)
    \\
     \leq &\, \bigl(1- 2h \hat{\rho}_1 \bigl)^{k+1} \frac{1}{N} \sum_{i=1}^N\mathbb E |\hat{X}_{t_{0}}^{i}|^2  + 2(L_b^*+L_f^*)h \sum_{j=0}^k (1- 2h \hat{\rho}_1)^j
    \\
     \leq & \, e^{-2\hat{\rho}_1 t_{k+1}} \frac{1}{N} \sum_{i=1}^N \mathbb E |\hat{X}_{t_{0}}^{i}|^2  + \frac{L_b^*+L_f^*}{ \hat{\rho}_1} \big\{1-(1- 2h \hat{\rho}_1)^{k+1}\big\}.
\end{align*}
 Gathering the estimates, and using the exchangeability property of the particles, one gets
\begin{align*}
    \mathbb E |\hat{X}_{t_{k+1}}^{i,N,n}|^2 \leq e^{-2 \hat{\rho}_1t_{k+1}}\mathbb E |\hat{X}_{t_{0}}^{i}|^2  + \frac{L_b^*+L_f^*}{ \hat{\rho}_1} \big\{1-(1- 2h \hat{\rho}_1)^{k+1}\big\}
\end{align*}
for all $k \in \mathbb N_0$. 
Notice that
\begin{align*}
   W^{(2)} (P_{t_{k+1}}^{i,N,n}\mu_0,\delta_0)^2 & \leq \mathbb E |\hat{X}_{t_{k+1}}^{i,N,n}|^2 \leq e^{-2 \hat{\rho}_1t_{k+1}} W^{(2)}(\mu_0, \delta_0)^2 + \frac{L_b^*+L_f^*}{ \hat{\rho}_1} \big\{1-(1- 2h \hat{\rho}_1)^{k+1}\big\}
\end{align*}
for all $k \in \mathbb N_0$ and  $ i \in \{1,\ldots,N \}$. Adding that  
$t_{k+1} = (k+1)h$, and for $h \in (0, 1/2\hat{\rho}_1)$
\begin{align*}
   \sup_{k\in \mathbb N_0} W^{(2)} (P_{t_{k+1}}^{i,N,n}\mu_0,\delta_0)^2 & \leq W^{(2)}(\mu_0, \delta_0)^2 + \frac{L_b^*+L_f^*}{ \hat{\rho}_1} \sup_{k\in \mathbb N_0} \big\{1-(1- 2h \hat{\rho}_1)^{k+1}\big\}
\end{align*}
and since $\sup_{k\in \mathbb N_0} \big\{1-(1- 2h \hat{\rho}_1)^{k}\} \leq 1$, thus this upper bound is uniform over $N,n,k,h$. This completes the proof. 
\newline
 
\noindent
(B) Notice that the case $k=0$ is satisfied trivially. Thus, we assume $k \in \mathbb N_0$. 
By adapting derivation similar to the one used in Equation \eqref{eq:dis:ito} along with the antisymmetric property of $f$, one can obtain, 
\begin{align*}
     \frac{1}{N} \sum_{i=1}^N \mathbb E  |\hat{X}_{t_{k+1}}^{i,N,n}  -\hat{Y}_{t_{k+1}}^{i,N,n}|^2 \leq & \, \frac{1}{N} \sum_{i=1}^N\mathbb E |\hat{X}_{t_k}^{i,N,n}-\hat{Y}_{t_k}^{i,N,n}|^2
    \\
     + \frac{2h}{N} \sum_{i=1}^N \mathbb E \big\{\bigl \langle &\, \hat{X}_{t_k}^{i,N,n}  -\hat{Y}_{t_k}^{i,N,n}, b^n(\hat{X}_{t_k}^{i,N,n}, \mu_{t_k}^{X,N,n})-b^n(\hat{Y}_{t_k}^{i,N,n}, \mu_{t_k}^{Y,N,n}) \bigl \rangle 
    \\
    & \qquad + |\sigma^n(\hat{X}_{t_k}^{i,N,n}, \mu_{t_k}^{X,N,n})-\sigma^n(\hat{Y}_{t_k}^{i,N,n}, \mu_{t_k}^{Y,N,n})|^2 \big\}
    \\
     + \frac{h}{N^2} \sum_{i,j=1}^N \mathbb E \big\{\bigl \langle &\,(\hat{X}_{t_k}^{i,N,n}-\hat{Y}_{t_k}^{i,N,n})-(\hat{X}_{t_k}^{j,N,n}-\hat{Y}_{t_k}^{j,N,n}), f^n (\hat{X}_{t_k}^{i,N,n},\hat{X}_{t_k}^{j,N,n}) - f^n (\hat{Y}_{t_k}^{i,N,n},\hat{Y}_{t_k}^{j,N,n})
    \\
    &  \qquad + 2|g^n (\hat{X}_{t_k}^{i,N,n},\hat{X}_{t_k}^{j,N,n}) - g^n (\hat{Y}_{t_k}^{i,N,n},\hat{Y}_{t_k}^{j,N,n}|^2\big\} 
    \\
     + \frac{2h^2}{N} \sum_{i=1}^N \mathbb E|b^n(\hat{X}_{t_k}^{i,N,n}, \mu_{t_k}^{X,N,n}) &\, -b^n(\hat{Y}_{t_k}^{i,N,n}, \mu_{t_k}^{Y,N,n})|^2 
 +  \frac{ 2h^2}{N^2} \sum_{i,j=1}^N \mathbb E |f^n (\hat{X}_{t_k}^{i,N,n},\hat{X}_{t_k}^{j,N,n}) - f^n (\hat{Y}_{t_k}^{i,N,n},\hat{Y}_{t_k}^{j,N,n}|^2
\end{align*}
for all $k \in \mathbb N_0$.
Using Lemma \ref{lem: Er sch b mono}, one can obtain
\begin{align*}
    \frac{1}{N} \sum_{i=1}^N \mathbb E &|\hat{X}_{t_{k+1}}^{i,N,n}  -\hat{Y}_{t_{k+1}}^{i,N,n}|^2  \leq \frac{1}{N} \sum_{i=1}^N \mathbb E |\hat{X}_{t_k}^{i,N,n}-\hat{Y}_{t_k}^{i,N,n}|^2 
     -\Big( \frac{L_{b\sigma}^{(1)}}{2} \wedge L_{b\sigma}^{(4)}\Big) \frac{2h}{N}  \sum_{i=1}^N  \mathbb E   |\hat{X}_{t_{k}}^{i,N,n} -\hat{Y}_{t_{k}}^{i,N,n}|^2 
    \\
    & +  2h (L_{b\sigma}^{(2)}+L_{b\sigma}^{(5)}) W^{(2)}(\mu_{t_k}^{X,N,n},\mu_{t_k}^{Y,N,n})^2 
    \\
    &-\Bigl ( \frac{L_{fg}^{(1)}}{2} \wedge L_{fg}^{(3)} \Bigl ) \frac{h}{N^2} \sum_{i,j=1}^N \mathbb E   |(\hat{X}_{t_k}^{i,N,n}-\hat{Y}_{t_k}^{i,N,n}) -(\hat{X}_{t_k}^{j,N,n}-\hat{Y}_{t_k}^{j,N,n})|^2 
    \\
    & + 4n (L_b^{(1)} \wedge L_b^{(3)}) \frac{2h^2 }{N} \sum_{i=1}^N  \mathbb E   |\hat{X}_{t_k}^{i,N,n}-\hat{Y}_{t_k}^{i,N,n}|^2 + 4h^2(L_{b} ^{(2)}+ L_{b} ^{(4)}) W^{(2)} (\mu_{t_k}^{X,N,n},\mu_{t_k}^{Y,N,n})^2 
    \\
    & + 4n(L_{f}^{(1)} \vee L_{f}^{(2)}) \frac{2h^2}{N^2} \sum_{i, j=1}^N \mathbb E  |(\hat{X}_{t_k}^{i,N,n}-\hat{Y}_{t_k}^{i,N,n})-(\hat{X}_{t_k}^{j,N,n}-\hat{Y}_{t_k}^{j,N,n})|^2 
\end{align*}
which, on further simplification and iterations yields
\begin{align*}
    \frac{1}{N}  & \sum_{i=1}^N \mathbb E|\hat{X}_{t_{k+1}}^{i,N,n}  -\hat{Y}_{t_{k+1}}^{i,N,n}|^2  \leq  \Big(1-2h\Big\{\Big( \frac{L_{b\sigma}^{(1)}}{2} \wedge L_{b\sigma}^{(4)} \Big)-\big(L_{b\sigma}^{(2)}+L_{b\sigma}^{(5)} \big) 
    \\
    & \qquad +2 \Bigl ( \frac{L_{fg}^{(1)}}{2} \wedge L_{fg}^{(3)} \Bigl ) - 4 (L_b^{(1)} \vee L_b^{(3)}) - 2(L_b^{(2)} \vee L_b^{(4)}) - 16(L_f^{(1)} \vee L_f^{(2)}) \Big\}\Big) \frac{1}{N} \sum_{i=1}^N \mathbb E|\hat{X}_{t_{k}}^{i,N,n}  -\hat{Y}_{t_{k}}^{i,N,n}|^2
    \\
    & \leq (1- 2 h \hat{\rho}_2)^{k+1} \frac{1}{N} \sum_{i=1}^N  \mathbb E|\hat{X}_{0}^{i, N,n}  -\hat{Y}_{0}^{i, N,n}|^2 
    \\
    & \leq  2(1- 2 h \hat{\rho}_2)^{k+1} \bigl(\mathbb E|X_{0}|^2 + \mathbb E|Y_{0}|^2 \bigl) 
\end{align*}
for all $k \in \mathbb N_0$. 
One can observe that
\begin{align*}
    W^{(2)}(P^{i,N,n}_{t_{k+1}}\mu_0,P^{i, N,n}_{t_{k+1}}\nu_0)^2 & \leq \mathbb E|\hat{X}_{t_{k+1}}^{i,N,n}  -\hat{Y}_{t_{k+1}}^{i,N,n}|^2 
    \leq  2(1- 2 h \hat{\rho}_2)^{k+1}
       \big(W^{(2)}(\mu_0,\delta_0)^2 + W^{(2)}(\delta_0,\nu_0)^2 \big) 
\end{align*}
for all $k \in \mathbb N_0$ and  $ i \in \{1,\ldots,N \}$. Using elementary identity $1-x \leq e^{-x}$ with $t_k = (k+1)h$, one obtains
\begin{align*}
    W^{(2)}(P^{i,N,n}_{t_{k+1}}\mu_0,P^{i, N,n}_{t_{k+1}}\nu_0)^2 & \leq 2 e^{-2\hat{\rho}_2t_{k+1}}
       \big(W^{(2)}(\mu_0,\delta_0)^2 + W^{(2)}(\delta_0,\nu_0)^2 \big)
\end{align*}
which completes the proof of part (B).
\newline  

\noindent
(C)
  Take $\mu_0 \in \mathcal{P}_{2}(\mathbb R^d)$,  $\hat{\rho}_1 >0$. Due to statement (A), one can observe that
\begin{align*}
   \sup_{m \in \mathbb N_0} W^{(2)} \big(\bar P_{t_{m}}^{N,n}\mu_0^{\otimes N},\mu_0^{\otimes N}\big)^{2} 
   & 
   \leq 2 N W^{(2)} (\mu_0,\delta_0)^{2} 
   +  2 N \sup_{m \in\mathbb N_0} 
   \sup_{i\in\{1,\ldots,N\} }
   W^{(2)} (P_{t_{m}}^{i, N,n}\mu_0,\delta_0)^{2} < \infty. 
\end{align*}
Since $b$ and $\sigma$ are independent of time, the application of statement (B) implies,
\begin{align*}
        \lim_{k \rightarrow \infty} 
     \sup_{m \in \mathbb N_0} 
     W^{(2)} \big( & \bar P_{t_k}^{N,n} \mu_0^{\otimes N},
     \bar P_{t_{k+m}}^{N,n}\mu_0^{\otimes N}) 
     =  
     \lim_{k \rightarrow \infty}  \sup_{m \in \mathbb N_0}  
     W^{(2)} \Big(\bar P_{t_k}^{N,n}\mu_0^{\otimes N},
                  \bar P_{t_{k}}^{N,n}\big( \bar P_{t_{m}}^{N,n}\mu_0^{\otimes N}\big)\Big)
     \\
     &
     \leq 
     \lim_{k \rightarrow \infty} 
     e^{-\hat{\rho}_2 t_k} N 
     \sup_{m \in \mathbb N_0} 
     \sup_{i\in\{1,\ldots,N\} } 
     \big(
     W^{(2)} (\mu_0,\delta_0) + W^{(2)} ( P_{t_{m}}^{i,N,n}\mu_0,\delta_0) 
     \big)
     =
     0, 
\end{align*}
where we use the uniform bound from (A) to ensure the limit. 
Since $(W^{(2)}, \mathcal{P}_2\big( (\mathbb R^d)^N \big)$ is complete metric space,  $\displaystyle \{\bar P_{t_k}^{N,n}\mu_0^{\otimes N}\}_{k \in \mathbb N_0}$ is convergent and there exists a limiting measure $\bar{\mu}^{N,n} \in \mathcal{P}_2\big( (\mathbb R^d)^N\big)$ to it, in other words, $\displaystyle \lim_{k \rightarrow \infty} W^{(2)}(\bar P_{t_k}^{N,n}\mu_0^{\otimes N},\bar{\mu}^{N,n})=0$. 
Using the continuity of the metric \cite[Corollary 6.11 and Remark 6.12 (p.97)]{villani2009OT} we have 
 \begin{align*}
    W^{(2)}(\bar P_{t_m}^{N,n} \bar {\mu}^{N,n}, \bar {\mu}^{N,n}) 
    & 
    \leq \liminf_{k \rightarrow \infty} 
    W^{(2)} \Big(\bar P_{t_m}^{N,n} \big(P_{t_k}^{N,n}\mu_0^{\otimes N}\big) 
    , \bar P_{t_k}^{N,n}\mu_0^{\otimes N}\Big)
    \\
    & =
    \liminf_{k \rightarrow \infty} 
    W^{(2)} \big( \bar P_{t_{k+m}}^{N,n}\mu_0^{\otimes N} 
    , \bar P_{t_k}^{N,n}\mu_0^{\otimes N}\big) 
    =
    0
 \end{align*}
 for any $m \in \mathbb N_0$. 
Thus, $\bar {\mu}^{N,n}$ is the invariant measure of the tamed Euler scheme \eqref{scheme}.

Lastly, for any $\hat{\nu} \in \mathcal{P}_{2} (\mathbb R^d)$, $h  \in (0,\min \{h^*,1/(2\hat{\rho}_1)\}) $ and $k\in \bN_0$ and using part (B) 
\begin{align*}
      W^{(2)} \big( \bar P_{t_k}^{N,n}\hat{\nu}^{\otimes N }
        ,\bar {\mu}^{N,n}\big) 
      & 
      \leq 
      W^{(2)}\big(  \bar P_{t_k}^{N,n} \hat{\nu}^{\otimes N}
                ,   \bar P_{t_k}^{N,n} \bar{\mu}^{N,n}  \big) 
      + 
      W^{(2)}\big(  \bar P_{t_k}^{N,n} \bar{\mu}^{N,n}
                ,   \bar {\mu}^{N,n}  \big) 
      \\ 
     & 
     \leq e^{-\hat{\rho}_2t_{k}} 
     \big\{
      W^{(2)}\big(\hat \nu^{\otimes N} , \delta_0^{\otimes N})+ 
      W^{(2)}\big(\bar {\mu}^{N,n},\delta_0^{\otimes N}\big)
      \big\}.
\end{align*}
This implies that for any $\hat{\nu}_0 \in \mathcal{P}_{2} (\mathbb R^d)$ we have 
\begin{align*}
    \lim_{k \to \infty} 
    W^{(2)}\big(  \bar P_{t_k}^{N,n}\hat{\nu}_0^{\otimes N} 
                , \bar{\mu}^{N,n}        \big) 
            =0,
\end{align*}
and thus the invariant measure is unique. 
 
\color{black}
\end{proof}

\appendix
\section{Auxiliary results}
\label{appendix}
In this section, we list the auxiliary results which are frequently used in the proofs of our main results. 
\subsection*{The Rosenthal inequalities}
Following \cite{Gobet2016MCmethodsandStochProcs}, the Rosenthal inequality gives a universal estimate of the \(p\)-th moment of a sum of independent centered random variables, as a function of the moments of each term in the sum (for a proof, see \cite[Theorem 2.9]{Petrov1995limittheorem}) -- in fact, since the variables are centered, the Rosenthal inequality relates higher moments of cumulative sum of random variables with its variance.
\begin{lemma}[Rosenthal inequality {\cite[Theorem A.2.4]{Gobet2016MCmethodsandStochProcs}} ]
\label{lemma:classicalRosenthalIneq}
Let \( X_1, \ldots, X_M\) be a sequence of independent real-valued random variables, having finite moments of order \(p \geq 2\) and with \(\mathbb{E}[X_m]=0\) (centered). 

Then, there exists a certain universal constant \(c_p>0\), such that
\begin{align*}
\mathbb{E}\Big[ \big|\sum_{m=1}^M X_m\big|^p\Big] 
\leq 
c_p\Big(\ \sum_{m=1}^M \mathbb{E}\big[\,|X_m|^p\big]
        +\big(\sum_{m=1}^M \mathbb{E}\big[\, |X_m|^2\big]  
        \Big)^{\frac p2} \ \Big) .
\end{align*}    
\end{lemma}

The Rosenthal inequality can be generalized to the conditionally independent case. Such inequality is coined \textit{conditional Rosenthal-type inequality} in \cite[Theorem 2.2]{yuan2015} (see also \cite[Lemma 3.2]{zhang2025}\footnote{See its ArXiv v2 version, \url{https://arxiv.org/pdf/2502.20786v2}.}).  
\begin{lemma}[Conditional Rosenthal-type inequality {\cite[Theorem 2.2]{yuan2015})}]
\label{lemma:conditionalRosenthalIneq}
Let \(\mathcal{G}\) be a sub-\(\sigma\)-algebra of \(\mathcal{F}\), and let \(A_1, A_2, \ldots, A_M\) be \(\mathcal{G}\)-independent random variables with \(\mathbb{E}\left[A_i \mid \mathcal{G}\right]=0\) a.s. for all \(i\). 

Then, for \(p>2\), there exists a constant \(C_p\) depending only on \(p\) such that
\begin{align*}
\mathbb{E}\Big[ \big|\sum_{i=1}^M A_i\big|^p \mid \mathcal{G}\Big] 
\leq
C_p 
\max \Big\{\
\sum_{i=1}^M \mathbb{E}\big[\, |A_i|^p \mid \mathcal{G}\big]
,
\Big( \sum_{i=1}^n \mathbb{E}\big[\, |A_i|^2 \mid \mathcal{G}\big]\Big )^{\frac{p}{2} }
\ \Big\}\qquad \textrm{a.s.}
\end{align*}
\end{lemma}
\subsection*{Other useful results}
  \begin{lemma} \label{lem:sym-gwt:erg}
Let $\mathbbm{f}: \mathbb R^d \times \mathbb R^d \mapsto \mathbb R^d$ and $\mathbbm{g}: \mathbb R^d \times \mathbb R^d \mapsto \mathbb R^{d\times l}$ be  Borel-measurable functions  satisfying the following conditions; 
\begin{itemize}
    \item[] (\textbf{C-1}): $\mathbbm{f}$ is anti-symmetric function, i.e.,  $\mathbbm{f}(x,y)= -\mathbbm{f}(y,x)$  for all $x, y, \in \mathbb R^d$.
    \item[] (\textbf{C-2}): There are constants   $\hat{L}_{\mathbbm{f} \mathbbm{g}}^{(1)},  \hat{L}_{\mathbbm{f}}^{(1)} \in \mathbb R $ and $\varrho\geq 2$  such that
    \begin{align*}
        \langle x -y,\mathbbm{f}  (x,y) \rangle + 2 (\varrho-1)|\mathbbm{g}(x,y)|^2 & \leq \hat{L}_{\mathbbm{f} \mathbbm{g}} + \hat{L}_{\mathbbm{f} \mathbbm{g}}^{(1)} |x-y|^2,
        \\
        (|x|^{\varrho-2} - |y|^{\varrho-2}) \langle x+y,\mathbbm{f}(x,y) \rangle & \leq \hat{L}_{\mathbbm{f}}^{(1)} (|x|^{\varrho}+|y|^{\varrho}),
    \end{align*}
    for all $x,y \in \mathbb R^d$ and $\mu \in \mathcal{P}_2(\mathbb{R}^d)$.  
\end{itemize}
 Then, for any $\mu \in \mathcal{P}_{\varrho}(\mathbb R^d)$ the following holds, 
      \begin{align*}
 \int_{\mathbb R^d}  \int_{\mathbb R^d} |x|^{\varrho -2}& \,\big \{ \langle x,  \mathbbm{f} (x,y) \rangle + (\varrho-1) |\mathbbm{g}(x,y)|^2 \big\} \mu (dx) \mu(dy) 
   \leq  \,  \Bigl( 2 \hat{L}_{\mathbbm{f} \mathbbm{g}}^{(1)+} + \frac{\varrho -2}{2\varrho} + \frac{\hat{L}_{\mathbbm{f}}^{(1)}}{2}  \Bigl) \int_{\mathbb R^d}   |x|^{\varrho} \mu (dx) + \frac{(\hat{L}_{\mathbbm{f} \mathbbm{g}})^{\varrho/2}}{\varrho}      .
\end{align*}
 \end{lemma}
\begin{proof}
Using anti-symmetric property of $\mathbbm{f}$ from (\textbf{C-1}), i.e., $\mathbbm{f}(x,y)=-\mathbbm{f}(y,x)$ for all $x,y\in \mathbb R^d$, 
\begin{align*}
      \int_{\mathbb R^d}  \int_{\mathbb R^d} |x|^{\varrho -2} & \, \langle x,  \mathbbm{f} (x,y) \rangle \mu(dx) \mu(dy) =  \frac{1}{2} \int_{\mathbb R^d}  \int_{\mathbb R^d}  \langle |x|^{\varrho -2} x - |y|^{\varrho -2}y,  \mathbbm{f} (x,y) \rangle \mu(dx) \mu(dy) 
       \\
    = & \, \frac{1}{2} \int_{\mathbb R^d}  \int_{\mathbb R^d}   |x|^{\varrho -2} \langle x-y,  \mathbbm{f} (x,y) \rangle \mu(dx) \mu(dy) + \frac{1}{2} \int_{\mathbb R^d}  \int_{\mathbb R^d} \{ |x|^{\varrho -2}-  |y|^{\varrho-2} \} \langle y,  \mathbbm{f} (x,y) \rangle \mu (dx) \mu(dy)
\end{align*}
and thus, 
\begin{align*}
       \int_{\mathbb R^d}  \int_{\mathbb R^d} |x|^{\varrho -2} & \, \big\{ \langle x,  \mathbbm{f} (x,y) \rangle + (\varrho-1) |\mathbbm{f}(x,y)|^2 \big\} \mu (dx) \mu(dy) 
    \\
    = & \, \frac{1}{2} \int_{\mathbb R^d}  \int_{\mathbb R^d}   |x|^{\varrho -2} \{\langle x-y,   \mathbbm{f}(x,y) \rangle +2(\varrho-1) |\mathbbm{f}(x,y)|^2\}\mu(dx) \mu(dy)  
    \\
    &\, + \frac{1}{4} \int_{\mathbb R^d}  \int_{\mathbb R^d} (x|^{\varrho -2}-|y|^{\varrho -2}) \langle x+y,\mathbbm{f}(x,y) \rangle \mu (dx) \mu(dy).
\end{align*}
The application (\textbf{C-2}) yields, 
\begin{align*}
 \int_{\mathbb R^d}  \int_{\mathbb R^d} |x|^{\varrho -2} & \, \big\{ \langle x,  \mathbbm{f} (x,y) \rangle + (\varrho-1) |\mathbbm{g}(x,y)|^2 \big\} \mu (dx) \mu(dy) 
        \\  
   \leq & \,  \frac{1} {2} \int_{\mathbb R^d}  \int_{\mathbb R^d}  |x|^{\varrho -2} \{  \hat{L}_{{\mathbbm{f} \mathbbm{g}} } + \hat{L}_{{\mathbbm{f} \mathbbm{g}} }^{(1)} |x-y|^2 \} \mu (dx) \mu(dy)
   + \frac{\hat{L}_{\mathbbm{f}}^{(1)}}{2}  \int_{\mathbb R^d} |x|^{\varrho}  \mu(dx)
   \\
    \leq & \, \Bigl( 2 \hat{L}_{\mathbbm{f} \mathbbm{g}}^{(1)+} + \frac{\varrho -2}{2\varrho} + \frac{\hat{L}_{\mathbbm{f}}^{(1)}}{2}  \Bigl) \int_{\mathbb R^d}   |x|^{\varrho} \mu (dx) + \frac{(\hat{L}_{\mathbbm{f} \mathbbm{g}})^{\varrho/2}}{\varrho}   
\end{align*}
This concludes the proof. 
\end{proof}
One obtains the following corollary as a consequence of the above lemma. 
\begin{corollary}
\label{lem:symm-growth}
Let Borel-measurable functions $\mathbbm{f}: \mathbb R^d \times \mathbb R^d \mapsto \mathbb R^d$ and $\mathbbm{g}: \mathbb R^d \times \mathbb R^d \mapsto \mathbb R^{d\times l}$  satisfy the conditions of Lemma \ref{lem:sym-gwt:erg}. 
Then,  
    \begin{align*}
          \frac{1}{N^2}\sum_{i,j=1}^N  |x^i|^{\varrho -2} \big\{\big\langle x^i,  \mathbbm{f} (x^i, x^j) \big\rangle + (\varrho-1)   \big|\mathbbm{g} (x^i,x^j) \big|^2 \big\}  \leq   \frac{2\hat{L}_{\mathbbm{f} \mathbbm{g}}^{(1)+} +  \frac{\hat{L}_{\mathbbm{f}}^{(1)}}{2} }{N}\sum_{i=1}^N |x^i|^{\varrho} 
    \end{align*}
          for all $x^1, \ldots, x^N \in \mathbb R^{d}$.  
\end{corollary}
 \begin{lemma}{\label{L_rate_x^i}}
 If Assumptions \ref{as:anti-sys} and \ref{as:f:g:PoC-2} are satisfied, then  
 \begin{align} 
       \frac{1}{N^2}\sum_{i,j=1}^N  |x^{i}  -  \bar{x}^{i}|^{p-2} \Big\{ \langle x^i  - \bar{x} ^{i},   \, f(x^{i},x^{j}) -  f(\bar{x}^i,\bar{x}^j) \rangle   +   2(p-1) |g(x^i,x^j)-g(\bar{x}^{i},\bar{x}^{j})|^2 \Big\} \notag
     \leq 
     \,  \frac{K}{N} \sum_{i=1}^N|x^i-\bar{x}^i|^p   \nonumber
    \end{align}
     for all $x^1, \ldots, x^N \in \mathbb R^{d}$ and $\bar x^1, \ldots, \bar x^N \in \mathbb R^{d}$ where  $K>0$ is independent of $N \in \mathbb N$. 
 \end{lemma}
 \begin{proof}
 First, one applies the anti-symmetry property of $f$, i.e., Assumption \ref{as:anti-sys} and gets the following equation, 
 \begin{align}
   \frac{1}{N^2}\sum_{i,j=1}^N & |x^{i}  -      \bar{x}^{i}|^{p-2} \langle x^i  - \bar{x}^{i},  f(x^{i},x^{j}) -  f(\bar{x}^i,\bar{x}^j) \rangle   \notag
     \\
      = & \, \frac{1}{2N^2}\sum_{i,j=1}^N   \big \langle |x^{i}  -  \bar{x}^{i}|^{p-2}(x^i  - \bar{x} ^{i})-|x^{j}  -  \bar{x}^{j}|^{p-2}(x^j  - \bar{x} ^{j}),  f(x^{i},x^{j}) -  f(\bar{x}^i,\bar{x}^j) \big \rangle       \notag
      \\
      = & \, \frac{1}{2N^2}\sum_{i,j=1}^N   \big |x^{i}  -  \bar{x}^{i}|^{p-2} \big \langle (x^i  - \bar{x} ^{i})-  (x^j  - \bar{x} ^{j}), f(x^{i},x^{j}) -  f(\bar{x}^i,\bar{x}^j) \big \rangle \notag
      \\
      & + \frac{1}{2N^2}\sum_{i,j=1}^N  \big \langle (|x^{i}  -  \bar{x}^{i}|^{p-2} -|x^{j}  -  \bar{x}^{j}|^{p-2})(x^j  - \bar{x} ^{j}),  f(x^{i},x^{j}) -  f(\bar{x}^i,\bar{x}^j) \big \rangle \notag
      \\
       = & \, \frac{1}{2N^2}\sum_{i,j=1}^N   \big |x^{i}  -  \bar{x}^{i}|^{p-2} \langle (x^i  - \bar{x} ^{i})-  (x^j  - \bar{x} ^{j}), f(x^{i},x^{j}) -  f(\bar{x}^i,\bar{x}^j) \big \rangle \notag
      \\
      & + \frac{1}{4N^2}\sum_{i,j=1}^N  \big \langle \{|x^{i}  -  \bar{x}^{i}|^{p-2} -|x^{j}  -  \bar{x}^{j}|^{p-2}\}\{(x^i  - \bar{x} ^{i})+(x^j  - \bar{x} ^{j})\},  f(x^{i},x^{j}) -  f(\bar{x}^i,\bar{x}^j) \big \rangle \notag
 \end{align}
 which, on using Assumption   \ref{as:f:g:PoC-2} yields, 
 \begin{align*}
     \frac{1}{N^2} & \sum_{i,j=1}^N   |x^i  - \bar{x}^i|^{p-2}  \big\{ \langle x^i - \bar{x}^i,  f(x^i,x^j) - f(\bar{x}^i,\bar{x}^j) \rangle  + 2(p-1) |  g(x^i,x^j) - g(\bar{x}^i,\bar{x}^j)  |^2 \big\} 
     \\
    & = \frac{1}{2N^2}\sum_{i,j=1}^N  |x^{i}  -  \bar{x}^{i}|^{p-2} \big \{\big \langle (x^i  - \bar{x} ^{i})-  (x^j  - \bar{x} ^{j}), f(x^{i},x^{j}) -  f(\bar{x}^i,\bar{x}^j) \big \rangle +  4 (p-1) |  g(x^i,x^j) - g(\bar{x}^i,\bar{x}^j)  |^2 \big\}
    \\
      & + \frac{1}{4N^2}\sum_{i,j=1}^N  \big \langle (|x^{i}  -  \bar{x}^{i}|^{p-2} -|x^{j}  -  \bar{x}^{j}|^{p-2})\{(x^i  - \bar{x} ^{i})+(x^j  - \bar{x} ^{j})\},  f(x^{i},x^{j}) -  f(\bar{x}^i,\bar{x}^j) \big \rangle \notag
      \\
      & \leq \frac{L}{2N^2}\sum_{i,j=1}^N |x^{i}  -  \bar{x}^{i}|^{p-2} |(x^i-x^j)-(\bar x^i-\bar x^j)|^2 + \frac{L}{4N^2}\sum_{i,j=1}^N \{|x^i-\bar x^i|^p+|x^j-\bar x^j|^p \}
    \end{align*}
 for all $(x^1, \ldots, x^N) \in \mathbb R^{d\times N}$ and $(\bar x^1, \ldots, \bar x^N) \in \mathbb R^{d\times N}$. The proof is completed on further simplification.
 \end{proof}


\end{document}